\documentclass[pdflatex,sn-mathphys,Numbered]{sn-jnl}

\usepackage{graphicx}%
\usepackage{multirow}%
\usepackage{amsmath,amssymb,amsfonts}%
\usepackage{amsthm}%
\usepackage{mathrsfs}%
\usepackage[title]{appendix}%
\usepackage[table,xcdraw]{xcolor}%
\usepackage{textcomp}%
\usepackage{manyfoot}%
\usepackage{booktabs}%
\usepackage{algorithm}%
\usepackage{algorithmicx}%
\usepackage{algpseudocode}%
\usepackage{todonotes}%
\usepackage{listings}%
\usepackage{anyfontsize}
\usepackage[absolute,overlay]{textpos}
\usepackage{pgfplotstable}
\usepackage{makecell}

\usepackage{siunitx}
\newcommand{\numQ}[1]{\num[
  fixed-exponent = 3,
  round-precision = 12,
  round-mode=figures,
  group-minimum-digits = 4,
  group-separator = {\,},
  drop-zero-decimal = true,
  table-format = 5.0]{#1}}
\usepackage{collcell}
\newcolumntype{N}{>{\collectcell\numQ}r<{\endcollectcell}}
\newcommand{\ca}[1]{\multicolumn{1}{r}{{\color{green!60!black}\underline{\numQ{#1}}}}}
\newcommand{\caK}[1]{\multicolumn{1}{r@{~K~~}}{{\color{green!60!black}\underline{\numQ{#1}}}}}

\usepackage[T1]{fontenc}
\usepackage[utf8]{inputenc}
\usepackage{lmodern}
\usepackage{csquotes}
\usepackage[british]{babel}
\usepackage{changes}
\usepackage{amssymb, amsmath, amsthm, amsopn}
\usepackage{thmtools}
\usepackage{mathtools}
\usepackage{commath}
\usepackage{nicefrac}
\usepackage{array}
\usepackage{multirow}

\usepackage{bm}
\usepackage{stmaryrd}
\usepackage{etoolbox}

\numberwithin{equation}{section}

\usepackage{booktabs}
\usepackage{hyperref}
\usepackage[capitalize]{cleveref}

\usepackage{graphicx}
\usepackage{tikz}
\usepackage{pgfplots}
\usepgfplotslibrary{colorbrewer,groupplots}

\usepackage{tikz-3dplot}
\usetikzlibrary{positioning, shapes, arrows, calc, arrows.meta, patterns, shadows, trees, intersections, spath3, calligraphy, spy}

\pgfplotsset{
  compat=1.15,
  cycle list/Set1-5,
  PPk/.style={mark=square*},
  TTk/.style={mark=*},
  TTkw/.style={mark=triangle*},
  TTkemb/.style={mark=diamond*},
}

\usepackage{marginnote}
\usepackage{todonotes}

\newcommand{\Th}{{\mathcal{T}_h}}
\newcommand{\Fh}{{\mathcal{F}_h}}
\newcommand{\Fhi}{{\mathcal{F}_h^{\text{i}}}}

\newcommand{\TT}{\mathbb{T}}
\newcommand{\X}{\mathbb{X}}

\providecommand{\HH}{\mathbb{H}}
\providecommand{\LL}{\mathbb{L}}
\newcommand{\RR}{\mathbb{R}}
\newcommand{\PP}{\mathbb{P}}
\newcommand{\BDM}{\mathbb{B}\mathbb{D}\mathbb{M}}

\newcommand{\cP}{\mathcal{P}}

\newcommand{\dof}{\texttt{dof}}
\newcommand{\dofs}{\texttt{dofs}}
\newcommand{\ndofs}{\texttt{ndofs}}

\newcommand{\dom}{{\Omega}}

\newcommand{\Kh}{K_h}

\newcommand{\dx}{\mathop{\mathrm{d} x}}

\DeclareMathOperator{\id}{id}

\DeclareMathOperator{\Div}{\mathrm{div}}
\DeclareMathOperator{\Grad}{\nabla}
\DeclareMathOperator{\Curl}{\mathrm{curl}}

\newcommand\restr[2]{{\left.\kern-\nulldelimiterspace #1 \right|_{#2} }}

\newcommand{\avg}[1]{\{\!\!\{ #1\}\!\!\}}
\newcommand{\jmp}[1]{[\![#1]\!]}
\newcommand{\jump}[1]{[\![#1]\!]}

\newcommand{\nrm}[1]{\Vert #1 \Vert}
\newcommand{\snrm}[1]{\vert #1 \vert}

\newcommand{\IP}{\mathbb{P}}

\newcommand{\IR}{\mathbb{R}}

\newcommand{\IT}{\mathbb{T}}

\newcommand{\calH}{\mathcal{H}}

\newcommand{\calL}{\mathcal{L}}

\newcommand{\calP}{\mathcal{P}}

\newcommand{\bl}{\mathbf{l}}

\newcommand{\bK}{\mathbf{K}}

\newcommand{\bT}{\mathbf{T}}

\newcommand{\bW}{\mathbf{W}}

\renewcommand{\part}{{\textrm{par}}}
\renewcommand{\hom}{{\textrm{hom}}}

\pgfplotsset{
    discard if not/.style 2 args={
        x filter/.append code={
            \edef\tempa{\thisrow{#1}}
            \edef\tempb{#2}
            \ifx\tempa\tempb
            \else
                
            \fi
        }
    }
}

\pgfplotscreateplotcyclelist{paulcolors2}{%
violet,every mark/.append style={solid,fill=violet},mark=square*,very thick,mark size=3pt\\%
violet,densely dashed,every mark/.append style={solid,fill=violet},mark=square*,very thick,mark size=3pt\\%
teal,every mark/.append style={solid,fill=teal},mark=*,very thick,mark size=3pt\\%
teal,densely dashed,every mark/.append style={solid,fill=teal},mark=*,very thick,mark size=3pt\\%
orange,every mark/.append style={solid,fill=orange},mark=diamond*,very thick,mark size=3pt\\%
orange,densely dashed,every mark/.append style={solid,fill=orange},mark=diamond*,very thick,mark size=3pt\\%
cyan,every mark/.append style={solid,fill=cyan},mark=star,very thick,mark size=3pt\\%
cyan,densely dashed,every mark/.append style={solid,fill=cyan},mark=star,very thick,mark size=3pt\\%
}

\pgfplotscreateplotcyclelist{paulcolors3}{%
violet,every mark/.append style={solid,fill=violet},mark=square*,very thick,mark size=3pt\\%
teal,every mark/.append style={solid,fill=teal},mark=*,very thick,mark size=2.8pt\\%
orange,every mark/.append style={solid,fill=orange},mark=diamond*,very thick,mark size=2pt\\%
violet,densely dashed,every mark/.append style={solid,fill=violet},mark=square*,very thick,mark size=3pt\\%
teal,densely dashed,every mark/.append style={solid,fill=teal},mark=*,very thick,mark size=2.8pt\\%
orange,densely dashed,every mark/.append style={solid,fill=orange},mark=diamond*,very thick,mark size=2pt\\%
violet,dash dot,every mark/.append style={solid,fill=violet},mark=square*,very thick,mark size=3pt\\%
teal,dash dot,every mark/.append style={solid,fill=teal},mark=*,very thick,mark size=2.8pt\\%
orange,dash dot,every mark/.append style={solid,fill=orange},mark=diamond*,very thick,mark size=2pt\\%
}

\pgfplotscreateplotcyclelist{paulcolors}{%
violet,every mark/.append style={solid,fill=violet},mark=square*,very thick,mark size=3pt\\%
teal,every mark/.append style={solid,fill=teal},mark=*,very thick,mark size=2.8pt\\%
orange,every mark/.append style={solid,fill=orange},mark=diamond*,very thick,mark size=2pt\\%
violet,densely dashed,every mark/.append style={solid,fill=violet},mark=square*,very thick,mark size=3pt\\%
teal,densely dashed,every mark/.append style={solid,fill=teal},mark=*,very thick,mark size=2.8pt\\%
orange,densely dashed,every mark/.append style={solid,fill=orange},mark=diamond*,very thick,mark size=2pt\\%
violet,dash dot,every mark/.append style={solid,fill=violet},mark=square*,very thick,mark size=3pt\\%
teal,dash dot,every mark/.append style={solid,fill=teal},mark=*,very thick,mark size=2.8pt\\%
orange,dash dot,every mark/.append style={solid,fill=orange},mark=diamond*,very thick,mark size=2pt\\%
}

\pgfplotscreateplotcyclelist{paulcolors4}{%
violet,every mark/.append style={solid,fill=violet},mark=square*,very thick,mark size=3pt\\%
teal,every mark/.append style={solid,fill=teal},mark=*,very thick,mark size=3pt\\%
orange,every mark/.append style={solid,fill=orange},mark=diamond*,very thick,mark size=3pt\\%
cyan,every mark/.append style={solid,fill=cyan},mark=star,very thick,mark size=3pt\\%
}

\pgfplotscreateplotcyclelist{christophcolors}{%
violet,every mark/.append style={solid,fill=violet},mark=square*,very thick,mark size=3pt\\%
teal,every mark/.append style={solid,fill=teal},mark=diamond,very thick,mark size=2.8pt\\%
orange,every mark/.append style={solid,fill=cyan},mark=star,very thick,mark size=3pt\\%
purple,every mark/.append style={solid,fill=purple},mark=*,very thick,mark size=2.8pt\\%
}

\renewcommand{\hat}[1]{\widehat{#1}}

\theoremstyle{remark}%
\newtheorem{theorem}{Theorem}
\newtheorem{lemma}[theorem]{Lemma}
\newtheorem{cor}[theorem]{Corollary}
\newtheorem{remark}{Remark}%

\theoremstyle{definition}%

\begin{document}
\title[Trefftz-DG for Stokes]{Trefftz Discontinuous Galerkin discretization for the Stokes problem}


\author[1]{\fnm{Philip L.} \sur{Lederer}}\email{p.l.lederer@utwente.nl}

\author[2]{\fnm{Christoph} \sur{Lehrenfeld}}\email{lehrenfeld@math.uni-goettingen.de}

\author[2]{\fnm{Paul} \sur{Stocker}}\email{p.stocker@math.uni-goettingen.de}

\affil[1]{\orgdiv{Department of Applied Mathematics}, \orgname{University of Twente}, \orgaddress{\street{Hallenweg 19}, \city{Enschede}, \postcode{7522NH}, \country{The Netherlands}}}

\affil[2]{\orgdiv{Institute for Numerical and Applied Mathematics}, \orgname{University of Göttingen}, \orgaddress{\street{Lotzestr. 16-18}, \city{Göttingen}, \postcode{37083}, \state{State}, \country{Germany}}}


\abstract{We introduce a new discretization based on a polynomial Trefftz-DG
    method for solving the Stokes equations. Discrete solutions of 
    this method fulfill the Stokes equations pointwise within
    each element and yield element-wise divergence-free solutions.
    Compared to standard DG methods, a strong reduction of the degrees
    of freedom is achieved, especially for higher polynomial
    degrees. In addition, in contrast to many other Trefftz-DG
    methods, our approach allows us to easily incorporate inhomogeneous
    right-hand sides (driving forces) by using the concept of the
    embedded Trefftz-DG method. On top of a detailed a priori error
    analysis, we further
    compare our approach to other (hybrid) discontinuous Galerkin
    Stokes discretizations and present numerical examples. 
    }




\maketitle
\section{Introduction}
Trefftz methods, originating from the work of E. Trefftz
\cite{trefftz1926}, use approximation spaces that lie in the kernel of
the target differential operator. 
Compared to standard approximation spaces, e.g. polynomial spaces, the corresponding Trefftz approximation spaces are constructed in such a way that they still have similar approximation properties as the original space, but, due to the kernel property, have a significantly reduced number of degrees of freedom.
Trefftz methods provide a natural way to reduce
degrees of freedom in discontinuous Galerkin (DG) methods by replacing
the element-wise basis functions of the DG method by Trefftz basis
functions. 

Trefftz-DG methods have been derived and analyzed for several partial
differential equations (PDEs), usually under the common restrictions of
a zero source term and piecewise constant coefficients, see for
example the Laplace equation treated in \cite{HMPS14},
as
for Maxwell's equation see e.g.\cite{EKSW15, Huttunen}, Schr\"odinger
equation \cite{2306.09571,21M1426079}, the acoustic wave equation
\cite{mope18,bgl2016}, elasto-acoustics~\cite{bcds20}, versions
related to the `ultra-weak variational formulation' see also
\cite{SpaceTimeTDG,mope18,bcds20,KSTW2014,EKSW15,tpwave}.
For Trefftz schemes for Helmholtz we refer to the survey
\cite{TrefftzSurvey} and the references therein.

So far, the Trefftz-DG methodology has been restricted to a small set
of particular PDEs due to the need to explicitly construct basis
functions for the corresponding PDE-dependent Trefftz-DG spaces and
the restrictions mentioned above.  
Recently, a simple way to circumvent the explicit construction of
Trefftz spaces for the Trefftz-DG method has been introduced with the
embedded Trefftz-DG method in \cite{2201.07041}. This approach allows
an easy construction of Trefftz-DG spaces for a broad class of PDEs,
e.g. problems with differential operators that have not been
considered in the Trefftz-DG literature before or PDEs with varying
coefficients. Further, the approach allows for a generic construction
of particular solutions to handle inhomogeneous source terms.

Trefftz methods for the Stokes equations that can be found in the
literature are spectral-type methods. Explicit Trefftz basis functions
have been constructed in \cite{POITOU2000561,Bouberbachene}. In
\cite{LifitsQTSM} a so-called `quasi-Trefftz spectral-method' is
presented, that solves an eigenvalue problem on an encompassing domain
to construct Trefftz-like functions and can also treat the Stokes
problem with a source term. In \cite{21710} the Stokes problem is
considered in two dimensions and singular solutions are introduced
into a collocation Trefftz methods to deal with corner singularities.

While there is, to the best of our knowledge, no Trefftz-DG method for
the Stokes equations in the literature, standard DG discretizations
for the Stokes equations have been well-established for decades. As one
way to reduce the computational costs of DG methods, hybridization,
see \cite{cockburn2009unified}, has become very popular leading to the
class of hybrid DG methods; with several hybrid DG methods developed
also for the Stokes equations. However, in the search for efficient
discretizations based on DG formulations, Trefftz-DG and hybrid DG
methods are incompatible. By \emph{incompatible} we don't mean that a
combination of both methods is impossible, but that a combination of
both does not give any significant improvement over one of the two
basic methods. 
That being said, instead of reducing the number of basis functions, 
non-polynomial Trefftz type functions have been used successfully 
in a hybrid setting to enrich the element-wise basis in \cite{farhat2001}.
Since we focus on Trefftz-DG methods in this work we
skip the large amount of literature on Stokes discretizations based on
hybrid DG formulations in this introduction. Note, however, that
several (hybrid) DG formulations are discussed in
Section~\ref{sec:numbercrunching}.

\paragraph*{Main contributions and structure of the article}
In this work, we extend the (embedded) Trefftz-DG methodology to the
Stokes problem. We introduce a DG method with local basis functions
that are polynomials and solve the Stokes problem pointwise. This
extends the approach introduced in
\cite{montlaur2010discontinuous,montlaur08} where a DG method using
element-wise divergence-free polynomials are presented for the Stokes
and Navier-Stokes equations. In view of the Trefftz methodology the
approach in \cite{montlaur2010discontinuous,montlaur08} can be seen as
a Trefftz-DG method with respect to the mass conservation equation,
while the Trefftz-DG discretization considered in this paper includes
also the momentum equation in the construction of the Trefftz-DG
space. 

The analysis of the discretization reveals that higher-order pressure
functions are locally determined by the velocity and only the
piecewise constant pressures appear as Lagrange multipliers in the
Stokes saddle point problem. We prove the stability of the saddle point
problem and derive a priori error bounds in the energy as well as the
$L^2$-norm of the velocity. In contrast to most other works on
Trefftz-DG methods we allow inhomogeneous right-hand side source terms
in the method and its analysis. While the method for treating
inhomogeneous right-hand sides have already been introduced in
\cite{2201.07041}, to the best of our knowledge, this is the first
time a rigorous error analysis is provided for this generic approach
in dealing with inhomogeneities.

The article proceeds as follows: We start with some preliminaries in
\cref{sec:notation} and then introduce a DG and the corresponding
Trefftz-DG method for the Stokes problem in \cref{sec:tdgstokes}. The
analysis is carried out in \cref{sec:analysis}. Numerical results are
presented in \cref{sec:num}. Final comments are given in \cref{sec13}.

\section{Preliminaries}\label{sec:notation} Consider an open bounded
Lipschitz domain $\Omega \subset \mathbb{R}^d$ with $d=2, 3$. The
Stokes equations determine a velocity $u$ and a pressure $p$ such that
\begin{subequations}\label{eq:basicpde}
\begin{alignat}{2} 
  -\nu  \Delta u + \nabla p & = f && \quad \textrm{in } \Omega, \label{eq::stokes-a} \\
  -\Div u &=g && \quad \textrm{in } \Omega, \label{eq::stokes-b}\\
  u &= 0 &&\quad \textrm{on } \partial \Omega, \label{eq::stokes-c}
\end{alignat}
\end{subequations}
where $f, g$ are external body forces and $\nu > 0$ is the dynamic
viscosity. 
For the ease of presentation we only consider homogeneous Dirichlet boundary conditions. 
The method (and its analysis) can be extended to more general boundary conditions using standard DG techniques, as demonstrated in the numerical examples. 
The weak formulation of the problem \eqref{eq:basicpde} is
then given by: Find $(u,p)\in [H^1_0(\Omega)]^d \times L^2_0(\Omega)$
such that
\begin{equation}\label{eq:weakbasicpde}
\begin{aligned}
\int_\Omega \nu\!~\nabla u\!:\!\nabla v \dx - \int_\Omega \Div v\!~p \dx &= \int_\Omega f\cdot v \dx && \forall v\in [H^1_0(\Omega)]^d,  \\
-\int_\Omega  \Div u\!~q \dx&= \int_\Omega g q \dx  && \forall q\in L^2_0(\Omega),
\end{aligned}
\end{equation}
where $L^2_0(\Omega)$ is the space of square-integrable functions with
a zero mean value. The (Trefftz-) DG formulation introduced in this
work is set on a sequence of shape regular simplicial triangulations
$\Th$ of a polygonal domain $\dom$. We denote by $\Fh$ the set of
facets in the mesh $\Th$ (edges/faces in 2d/3d, respectively) and define $\Fhi$ as the subset of interior
facets.
With respect to a triangulation $\Th$ we denote by $h$ the piecewise constant field of local mesh sizes with $h|_T = h_T = \operatorname{diam}(T)$ for $T \in \Th$. 
With abuse of notation, when $h$ is evaluated on facets $F \in \Fh$ we set $h|_F = h_F = \operatorname{diam}(F)$,
and also denote $h = \max_{T\in\Th} h_T$ as the
maximum mesh size when $h$ appears as a scalar quantity.

We denote by $\cP^k(S)$ the space of polynomials up to degree $k$ on
an domain $S$ and set $\cP^k = \cP^k(\RR^d)$. 
Further, for $k < 0$ we set $\cP^k = \{0\}$.
With $\IP^k(\Th)$ and
$\IP^k(\Fh)$ we denote the broken, i.e. element- or facet-wise
polynomial space, such that for instance for $v\in\IP^k(\Th)$ it holds
$\restr{v}{T}\in\cP^k(T),$ for all $T\in\Th$. We use an according
notation for other broken function spaces as well, e.g. $H^2(\Th)$
denotes the space of functions with a local element-wise $H^2$
regularity. Due to its frequent appearance, we abbreviate the broken
spaces on the mesh $\Th$ by setting $\IP^k = \IP^k(\Th)$ (and similarly for
other broken function spaces) when the mesh is clear from the context.

The $L^2(S)$ inner product over a domain $S$ is denoted by
$(\cdot,\cdot)_S$. Specifically, we introduce the following notation for certain inner products:
\begin{align*}
(\cdot,\cdot)_\Th & \quad\text{sum of element-wise $L^2$-inner products over all mesh elements,}  \\
(\cdot,\cdot)_{\partial \Th} & \quad\text{sum of facet-wise $L^2$-inner products over the boundary of mesh elements,}   \\
(\cdot,\cdot)_\Fh &\quad\text{sum of facet-wise $L^2$-inner products over the facets of the mesh.}
\end{align*}
The unit outer normal is denoted by $n$. Correspondingly we use
$\nrm{\cdot}_S$ for the $L^2$-norm on a domain or set of
elements/facets $S$ and define $\nrm{\cdot}_0 \coloneqq
\nrm{\cdot}_\Th$.  

By $\Pi^k_S: L^2(S) \to \PP^k(S)$ we denote the $L^2$ projection into
the scalar-valued polynomial space of order $k$ on $S$. With abuse of
notation we use the same symbol for vectorial $L^2$-projections and
denote by $\Pi^k$ the element-wise $L^2$-projection on $\PP^k(\Th)$ or
$[\PP^k(\Th)]^d$, respectively.

Let $T$ and $T'$ be two neighboring elements sharing a common facet $F
\in \Fhi$ where $T$ and $T'$ are uniquely ordered. On $F$ the
functions $u_{T}$ and $u_{T'}$ denote the two limits of a discrete
function from the different sides of the element interface. The
vectors $n_T$ and $n_{T'}$ are the unit outer normals to $T$ and $T'$.
For $F \in \Fhi$ we then define the jump and the mean value by
\begin{equation*}
    \jump{v} := v_{T} - v_{T'}, \qquad \text{and} \qquad
    \avg{v} := \frac12 v_{T} + \frac12 v_{T'},
\end{equation*}
while on facets on the domain boundary we set $\jump{v} = \avg{v} =
v$.

Finally, we use the notation $A \simeq B$ when there are constants
$c,C > 0$ independent of $h$ and $\nu$ such that $A \le C B$ and $ B
\le c A$. Similarly, we also use the symbols $\lesssim$ and $\gtrsim$.

\section{Trefftz-DG for Stokes}\label{sec:tdgstokes} In this section
we derive the Trefftz-DG formulation for the Stokes problem. For that,
we first introduce a standard DG method and then include the Trefftz
spaces in the formulation. In \cref{sec:impl} we show a possible way
to implement the Trefftz spaces and construct local particular
solutions to treat the inhomogeneous problem using the embedded
Trefftz-DG method, see \cite{2201.07041}.

\subsection{The underlying DG Stokes discretization} \label{sec:underlyingDG}
As a basis for the following, we consider the established symmetric interior penalty DG discretization of the Stokes problem \eqref{eq:basicpde}, cf. e.g. \cite[Section 6.1.5]{di2011mathematical}: \\
Find $(u_h, p_h) \in [\PP^k]^d \times \PP^{k-1}/\,\RR$, s.t.
\begin{subequations}\label{eq:dgstokes}
\begin{align}
a_h(u_h,v_h) &+ b_h(v_h,p_h) &\hspace*{-1cm}&= (f,v_h)_{\Th} && \forall ~ v_h \in [\PP^k]^d,\\
  b_h(u_h,q_h) &  &\hspace*{-1cm}&= (g,q_h)_{\Th} && \forall ~ q_h \in \PP^{k-1} /\, \RR,
  \end{align}
\end{subequations}
with the bilinear forms
\begin{align*}
  a_h(u_h,v_h) &\coloneqq (\nu\! \Grad\! u_h,\Grad\! v_h)_\Th\! - (\avg{ \nu \partial_n u_h },\jmp{v_h})_\Fh\! - (\avg{\nu \partial_n v_h},\jmp{u_h})_\Fh  \\ & \qquad \!+ \frac{\alpha \nu }{h} (\jmp{u_h},\jmp{v_h})_\Fh, \\ 
b_h(v_h,p_h) &\coloneqq -(\Div v_h, p_h)_{\Th} + (\jmp{v_h \cdot n}, \avg{p_h})_{\Fh},
\end{align*}
where the interior penalty parameter $\alpha = \mathcal{O}(k^2)$ is
chosen sufficiently large and we used the notation $\partial_n w \coloneqq \nabla w \cdot
n$.

In this work, we want to introduce a Trefftz-DG method by reducing the
finite element space to a subspace, the Trefftz-DG space. For the
underlying DG space, we introduce the notation $\X^k_h(\Th) =
[\PP^k(\Th)]^d \times \PP^{k-1}(\Th) /\, \RR$ (in short:
$\X^k_h \coloneqq [\PP^k]^d \times \PP^{k-1}/\,\RR$). We define $\X^k_h(T) \coloneqq [\PP^k(T)]^d \times
\PP^{k-1}(T)$ as a local version\footnote{Note that we do not factor out the constants for the local version.}. 
For simplicity, we omit the superscript $k$ and write
$\X^k_h = \X_h$ and $\X^k_h(T) = \X_h(T)$. Further, we introduce the
bilinear form $K_h$ on the product space $\X_h$ by
\begin{equation} \label{eq:K}
  \Kh((u_h,p_h),(v_h,q_h)) := a_h(u_h,v_h) + b_h(u_h,q_h) + b_h(v_h,p_h).
\end{equation}
Then \eqref{eq:dgstokes} also reads as: Find $(u_h,p_h) \in \X_h$ such that
\begin{align} \label{eq:DGK}
  \Kh((u_h,p_h),(v_h,q_h)) = (f,v_h)_{\Th} + (g,q_h)_{\Th} \quad \forall (v_h,q_h) \in \X_h.
\end{align}

\subsection{The Trefftz-DG Stokes discretization}
The main idea of the Trefftz-DG method is to select a proper (affinely
shifted) lower-dimensional subspace of $\X_h$ that allows to reduce
the computational costs of the underlying DG discretization without
harming the approximation quality too much. To this end, we pick the
affine subspace of $\X_h$ that fulfills the Stokes equations
\eqref{eq:basicpde} pointwise \textbf{inside each element} up to data
approximation:
\begin{equation}\label{eq:Trefftzconditions}
  \TT^k_{f,g}(\Th) \!:=\! \{ (u_h,p_h) \!\in\! \X^k_h(\Th) \!\mid\!   -\Delta u_{h} + \nabla p_{h}\!=\!\Pi^{k-2} f, \!-\Div u_{h} \!= \Pi^{k-1}\! g \text{ on } \Th\}. 
\end{equation}
Again, when the relation to the mesh $\Th$ is clear from the context, we will abbreviate $\TT^k_{f,g} = \TT^k_{f,g}(\Th)$. The definition of a local Trefftz space $\TT^k_{f,g}(T)$ is correspondingly based on $\X^k_h(T)$.

Note that we allow for the case $k=1$ for which the first constraint $-\Delta u_{h} + \nabla p_{h} = \Pi^{k-2} f$ is automatically fulfilled as $u_h$ is piecewise linear and $p_h$ is piecewise constant while the r.h.s. projection maps to zero. Only the constraint $- \Div u_h = \Pi^0 g$ remains non-trivial in this case.

To work with linear spaces we decompose $\TT^k_{f,g}$ into a (non-unique) particular solution to the \emph{Trefftz constraints}
$(u_{h}^\part,p_{h}^\part) \in \TT^k_{f,g}$ and a linear space
$\TT^k=\TT^k(\Th)=\TT^k_{0,0}(\Th)$, typically denoted as \emph{the
Trefftz-DG space}. Again we omit the superscript $k$ and simply write
$\TT^k_{f,g} = \TT_{f,g}$ and $\TT^k=\TT=\TT(\Th)=\TT_{0,0}(\Th)$.
We prove below in \cref{sec:dofs} that a particular solution always exists, cf. \cref{lem:surj}, and can be
constructed in a straight-forward manner, cf. \cref{sec:impl}.

This yields the \emph{Trefftz problem}: Find $(u_h,p_h) \in \TT_{f,g}(\Th)
=  \TT(\Th) + (u_{h}^\part,p_{h}^\part)$ such that
\begin{align} \label{eq:discreteTrefftz}
  \Kh((u_h,p_h),(v_h,q_h)) = (f,v_h)_{\Th} + (g,q_h)_{\Th}\quad \forall (v_h,q_h) \in \TT(\Th).
\end{align}
Equivalently, to highlight the homogenization, we can also write:
Find $(u_{h}^\hom,p_{h}^\hom)~\in~\TT(\Th)$ so that for all $(v_h,q_h) \in \TT(\Th)$ there holds 
\begin{align} \label{eq:discreteTrefftz0}
  \Kh((u_{h}^\hom,p_{h}^\hom),(v_{h},q_{h})) = (f,v_h)_{\Th} + (g,q_h)_{\Th} - \Kh((u_{h}^\part,p_{h}^\part),(v_{h},q_{h})).
\end{align}

\subsection{Trefftz constraints and the dimension of
\texorpdfstring{$\TT$}{T}} \label{sec:dofs} In this subsection we want
to discuss some properties of the Trefftz spaces $\TT_{f,g}$ and
$\TT$. First, as we will typically have $g=0$, we note that the
velocities in $\TT_{f,0}$ (and $\TT$) will be divergence-free on each
element with the additional constraint of $-\Delta u_h+\Grad p_h=
\Pi^{k-2} f$, making $\TT_{f,0}$ (and $\TT$) a subspace of the
solenoidal vector fields considered in~\cite{BJK90,montlaur2010discontinuous}.

Next, we prove that a particular solution to the Trefftz constraints
always exist. Furthermore, we state the dimension of the local Trefftz
space.
\begin{lemma} \label{lem:surj} The pointwise Stokes operator $\calL :
 [\cP^{k}(T)]^d \times \cP^{k-1}(T) \to [\cP^{k-2}(T)]^d \times
 \cP^{k-1}(T)$, $(v,q) \mapsto (-\Delta v + \nabla p, - \Div v)$ is
 surjective and the local Trefftz space on an element $T\in\Th$ has
 the dimension\footnote{
  The following expression is also valid for $k=1$ 
  if we set
   $\left(\begin{array}{c} \ell \\ m \end{array} \right) = 0$ for $\ell < m$. 
 }
\begin{align} \label{eq:dimTT}
\dim( \TT(T)) & = 
\dim(\X_h(T)) \!-\! \dim([\cP^{k-2}]^d) \!-\! \dim(\cP^{k-1})
\\ & =
d \Big(  \Big( \begin{array}{c} k + d \\ d \end{array} \Big) - \Big( \begin{array}{c} k- 2 + d \\ d \end{array} \Big)  \Big). \nonumber
\end{align}
\end{lemma}
\begin{proof}
  We give the proof for $d=3$. The proof for $d=2$ follows similar lines. 
  To prove surjectivity, first note that $\Div [\cP^{k}]^d = \cP^{k-1}$ and hence it remains only to show that $\{-\Delta u_h + \nabla p_h \mid u_h\in[\cP^k]^d, p_h\in\cP^{k-1}, \Div u_h=0\}=[\cP^{k-2}]^d$. The task is hence to find $u_h \in [\cP^k]^d$ with $\Div u_h = 0$ and $p_h \in \cP^{k-1}$ to every $v_h\in [\cP^{k-2}]^d$ so that $-\Delta u_h + \nabla p_h = v_h$. We prove this in two steps.
\begin{enumerate}
    \item
    First we note that
    $[\cP^{k-2}]^d=\Delta [\cP^{k}]^d=\Curl\Curl[\cP^{k}]^d + \Grad\Div[\cP^{k}]^d,$
    where for the first equality we have used the surjectivity of the Laplace operator, see \cref{cor:lapsurj} in the appendix.
    Now using that  $\Curl[\cP^{k}]^d \subseteq [\cP^{k-1}]^d$ and $\Div[\cP^{k}]^d = \cP^{k-1}$ we get 
    \begin{equation*}
      [\cP^{k-2}]^d \subseteq \Curl[\cP^{k-1}]^d + \Grad\cP^{k-1}.
    \end{equation*}
    The other direction, $\Curl[\cP^{k-1}]^d + \Grad\cP^{k-1}
    \subseteq [\cP^{k-2}]^d$ is obvious and hence equality holds. We
    used the notation $V+W=\{v+w\mid v\in V,\ w\in W\}$ for the sum of
    two spaces. Note that this is not an orthogonal decomposition, and
    that the spaces considered here have a non-trivial intersection.
    In total this gives that for every $v_h\in [\cP^{k-2}]^d$ there is
    $v_1 \in [\cP^{k-1}]^d$ and $v_2 \in \cP^{k-1}$ so that $v_h =
    \Curl v_1 + \Grad v_2$.
    \item To match $\Grad v_2$ we can set $p_h = v_2$ and it remains to find $u_h \in [\cP^k]^d$ with $\Div u_h = 0$ so that $-\Delta u_h = \Curl v_1$. 
    To this end note that 
    \begin{equation*}
      \Curl [\cP^{k-1}]^d= \Curl \Delta [\cP^{k+1}]^d=  \Delta \Curl [\cP^{k+1}]^d = \{\Delta u_h\mid u_h\in[\cP^{k}]^d, \Div u_h=0\},
    \end{equation*}
    where we used that to every $u\in[\cP^{k}]^d$ with $\Div u=0$ we can find a $w\in[\cP^{k+1}]^d$ such that $u=\Curl w$ (and vice versa).

\end{enumerate}
\end{proof}

\begin{table}[!ht]
  \small
      \begin{tabular}{@{}r@{~}|@{~}r@{~}|@{~~}r@{~~}r@{~~}r@{~~}r@{~~}r@{~~}@{}|@{~}r@{~}|@{~~}r@{~~}r@{~~}r@{~~}r@{~~}r@{~~}r@{~~}r@{~~}r@{}}
      \toprule
        \multicolumn{7}{c}{$d=2$} & \multicolumn{6}{c}{$d=3$} \\
      \midrule
$k $   && $1$
    & $2$
    & $3$
    & $4$
    & $5$ &
    & $1$
    & $2$
    & $3$
    & $4$
    & $5$ \\
    $\dim \X_h(T)$ & $\frac32 k^2\!+\!\frac72 k\!+\!2$ & 7 & 15 & 26 & 40 & 57 
    & $\frac23 k^3\!+\!\frac72 k^2\!+\!\frac{35}{6} k\!+\!3$ & 13 & 34 & 70 & 125 & 203 \\
      $\dim \TT_{\hphantom{h}}(T)$ & $4 k \!+\! 2$& 6 & 10 & 14 & 18 & 22 
    & $3 k^2 + 6 k + 3 $&  12 & 27 & 48 & 75 & 108 \\
    \bottomrule
  \end{tabular}
  \caption{Dimensions of the local finite element spaces $\X_h(T)$ and $\TT(T)$.}\label{tab:localdims}
\end{table}

\begin{remark}
  We note that in the definition of the space $\X_h^k(\Th)$ -- which the Trefftz space $\TT^k(\Th)$ is a subspace of -- we factored out the (globally) constant pressure ($\mathbb{R}$). This however does not affect the local spaces $\X_h(T)$ and $\TT(T)$ which are defined without consideration of the global constant pressure. 
\end{remark}

~
\begin{remark}
  The formula derived in \eqref{eq:dimTT} shows the number of degrees
  of freedom per element (\ndofs) for the Trefftz-DG space. We observe
  that the \ndofs ~are exactly reduced by the number of scalar
  constraints that are imposed. The Stokes equations can be seen as
  the momentum equation formulated on the subspace of divergence-free
  functions where the pressure takes the role of the corresponding
  Lagrange multiplier. The costs of these additional unknowns are
  directly removed by the corresponding divergence-free constraint so
  that the \ndofs ~remaining coincide with those of harmonic vector
  polynomials up to degree $k$. Note however that the remaining system
  is still a saddle-point problem, cf. also \cref{rem:loworder}. Let
  us now take a look at the dimension reduction. While the dimension
  of $\X_h$ grows cubic and quadratic with respect to the polynomial
  degree $k$ for $d=3$ and $d=2$, respectively, the dimension of $\TT$
  grows only quadratic and linear for $d=3$ and $d=2$. In
  \cref{tab:localdims} we present concrete numbers for $k \in
  \{1,..,6\}$ to give an idea of the reduction. This dimension
  reduction gives the same complexity in $k$ that is also obtained in
  hybrid DG methods after static condensation.
\end{remark}

\begin{figure}[!ht]
    \centering
    \hspace*{-0.1\textwidth}
    \includegraphics[width=0.21\textwidth]{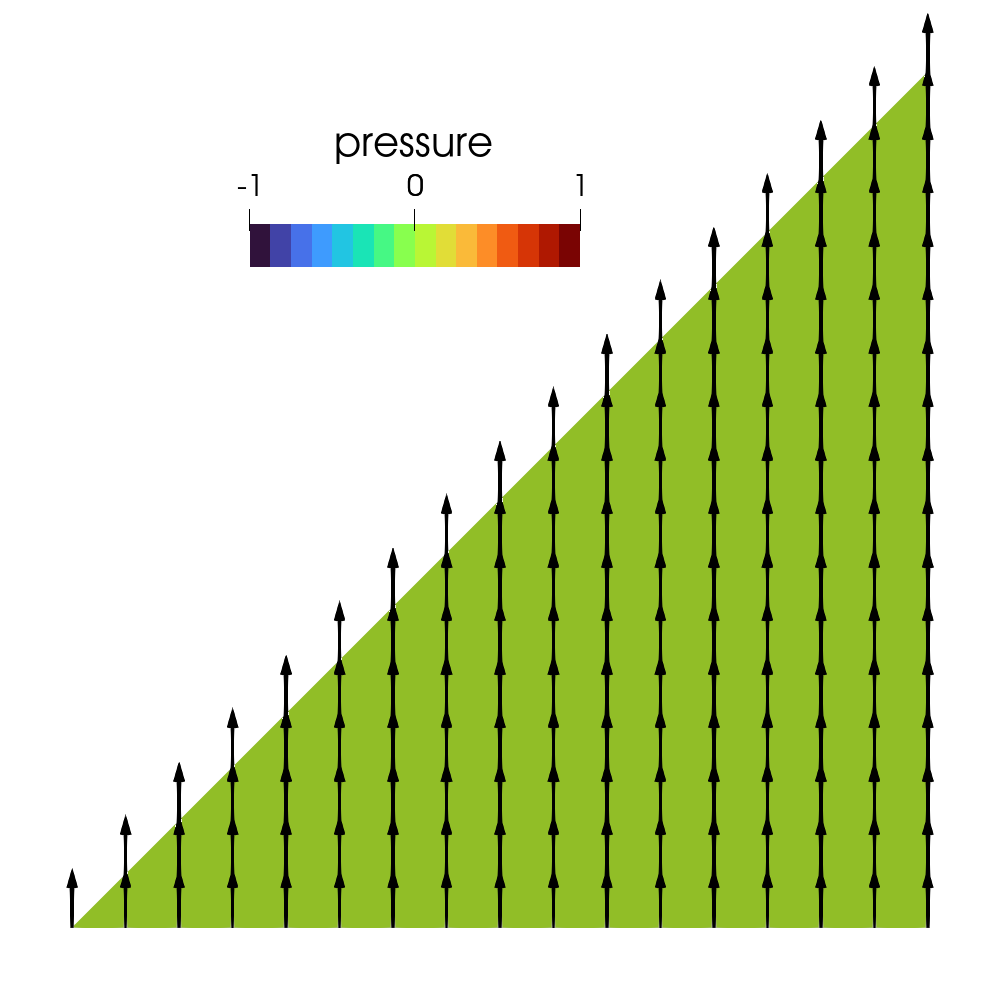} \hspace*{-0.03\textwidth}
    \includegraphics[width=0.21\textwidth]{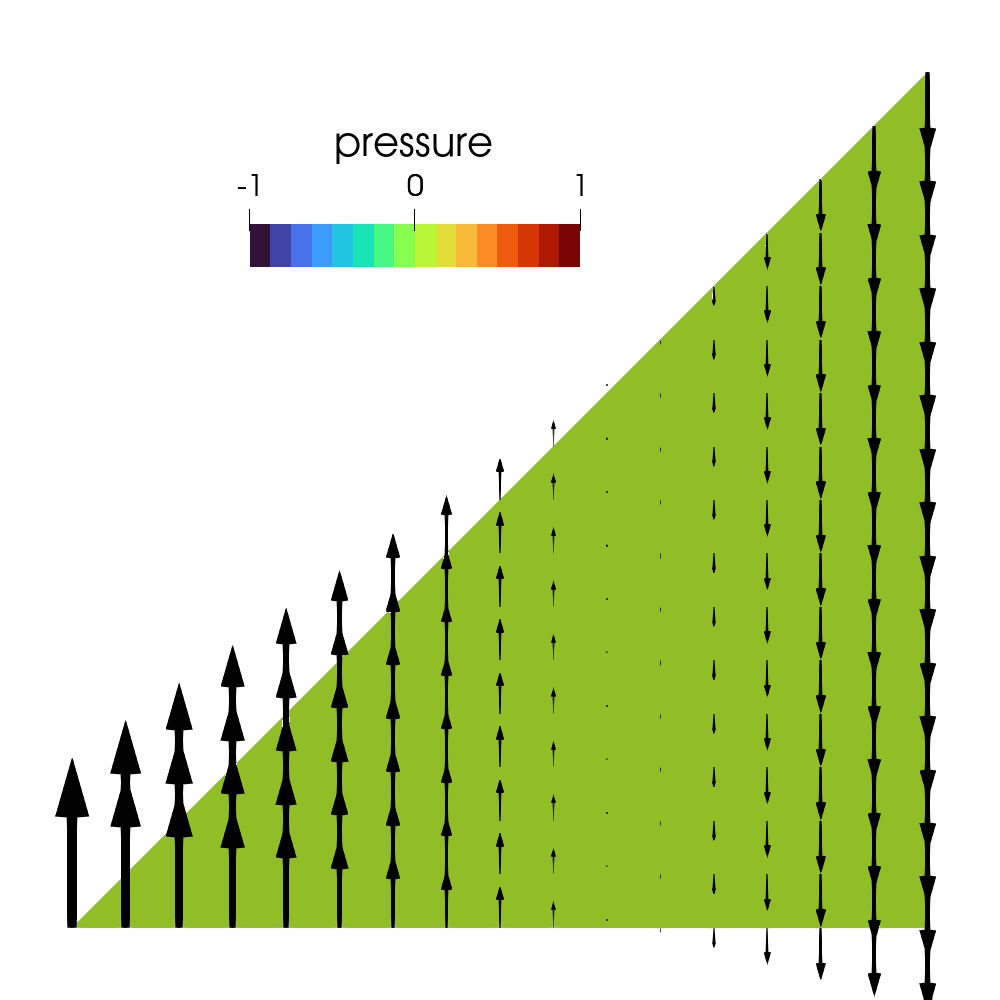} \hspace*{-0.03\textwidth}
    \includegraphics[width=0.21\textwidth]{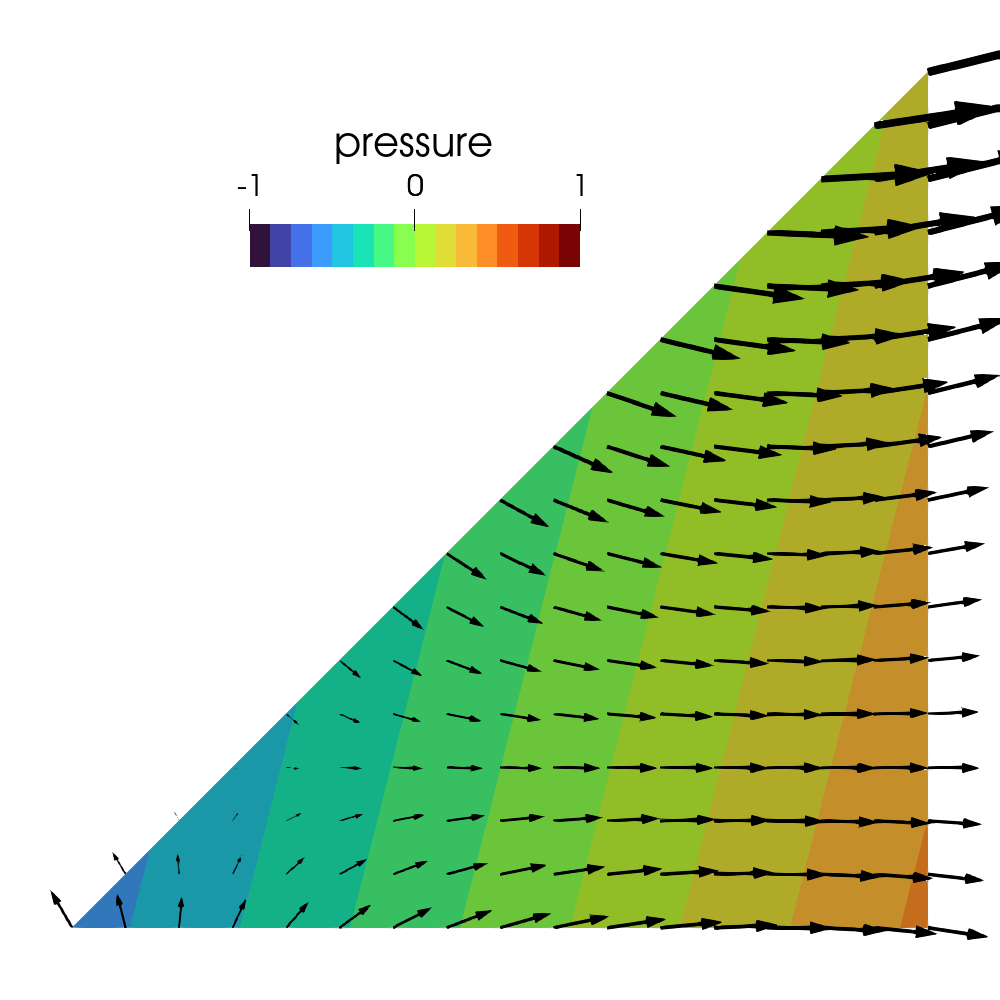} \hspace*{-0.03\textwidth}
    \includegraphics[width=0.21\textwidth]{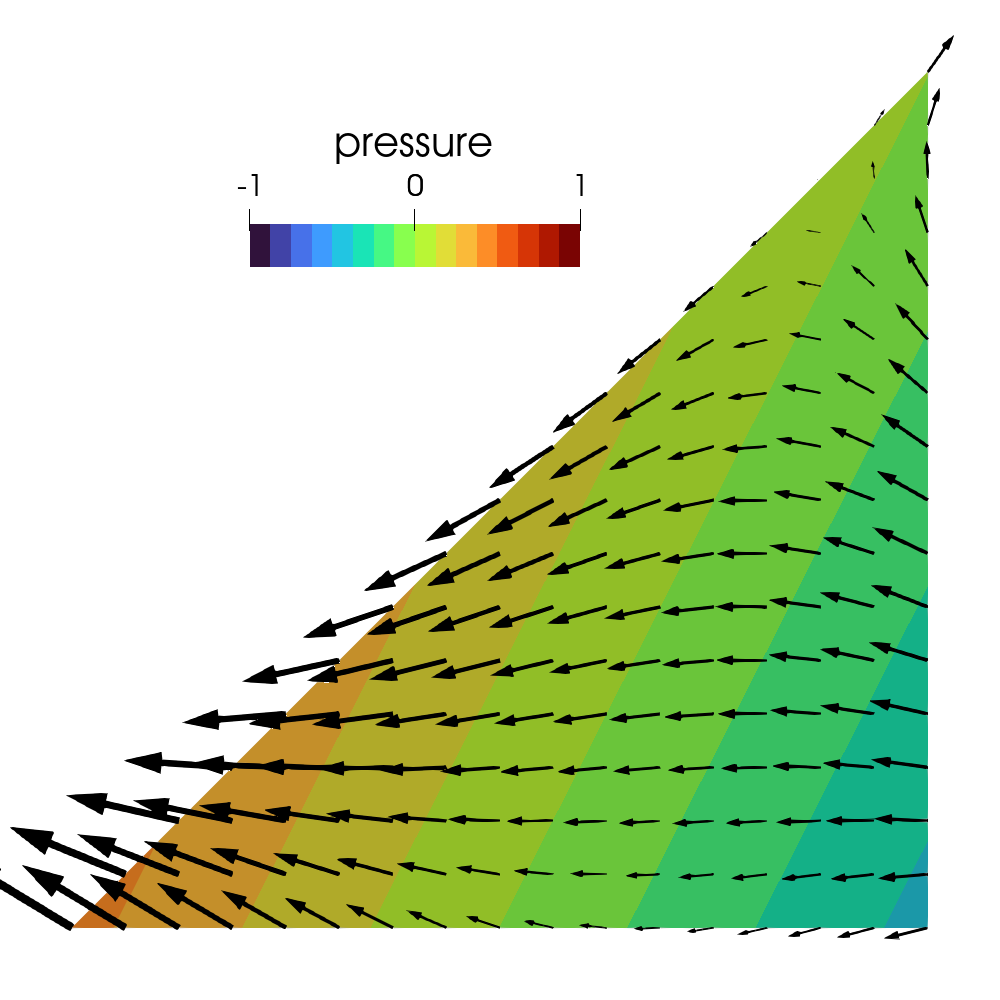} \hspace*{-0.03\textwidth}
    \includegraphics[width=0.21\textwidth]{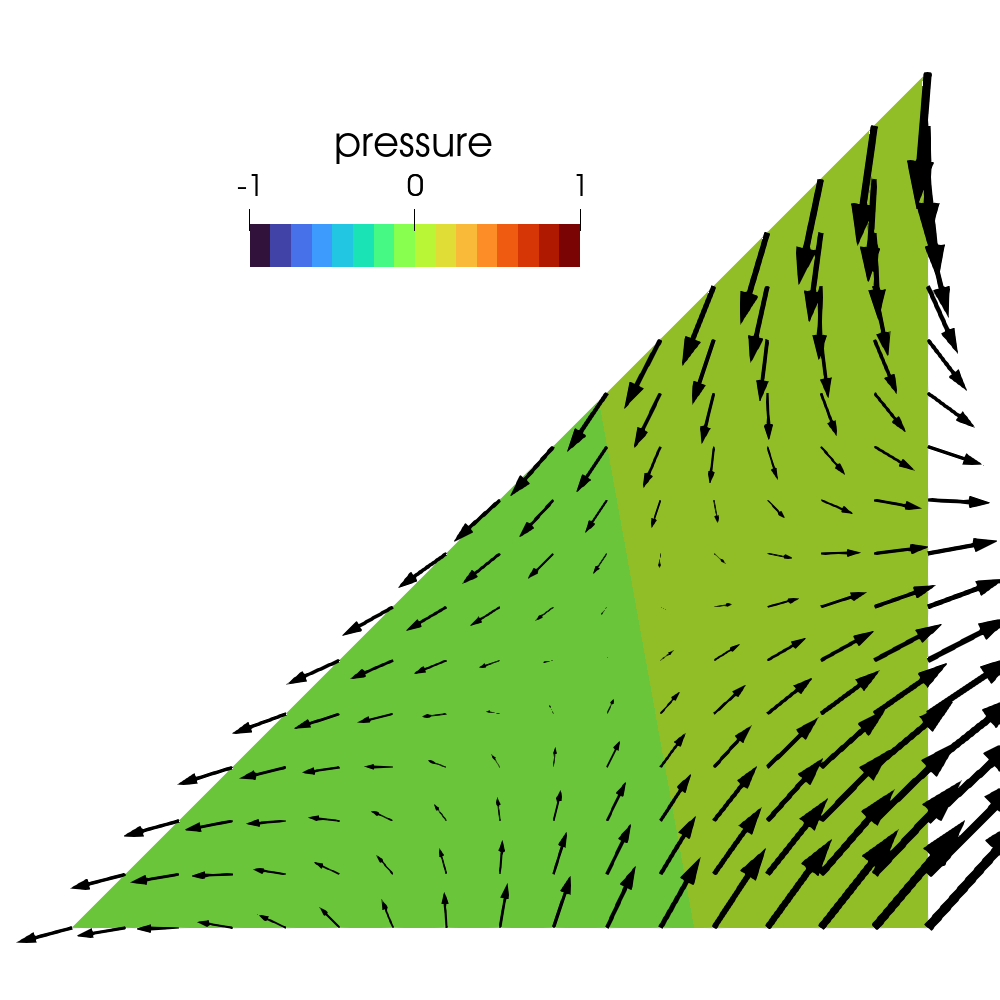} 
    \hspace*{-0.1\textwidth}
    \\
    \hspace*{-0.1\textwidth}
    \includegraphics[width=0.21\textwidth]{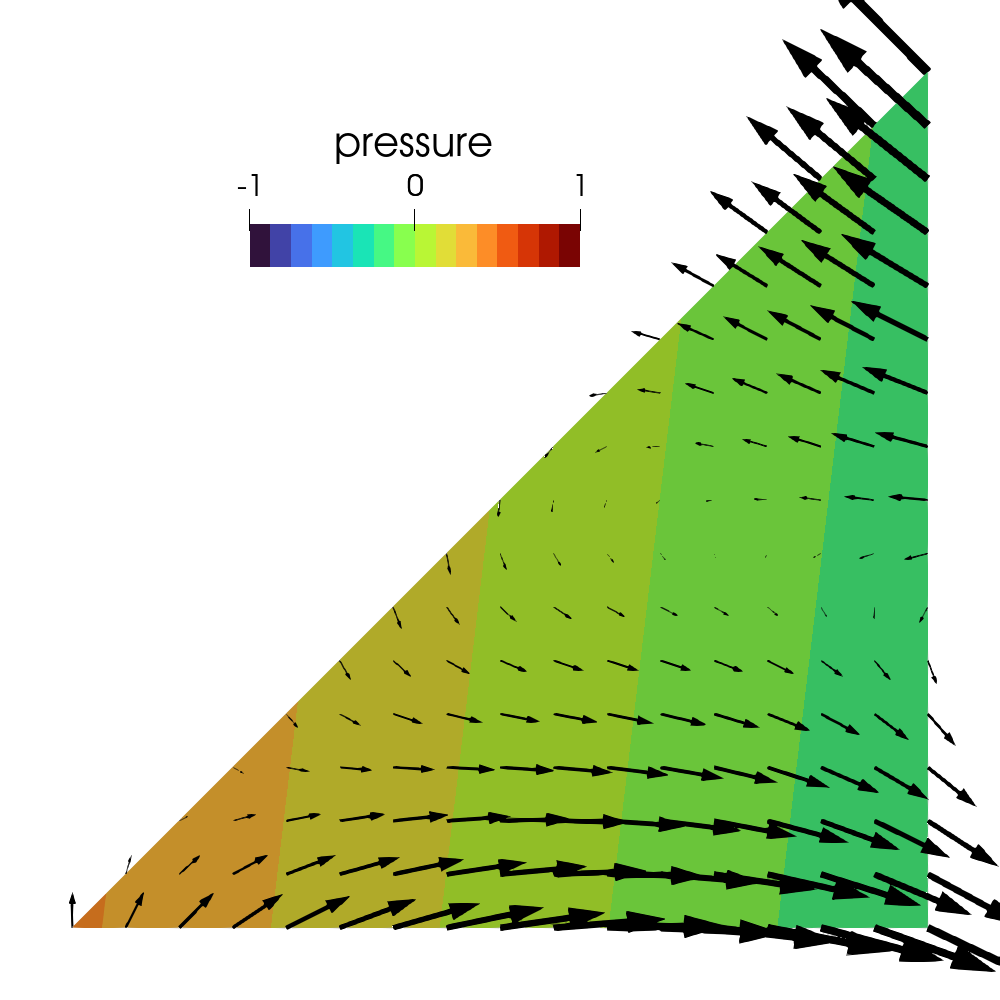} \hspace*{-0.03\textwidth}
    \includegraphics[width=0.21\textwidth]{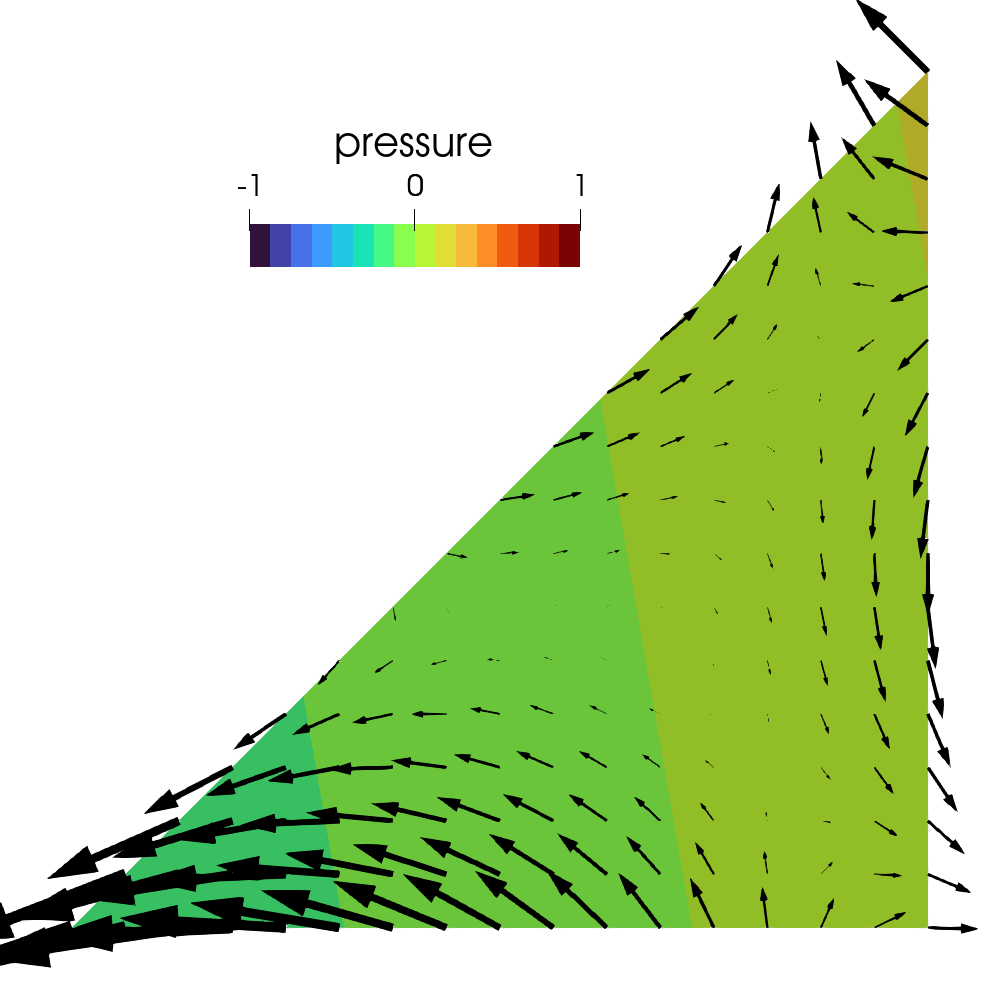} \hspace*{-0.03\textwidth}
    \includegraphics[width=0.21\textwidth]{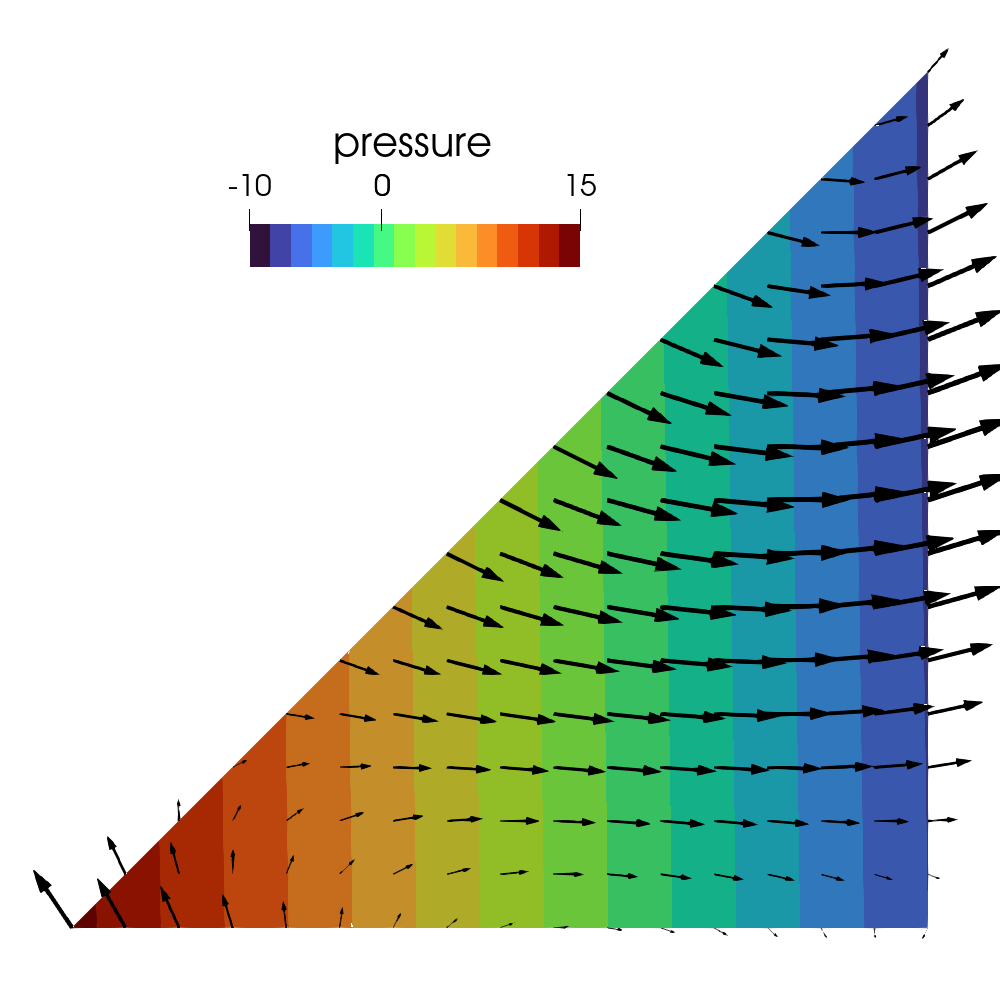} \hspace*{-0.03\textwidth}
    \includegraphics[width=0.21\textwidth]{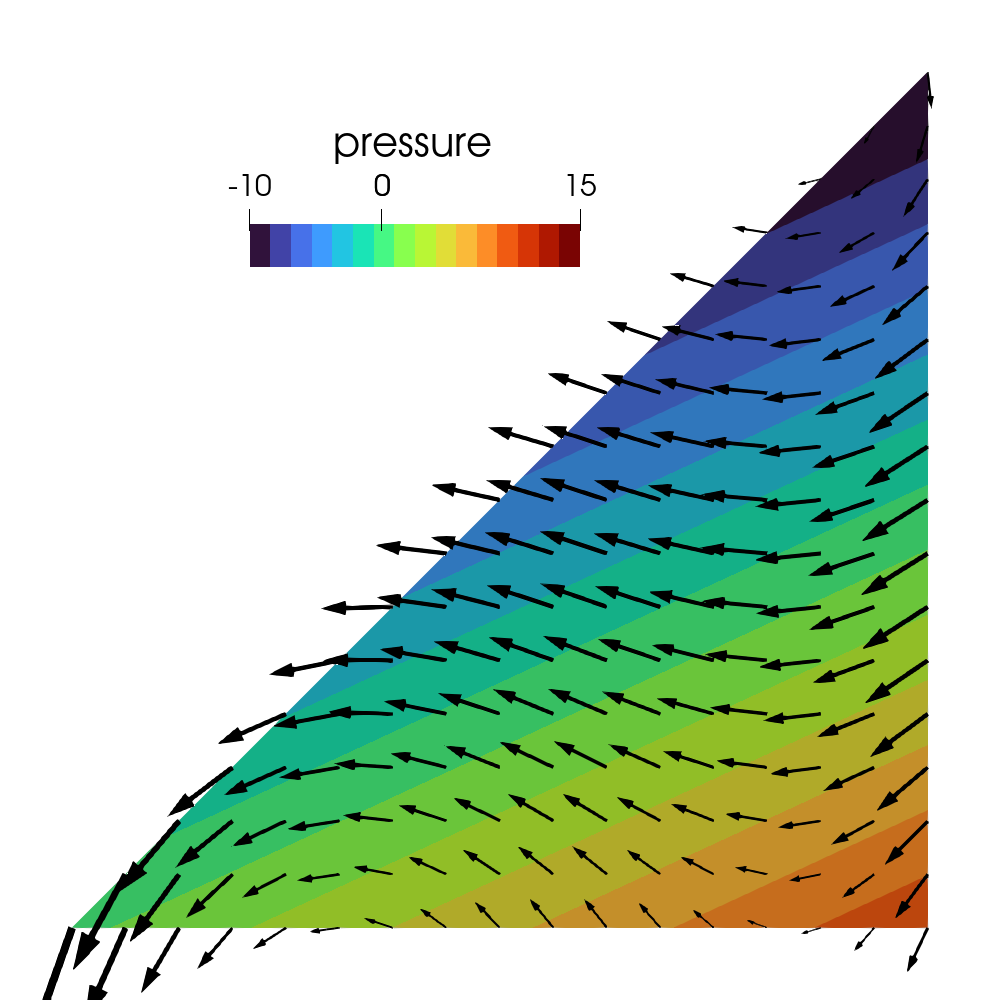} \hspace*{-0.03\textwidth}
    \includegraphics[width=0.21\textwidth]{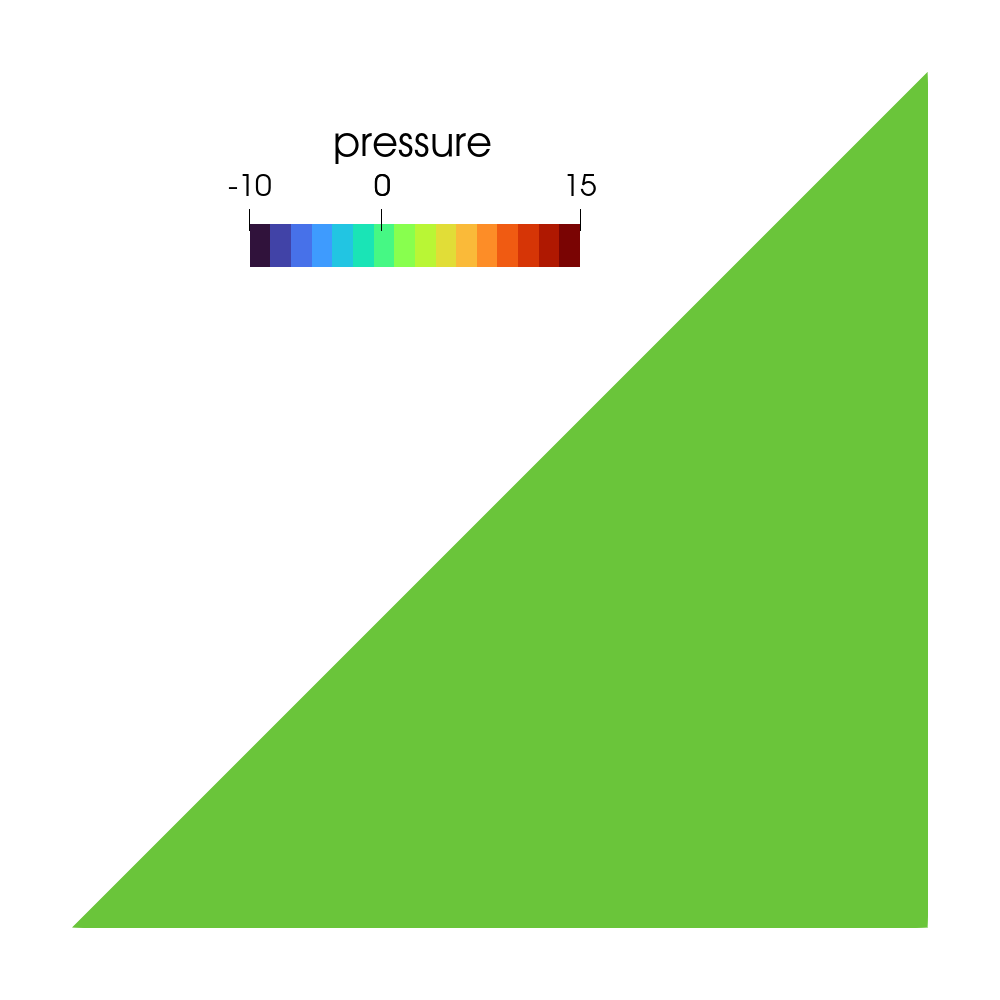} 
    \hspace*{-0.1\textwidth}
    \caption{Example basis functions of a Stokes-Trefftz space for
    $k=2$ on a triangle. The coloring corresponds to the pressure
    value while the arrows indicate the velocity. The pressure scaling
    is different between the first seven and the last three basis
    functions. Note, that these basis functions are obtained from the
    generic approach of the embedded Trefftz-DG method
    \cite{2201.07041}, cf. \cref{sec:impl}, and hence do not offer a
    complete insight into an available structure of the Trefftz space
    such as a clean decomposition into lower and higher order basis
    functions. Nevertheless, we observe that velocity and pressure
    functions are coupled for most basis functions except for the
    three lowest order Trefftz basis functions: The first two basis
    functions (in the upper row) have a zero pressure and are constant
    and linear, respectively, and divergence-free and the last (in the
    lower row) basis function corresponds to a zero velocity and a
    constant pressure.}
    \label{fig:trefftzbasis}
\end{figure}

~
\begin{remark}[Lowest order subspaces]\label{rem:loworder} Due to the
  Trefftz constraint, basis functions of the Trefftz-DG space can no
  longer be separated into velocity and pressure functions in general.
  See \cref{fig:trefftzbasis} for an example set of basis functions
  for $k=2$. However, for the lowest-order polynomial degrees, some
  degrees of freedom can be associated only with velocities and others
  only with pressure functions. The piecewise constant pressure
  functions $\PP^0(\Th)$ are not seen by the Trefftz constraints as
  the pressures only appear as gradients, i.e. these degrees of
  freedom remain in the Trefftz space. From the velocity space, the
  velocity-pressure coupling is removed (from the space) for all
  linear velocity functions, i.e. $[\PP^1(\Th)]^d$, as here the
  Laplacian vanishes. The element-wise divergence constraint removes
  exactly one \dof ~per element from this space ($\Div [\PP^1(\Th)]^d
  = \PP^0(\Th)$). Especially, piecewise constant functions
  $[\PP^0(\Th)]^d$ remain in the Trefftz-DG space. Finally, let us
  note that the problem formulated in the Trefftz-DG space is still of
  saddle-point form with the Lagrange multiplier space being only the
  space of piecewise constant functions. 
\end{remark}


\subsection{Implementation via embedded Trefftz method}\label{sec:impl}
So far we characterized the Trefftz-DG space only implicitly through the Trefftz
constraints. In principle a basis for the homogeneous Trefftz space $\TT_{0,0}$ could be
constructed, see e.g. \cite{POITOU2000561,Bouberbachene}. 
An alternative approach, based on the implicit characterization, is the use of the embedded Trefftz-DG method,
presented in \cite{2201.07041}.
The method allows to construct an embedding $T:
\IT(\Th) \to \X_h(\Th)$ based on the Trefftz constraint equations. Then, all
essential operations can be done by exploiting the underlying DG space. In this
approach it is also straightforward to find a generic element-wise particular
solution $(u_{h}^\part,p_{h}^\part)\in \X_h(\Th)$ required for the homogenization of the
Trefftz problem. In the following, we shortly recap the procedure of the embedded Trefftz method for the Stokes
setting, but refer to \cite{2201.07041} for more details. 

Let us define the matrix and vector associated to the discrete linear system \eqref{eq:dgstokes}
\begin{equation*}\label{def:matrices}
    (\bK)_{ij}=a_h(\phi_j,\phi_i)+b_h(\phi_j,\psi_i)+b_h(\phi_i,\psi_j),
     \qquad 
    (\bl)_i = (f,\phi_i)_\Th + (g,\psi_i)_\Th,
\end{equation*}
for $i,j = 1,\dots, N$ and $(\phi_i,\psi_i)\in \X_h(\Th)$ a set of basis functions with $N$ the number of degrees of freedom.
The Trefftz embedding can then be represented by the kernel of the matrix
\begin{equation*}
    (\bW)_{ij}=(-\Delta \phi_j + \Grad\psi_j, \tilde\phi_i)_\Th + (\Div\phi_j, \tilde\psi_i)_\Th,\qquad \bT=\ker\bW,
\end{equation*}
for a second set of test functions $(\tilde\phi_i, \tilde\psi_i)\in
[\PP^{k-2}(\Th)]^d \times \PP^{k-1}(\Th)$. The Trefftz embedding matrix $\bT$
then allows to characterize a Trefftz basis function as a linear
combination of polynomial DG basis functions. This allows the reduction of the
size of global and local finite element matrices reducing the overall costs
associated with the solution of the arising linear systems.

\section{A-priori error analysis} \label{sec:analysis} In this section
we derive a-priori error bounds for the Trefftz-DG method. To this end
we define the norms
\begin{align*}
  \norm{u}_{1,h}^2 := \nrm{\Grad u}_{\Th}^2 +   \nrm{ h^{-\frac12} \jmp{u}}^2_{\Fh}, \qquad 
  \norm{p}_{0,h}^2:= \nrm{ h \Grad p}_{\Th}^2 + \nrm{ h^\frac12 \jmp{\Pi^0 p}}_{\Fhi}^2.
\end{align*}
For $p_h \in \PP^{k-1}(\Th) /\, \RR$ there holds $\norm{p_h}_0 \simeq \norm{p_h}_{0,h}$.
For the overall problem we define the norm
\begin{align*}
  \nrm{(u,p)}_\TT^2 := \nu \norm{u}_{1,h}^2 + \nu^{-1} \norm{p}_{0,h}^2.
\end{align*}
In addition, as common in the analysis of DG methods, we introduce the stronger norms
\begin{align*}
  \norm{u}_{1,h,\ast}^2 := \norm{u}_{1,h}^2 + \nrm{ h^{\frac12} \ \partial_n u}^2_{\partial \Th}, \qquad 
  \nrm{(u,p)}_{\TT,\ast}^2 := \nu \norm{u}_{1,h,\ast}^2 + \nu^{-1} \norm{p}_{0,h}^2.
\end{align*}
We note that on $\TT(\Th)$  the norms $\nrm{(\cdot,\cdot)}_{\TT}$
and $\nrm{(\cdot,\cdot)}_{\TT,\ast}$ are equivalent with constants independent of $h$ and $\nu$.

For the analysis we will repeatedly rely on discrete trace and inverse inequalities, see e.g. \cite[Lemma 1.44-1.46]{di2011mathematical}, and optimality of the $\Pi^0$-projection on mesh elements and on faces, see e.g. \cite[Lemma 1.58 \& 1.59]{di2011mathematical}.
To show well-posedness of the Trefftz-DG method we will show a discrete inf-sup condition. For a brief recap on the topic see e.g. \cite[Section 1.3.2]{di2011mathematical} or \cite[Section 3.1 and 3.3]{zbMATH06625569}.

\subsection{Saddle-point structure and space decomposition}
In this section, we introduce some structures to analyse and exploit the saddle-point
structure of the variational problem. 
We first note that the space $\{0\} \times \PP^0(\Th) /\,\RR$, with $\PP^0(\Th)$ the space of piecewise
constant pressures form a subspace of $\TT(\Th)$. Accordingly, we can introduce the
decomposition into two subspaces: one that contains all velocity functions and
the high order ($\geq 1$) pressure functions and a second subspace only
consisting of the zero velocity and low-order pressures 
\begin{align} \begin{split}
    &\TT(\Th) = \HH(\Th) \oplus \LL(\Th) \\ 
    \text{with} \quad \LL(\Th) := \{0\} \times \PP^0(\Th) &\quad \text{and} \quad \HH(\Th) := \{(u_h,p_h) \in \TT(\Th) \mid \Pi^0 p_h=0\}.
\end{split}\end{align}
We recall $\TT = \TT(\Th)$, $\HH = \HH(\Th)$ and $\LL = \LL(\Th)$.
Corresponding projection operators for $\LL$ and $\HH$ are denoted by $\Pi^{\LL}$ and
$\Pi^{\HH} = \id -  \Pi^{\LL}$, respectively, 
so that for $(u_h,p_h) \in \TT$ we have $(u_h,p_h) = \Pi^{\HH}(u_h,p_h) + \Pi^{\LL}(u_h,p_h)$ with $  \Pi^{\LL}(u_h,p_h) = (0, \Pi^0 p_h) \in \LL$ and $ \Pi^{\HH} (u_h, p_h) = (u_h, (\id - \Pi^0) p_h)\in \HH$. 
Note that the decomposition effectively operates on the pressure only and is hence $L^2$-orthogonal on the pressure.

We can then introduce norms for the subspaces $\HH$ and $\LL$ by
\begin{align}
  \nrm{(u,p)}_{\HH}^2 := \nu \norm{u}_{1,h}^2 + \nu^{-1} \nrm{h \Grad p}_{\Th}^2,  \qquad
  \nrm{(u,p)}_{\LL}^2 := \nu^{-1} \nrm{h^\frac12 \jmp{\Pi^0 p}}_{\Fh}^2.
\end{align}
These norms are natural in the sense that $\nrm{(u_h,p_h)}_{\HH} =
\nrm{\Pi^\HH(u_h,p_h)}_{\TT}$,  $\nrm{(u_h,p_h)}_{\LL} =
\nrm{\Pi^\LL(u_h,p_h)}_{\TT}$
and 
$
\nrm{(u_h,p_h)}_\TT^2 = \nrm{(u_h,p_h)}_{\HH}^2 +   \nrm{(u_h,p_h)}_{\LL}^2
$
for $(u_h,p_h) \in \TT$.

Next, we observe that the higher-order pressures in $\HH$ are controlled by the velocity
due to the Trefftz constraint while on $\LL$ norms are completely determined by
the lowest order pressure contributions. 
This norm control is highlighted in the following lemma for general functions in $\TT$.
\begin{lemma}\label{lem:normeq}
  For $(u_h,p_h) \in \HH(\Th)$ there holds
  $
  \nrm{(u_h,p_h)}_\HH^2 \simeq \nu \nrm{u_h}_{1,h}^2
  $
  and for $(u_h,p_h) \in \LL(\Th)$ there holds
  $
  \nrm{(u_h,p_h)}_{\LL}^2 = \nu^{-1} \nrm{p_h}_{0,h}^2. 
  $
  Hence $\nrm{(u_h,p_h)}_{\TT}^2 \simeq \nu \nrm{u_h}_{1,h}^2 + \nu^{-1} \nrm{\Pi^0 p_h}_{0,h}^2$
  for $(u_h,p_h) \in \TT(\Th)$. 
\end{lemma}
\begin{proof}
  From the Trefftz constraint we have element-wise $\nabla p_h = \nu
  \Delta u_h$ for $(u_h,p_h) \in \TT$. Exploiting this, together with
  standard local inverse inequalities, yields
  \begin{align*}
    \nu \norm{u_h}_{1,h}^2 \leq \nrm{(u_h,p_h)}_\HH & = \nu \norm{u_h}_{1,h}^2 + \nu^{-1} \nrm{h \Grad p_h}_{\Th}^2 = \nu \norm{u_h}_{1,h}^2 + \nu \nrm{h \ \Delta u_h}_{\Th}^2 \\
    & \lesssim \nu \norm{u_h}_{1,h}^2 +  \nu \nrm{\Grad u_h}_{\Th}^2 \lesssim \nu \norm{u_h}_{1,h}^2. 
  \end{align*}
  Similarly we also have $\nrm{(u_h,p_h)}_{\LL}^2 = \nu^{-1} \nrm{p_h}_{0,h}^2$ which follows from
  $u_h=0$ and $\Grad p_h = 0$ for a function $(u_h,p_h) \in \LL$. Now, splitting $(u_h,p_h) \in \TT$ into its components from
  $\LL$  and $\HH$ and exploiting the previous statements we have
  \begin{align*}
    \nrm{(u_h,p_h)}_\TT^2
    & = \nrm{\Pi^\HH(u_h,p_h)}_{\TT}^2 +   \nrm{\Pi^{\LL}(u_h,p_h)}_{\TT}^2
    = \nrm{\Pi^\HH(u_h,p_h)}_{\HH}^2 +   \nrm{\Pi^{\LL}(u_h,p_h)}_{\LL}^2 \\
    & = \nrm{(u_h,(\id - \Pi^0)p_h)}_{\HH}^2 + \nrm{(u_h, \Pi^0 p_h)}_{\LL}^2
    \simeq \nu \nrm{u_h}_{1,h}^2 + \nu^{-1} \nrm{\Pi^0 p_h}_{0,h}^2.
  \end{align*}
\end{proof}
For the stability analysis, we will make use of the special saddle-point structure,  
that develops from the subspace splitting $\TT = \HH \oplus \LL$, to show that the bilinear form $\Kh(\cdot,\cdot)$ satisfies a discrete inf-sup condition.
To this end we can split the bilinear form $\Kh(\cdot,\cdot)$ into its
contributions that we obtain by restricting to the corresponding subspaces $\LL$
and $\HH$. This gives
\begin{align} \label{eq::K_on_TT}
  \Kh((u_h,p_h),&(v_h,q_h)) \nonumber \\ =
&\Kh(\Pi^\HH(u_h,p_h),\Pi^\HH(v_h,q_h)) + \Kh(\Pi^{\HH}(u_h,p_h),\Pi^\LL(v_h,q_h)) \nonumber \\
&+ \Kh(\Pi^{\HH}(v_h,q_h),\Pi^\LL(u_h,p_h)) + \Kh(\Pi^{\LL}(u_h,p_h),\Pi^\LL(v_h,q_h)) \nonumber \\ 
  =& \Kh(\Pi^\HH(u_h,p_h),\Pi^\HH(v_h,q_h)) + b_h(u_h,\Pi^0 q_h) + b_h(v_h,\Pi^0p_h),
\end{align}
where we used that $\Kh(\Pi^{\LL}(u_h,p_h),\Pi^\LL(v_h,q_h)) = 0$ and $\Pi^{\LL}(u_h,p_h) = (0,\Pi^0 p_h)$ so that the velocity contributions of $\Pi^{\LL}(u_h,p_h)$ vanish.

\subsection{Continuity and kernel-coercivity}
\begin{lemma}
  The bilinear form $\Kh(\cdot,\cdot)$ is continuous, i.e.
    we have for all $(u,p), (v,q) \in \TT(\Th) + [H^{2}(\Th)]^d \times H^1(\Th)$
    \begin{align} \label{eq:Khcontinuity}
        \Kh((u,p),(v,q)) \lesssim \nrm{(u,p)}_{\TT,\ast} \nrm{(v,q)}_{\TT, \ast}.
    \end{align}
\end{lemma}
\begin{proof}
    The bound for $ a_h(u,v)\leq \nrm{u}_{1,h,*} \nrm{v}_{1,h,*} $ follows by standard arguments, see e.g. \cite[Lemma 4.16]{di2011mathematical}.
    Where as for the mixed terms we compute
    \begin{align*}
        b_h&(u,q) \leq \nrm{\Div v}_{0} \nrm{p}_{0} + \nrm{h^{-\frac12}\jmp{v\cdot n}}_{\Fh} \nrm{h^{\frac12}\avg{p}}_{\Fh} 
        \leq \nrm{v}_{1,h} (\nrm{p}_{0} + \nrm{h^{\frac12}\avg{p}}_{\Fh} ).
    \end{align*}
    Using optimality of the $\Pi^0$-projection and a discrete trace inequality we get
    \begin{align*}
        \nrm{h^{\frac12}\avg{p}}_{\Fh} \leq \nrm{h^{\frac12}p}_{\partial \Th} \leq \nrm{h^{\frac12}(p-\Pi^0 p)}_{\partial \Th} + \nrm{h^{\frac12}\Pi^0p}_{\partial \Th} \leq \nrm{h\nabla p}_0+ \nrm{\Pi^0p}_{0},
    \end{align*}
    and by similar reasoning for the volume term $ \nrm{p}_0
    \leq \nrm{h \nabla p}_{0} + \nrm{\Pi^0 p}_0$.
    Finally, we use \cite[Remark 1.1]{brenner2003poincare}, with the fact that $\Pi^0 p \in \PP^0(\Th)/ \IR$, to obtain
    $\nrm{\Pi^0 p}_0 \leq \nrm{h^{\frac12}\jmp{\Pi^0 p}}_{\Fhi}$, finishing the proof.
\end{proof}

Next, we show coercivity of $\Kh|_{\HH \times \HH}$ (``the upper left
block''). 

\begin{lemma}\label{lem:coercivity} Let the stabilisation parameter
  $\alpha>0$ be sufficiently large, then for $(u_h,p_h) \in \HH(\Th)$ there
  holds
  \begin{equation}
    \Kh((u_h,p_h),(u_h,p_h)) \geq c_\HH \nrm{(u_h,p_h)}_{\HH}^2,
  \end{equation}
  for a constant independent of the mesh size $h$ and $\nu$.
\end{lemma}
\begin{proof}
Using decomposition \eqref{eq::K_on_TT} we first have
  \begin{align*}
    \Kh((u_h,p_h),(u_h,p_h)) & = a_h(u_h,u_h) + 2 b_h(u_h,p_h).
    \end{align*}
    For the first term $a_h(u_h,u_h)$ it is well-known that there holds coercivity on $\PP^k(\Th)$ with respect to the discrete norm $\nu^{\frac12} \norm{\cdot}_{1,h}$ for $\alpha>0$ sufficiently large, see for example \cite{arnold2002unified} or \cite[Lemma 4.12]{di2011mathematical}. We denote by $c_\alpha$ a sufficiently lower bound such that for $\alpha\geq c_\alpha$ coercivity with corresponding coercivity constant $c_a$ is guaranteed.
Hence $a_h(u_h,v_h) = a_h^{\alpha=c_\alpha}(u_h,v_h) + (\alpha - c_\alpha) \nu (h^{-\frac12}\jmp{u_h},\jmp{v_h})_{\Fh}$, where $a_h^{\alpha=c_\alpha}(u_h,v_h)$ is the bilinear form with interior penalty stabilization parameter $c_\alpha$ for which there holds
$a_h^{\alpha=c_\alpha}(u_h,u_h) \geq c_a \nu \Vert u_h \Vert_{1,h}^2$.
As $u_h$ is pointwise divergence-free the volume contribution in $b_h(\cdot,\cdot)$ vanishes and we have
  \begin{align*}
    \Kh((u_h,p_h),(u_h,p_h)) & \geq c_{a} \nu \norm{u_h}_{1,h}^2 + \frac{\alpha - c_\alpha}{h} \nu \norm{\jmp{u_h}}_{\Fh}^2 + 2 (\jmp{u_h \cdot n}, \avg{p_h})_{\Fh}.
  \end{align*}
  Thus it remains to bound the latter term. Let $\varepsilon > 0$, then Cauchy-Schwarz
  and Young's inequality ($2 a b \leq \delta^{-1} a^2 + \delta
  b^2,~a,b,\delta \geq 0$) yield
  \begin{align*}
    2 (\jmp{u_h \cdot n}, \avg{p_h})_{\Fh} & \leq \nu \varepsilon^{-1} \Vert{h^{-\frac12}\jmp{u_h \cdot n}}\Vert_{\Fh}^2 + \varepsilon \nu^{-1} \nrm{h^{\frac12}\avg{p_h}}_{\Fh}^2.
  \end{align*}
  Using inverse inequalities for polynomials and recalling that $p_h =
  (\id-\Pi^0)p_h$ (as $(u_h,p_h)\in \HH$) we have with the
  element-wise Trefftz constraint $\Grad p_h = \nu \Delta u_h$ 
  \begin{align*}
    \varepsilon \nu^{-1} \nrm{h^{\frac12} \avg{p_h}}_{\Fh}^2  &\leq c_{\text{inv}} \varepsilon \nu^{-1} \nrm{ h \nabla p_h}_{\Th}^2 = c_{\text{inv}} \varepsilon \nu^{-1}  \nrm{ h \nu \Delta u_h}_{\Th}^2 \leq c_{\text{inv}}' \varepsilon \nu \nrm{\nabla u_h}_{\Th}^2,
  \end{align*} 
  for a constant $c_{\text{inv}}'$ that does not depend on $h$. Choosing $\varepsilon = \frac{c_a}{2 c'_{\text{inv}}}$ we conclude
  \begin{align*}
    \Kh((u_h,p_h),(u_h,p_h)) & \geq \frac12 c_a \nu  \norm{u_h}_{1,h}^2 + \frac{\alpha - c_\alpha - \varepsilon^{-1}}{h} \nu \norm{\jmp{u_h}}_{\Fh}^2,
  \end{align*}
  and hence, for $\alpha$ sufficiently large we obtain the claim with the norm equivalence of \cref{lem:normeq}.
\end{proof}

\subsection{LBB-stability}
We now aim to show the Ladyzhenskaya–Babuška–Brezzi (LBB)-condition for $\Kh|_{\HH \times \LL}$ (``the off-diagonal blocks'').

For the LBB-stability we use a stability result of an auxiliary (low order) discrete problem: Given $\tilde{p}_h \in \PP^1(\Fhi)$ find $(w_h,r_h,\hat{r}_h) \in [\PP^{1}(\Th)]^d \times \PP^0(\Th) /\, \RR \times \PP^1(\Fh) $, s.t.
\begin{subequations}\label{eq:aux}
\begin{align}
  a_{h}^{\nu = 1}(w_h,v_h) & \!+\! (r_h, \Div v_h)_{\Th} \!+\! (\hat{r}_h, \jmp{v_h \cdot n})_{\Fh}  &&
  \!\!\!\!\!\!= 0 &&\!\!\! \forall v_h \in [\PP^1(\Th)]^d, \label{eq:auxa}\\
  (s_h, \Div w_h)_{\Th} &&& \!\!\!\!\!\!= 0 &&\!\!\! \forall s_h \in \PP^0(\Th)  /\, \RR,\!\! \label{eq:auxb}\\
  (\hat{s}_h, \jmp{w_h \!\cdot\! n})_{\Fh} 
  &&& \!\!\!\!\!\! = (\tilde{p}_h, \hat{s}_h)_{\Fhi}   &&\!\!\! \forall \hat{s}_h \in  \PP^1(\Fh). \label{eq:auxc}
\end{align}
\end{subequations}
Here, the bilinear form $a_{h}^{\nu = 1}(\cdot,\cdot)$ is to be understood as the bilinear form that is obtained from $a_{h}(\cdot,\cdot)$ when setting $\nu$ to $1$, so that the auxiliary problem is independent of $\nu$. In this auxiliary problem, $r_h$ and $\hat{r}_h$ are the Lagrange multipliers for enforcing the pointwise divergence-constraint $\Div w_h=0$ and the normal jump condition $\jmp{w_h \cdot n} = \tilde{p}_h$, respectively.
For $\tilde{p}_h = 0$ the solution $w_h$ is normal-continuous, i.e. $H(\Div)$-conforming. More precisely we would have $w_h \in \BDM^{1}(\Th) = [\PP^1(\Th)]^d \cap H(\Div;\Omega)$ with $\BDM^{1}$ the Brezzi-Douglas-Marini space, cf. \cite[Section 2.3.1]{brezzi}. For general $\tilde{p}_h$ we instead have 
$w_h \in \BDM^{-1}(\Th) = [\PP^1(\Th)]^d$ with $\BDM^{-1}(\Th)$ the Brezzi-Douglas-Marini element with broken $H(\Div)$-continuity.

\begin{lemma}\label{lem:aux} The velocity solution $w_h \in
  [\PP^1(\Th)]^d$ of the auxiliary problem \eqref{eq:aux} is
  element-wise divergence-free $\Div w_h = 0$ and further there holds
  $\jmp{w_h \cdot n} = \tilde{p}_h$ on $\Fhi$ and $\nrm{w_h}_{1,h}
  \lesssim \nrm{h^{-\frac12}\tilde{p}_h}_{\Fhi}$.
\end{lemma}
\begin{proof}
  Very closely related statements can be found in the literature of hybridized mixed and $H(\Div)$-conforming DG and hybrid DG methods, cf. e.g. \cite{cockburn2005locally,cockburn2007note}. For completeness, we give a proof here.
  We analyse the twofold saddle-point problem \eqref{eq:aux} with the bilinear forms $d_h: [\PP^1(\Th)]^d \times \PP^0(\Th) \to \mathbb{R}$, $d_h(w_h,s_h) = (s_h,\Div w_h)_\Th$ and 
  $e_h: [\PP^1(\Th)]^d \times \PP^1(\Fh) \to \mathbb{R}$, 
  $e_h(w_h, \hat{s}_h) = (\hat{s}_h, \jmp{w_h \cdot n})_\Fh$. A similar analysis has been carried out in \cite[Lemma 4.3.6]{alemanmaster}. From element-wise $H(\Div)$-interpolation, cf. e.g. \cite[Proposition 2.5.1]{brezzi}, and standard scaling arguments we have that $\jump{w_h\cdot n}$ can be matched with scalar data on facets in the following sense: For all $\hat{s}_h \in \PP^0(\Fh)$, and there holds
  \begin{equation}
    \sup_{w_h \in [\PP^1]^d} \frac{(\hat{s}_h, \jmp{w_h \cdot n})_\Fh}{\nrm{w_h}_{1,h}} \gtrsim \nrm{h^{-\frac12}\hat{s}_h}_{\Fh}.
  \end{equation}
  The kernel of $e_h(\cdot,\cdot)$ are $H(\Div)$-conforming functions in $\BDM^1(\Th)$. For this subspace, we have the following well-known inf-sup result, cf. \cite{hansbolarson02,cockburn2007note}:
  \begin{equation}
    \sup_{w_h \in \BDM^1} \frac{(s_h, \Div w_h)_\Th}{\nrm{w_h}_{1,h}} \gtrsim \nrm{s_h}_{0}.
  \end{equation}
  Hence, we have inf-sup stability of $d_h(\cdot,\cdot)$ on $\ker e_h$ (i.e. functions with normal continuity), inf-sup stability of $e_h(\cdot,\cdot)$ and coercivity of $a_h(\cdot,\cdot)$ on $\ker d_h \cap \ker e_h$ (i.e. divergence free and normal continuous functions). Here we use a slight abuse of notation associating the kernel of the bilinear form with the kernel of the corresponding operator. This implies stability of the twofold saddle-point problem, cf. e.g. \cite{saddle} and thus the stability estimate in the claim.
\end{proof}

We can now state the LBB-stability result for the Trefftz-DG problem:
\begin{lemma}\label{lem:LBB}
  There holds
  \begin{align*}
    &  \inf_{(v_h,q_h) \in \LL} \sup_{(u_h,p_h) \in \HH} \frac{\Kh((u_h,p_h),(v_h,q_h))}{\nrm{(u_h,p_h)}_{\HH} \nrm{(v_h,q_h)}_{\LL}}  \simeq
    \inf_{q_h \in \PP^0} \sup_{(u_h,p_h) \in \HH} \frac{b_h(u_h,q_h)}{\nrm{u_h}_{1,h} \nrm{q_h}_{0,h}}
    \geq c_b,
  \end{align*}
  for a constant $c_b > 0$ independent of $h$ and $\nu$.
\end{lemma}
\begin{proof}
  The first equivalence immediately follows by \cref{lem:normeq}.
  Now let $(v_h,q_h) \in \LL$ be arbitrary, i.e. $v_h = 0$ and $q_h \in \PP^0$.
  We define $\tilde{p}_h = -\jmp{q_h} \in \PP^0(\Fhi)$ on $\Fhi$ and will construct a suitable velocity field $u_h$ that allows to control $\tilde{p}_h$.
  For this we use the element-wise divergence-free space $\BDM^{-1}_0(\Th) =
  \{v \in \BDM^{-1}(\Th) \mid \Div v|_{\Th} = 0\}$. Since for functions $v_h$ in $\BDM^{-1}_0(\Th)$ there holds on each element that $\Delta v_h =
  0$ (since $v_h$ is linear) and $\Div v_h = 0 $, the tuple $(v_h,0) \in \HH$, thus $\BDM^{-1}_0 \times \{0\} \subset \HH$. This gives
  $$
  \inf_{q_h \in \PP^0} \sup_{(u_h,p_h) \in \HH} \frac{b_h(u_h,q_h)}{\nrm{u_h}_{1,h} \nrm{q_h}_{0,h}}
  \geq \inf_{q_h \in \PP^0} \sup_{u_h \in \BDM^{-1}_0(\Th)} \frac{b_h(u_h,q_h)}{\nrm{u_h}_{1,h} \nrm{q_h}_{0,h}}.
  $$
  Now choose $u_h=w_h$ with $w_h$ being the solution of the auxiliary problem \eqref{eq:aux} with data (used for the right hand side) $\tilde{p}_h$. This gives
  \begin{align*}
    b_h(u_h,q_h)  = (\avg{q_h}, \jmp{u_h \cdot n})_{\Fh}  
    &=
    \sum_{T\in\Th} (\avg{q_h}, u_h \cdot n)_{\partial T} - \underbrace{(\Div u_h, q_h)_T}_{=0} \\
    & =\sum_{T\in\Th} (\avg{q_h}-q_h, u_h \cdot n)_{\partial T} + ( u_h, \underbrace{ \nabla q_h}_{=0})_T  \\
    &= -(\frac12 \jmp{q_h}, \jmp{u_h \cdot n})_{\Fhi} = \frac12 \nrm{\jmp{q_h}}_{\Fhi}^2,
  \end{align*}
  and thus by the stability estimates of \cref{lem:LBB} for $u_h=w_h$ we get
  \begin{align*}
    b_h(u_h,q_h) & = \frac12 \nrm{h^{\frac12}\jmp{q_h}}_{\Fhi} \nrm{h^{-\frac12}\jmp{q_h}}_{\Fhi} \gtrsim \nrm{h^{\frac12}\jmp{q_h}}_{\Fhi} \nrm{u_h}_{1,h} \simeq \nrm{q_h}_{0,h} \nrm{u_h}_{1,h},
  \end{align*}
  which concludes the proof.
\end{proof}

\subsection{Inf-Sup-Stability}
\begin{theorem} \label{th:Khinfsup}
  For $(v_h,q_h) \in \TT(\Th)$ there holds
  \begin{equation}
    \sup_{(w_h,r_h) \in \TT} \frac{\Kh((w_h,r_h),(v_h,q_h))}{\nrm{(w_h,r_h)}_{\TT} } \geq c_\TT \nrm{(v_h,q_h)}_{\TT} \geq c^\ast_\TT \nrm{(v_h,q_h)}_{\TT,\ast},
  \end{equation}
  for constants $c_\TT, c^\ast_\TT$ independent of $h$ and $\nu$ and hence the Trefftz-DG problem \eqref{eq:discreteTrefftz} admits a unique solution that continuously depends on the data. 
\end{theorem}
\begin{proof}
    Together with \cref{lem:coercivity} and \cref{lem:LBB} and the continuity \eqref{eq:Khcontinuity} we have shown that $\Kh(\cdot,\cdot)$ is inf-sup-stable on $\TT$ and thus well-posed, see e.g. \cite[Lemma 1.30]{di2011mathematical}.
\end{proof}

\subsection{Quasi-best approximation and Aubin-Nitsche}
In the following, we consider a best approximation result in the Trefftz space and discuss
an Aubin-Nitsche-like result for the $L^2$-error of the velocity.

\begin{lemma}
  Let $(u,p) \in [H^{2}(\Th)]^d \times H^{1}(\Th) /\, \RR$ be
  the solution of the Stokes problem~\eqref{eq:weakbasicpde} and
  $(u_h,p_h) \in \TT_{f,g}(\Th)$ be the discrete solution to
  \eqref{eq:discreteTrefftz}. Then, there holds
  \begin{equation}
    \nrm{(u_h-u,p_h-p)}_{\TT} \leq \nrm{(u_h-u,p_h-p)}_{\TT,\ast} \lesssim \inf_{(v_h,q_h) \in \TT_{f,g}} \nrm{(u-v_h,p-q_h)}_{\TT,\ast}.
  \end{equation}
\end{lemma}
\begin{proof} For $(v_h,q_h) \in \TT_{f,g}$ arbitrary we first have $(u_h-v_h,p_h-q_h) \in \TT$. 
  Next, let $(w_h,r_h) \in \TT$ be a supremizer in \cref{th:Khinfsup} to $(u_h-v_h,p_h-q_h) \in \TT$, then
  \begin{align*}
    c^\ast_\TT  \nrm{(u_h-v_h,p_h-q_h)}_{\TT,\ast} & \leq \frac{\Kh((w_h,r_h),(u_h-v_h,p_h-q_h))}{\nrm{(w_h,r_h)}_{\TT} } \\ & = \frac{\Kh((w_h,r_h),(u-v_h,p-q_h))}{\nrm{(w_h,r_h)}_{\TT} }
    \lesssim \nrm{(u-v_h,p-q_h)}_{\TT,\ast},
  \end{align*}
  where we made use of the consistency of the formulation, and continuity~\eqref{eq:Khcontinuity}. Finally, the claim follows with an application of the triangle inequality.
\end{proof}

To show error estimates in the $L^2$-norm we use a duality argument in the style of Aubin-Nitsche. This requires additional regularity for the solution of the Stokes problem, see e.g. \cite{10.3792/pjaa.67.171} for regularity results of the Stokes problem mentioned below.

\begin{theorem} \label{th::aubin}
  Assume the domain boundary to be sufficiently smooth or $\Omega$ to
  be convex so that $L^2$-$H^2$-regularity holds and let $(u,p) \in
  [H^{2}(\Th)]^d \times H^{1}(\Th) /\, \RR$, be the solution of
  the Stokes problem \eqref{eq:weakbasicpde} and $(u_h,p_h) \in
  \TT_{f,g}(\Th)$ be the discrete solution to \eqref{eq:discreteTrefftz}.
  Then, there holds
  \begin{equation} \label{eq::aubinerror}
    \nrm{u-u_h}_{\Omega} \lesssim h \nrm{(u_h-u,p_h-p)}_{\TT,\ast}.
  \end{equation}
\end{theorem}
  \begin{proof}
    We apply the Aubin-Nitsche trick and pose the auxiliary adjoint Stokes problem $\calL(w,r) =
    (u-u_h,0)$ on $\Omega$. With $u-u_h \in [L^2(\Omega)]^d$ and the assumed
    $L^2$-$H^2$-regularity, we have $(w,r) \in [H^2(\Omega)]^d \times H^1(\Omega)$.

    Now, from the (adjoint) consistency of the bilinear forms
    $a_h(\cdot,\cdot)$ and $b_h(\cdot,\cdot)$ we have
    \begin{align*}
      \nrm{u-u_h}_{\Omega}^2
      =&a_h(w,u-u_h) + b_h(u-u_h,r) \quad\text{and}\quad b_h(w,p-p_h) = 0.
    \end{align*}      
    Due to the consistency of the primal problem, we can add the following
    expression which adds up to zero for any $(w_h,r_h) \in \TT$:
    \begin{align*}
      \nrm{u-u_h}_{\Omega}^2 =&a_h(w,u-u_h) + b_h(u-u_h,r) + b_h(w, p-p_h) 
       - a_h(w_h,u-u_h) \\ & -b_h(u-u_h,r_h) - b_h(w_h,p-p_h) 
        \\ 
         =&a_h(w-w_h,u-u_h) + b_h(u-u_h,r-r_h) + b_h(w-w_h, p-p_h). 
    \end{align*}
    With $w_h = \Pi^{\BDM,1} w$, the $\BDM^1$-interpolant of $w$, we have
    $\Div w_h = \Pi^0 \Div w = 0$ and hence with $r_h = \Pi^0 r$ there holds
    $(w_h,r_h) \in \TT^1 \subset \TT$. 
    Now applying
    standard interpolation results for the $\BDM$-interpolator as well as exploiting the $L^2$-$H^2$-regularity yields the bound 
    $$ \nrm{ (w-w_h,r-r_h) }_{\TT,\ast} \lesssim h \left( \snrm{w}_{H^2(\Omega)} +
      \snrm{r}_{H^1(\Omega)} \right) \lesssim h \nrm{u-u_h}_{\Omega}.$$
    We can then conclude with continuity~\eqref{eq:Khcontinuity} to obtain
    \begin{align*}
      \nrm{u-u_h}_{\Omega}^2
      =&a_h(w-w_h,u-u_h) + b_h(u-u_h,r-r_h) + b_h(w-w_h, p-p_h) \\
      \lesssim&  \nrm{ (w-w_h,r-r_h) }_{\TT,\ast} \nrm{ (u-u_h,p-p_h) }_{\TT,\ast}
      \\
      \lesssim&  h \nrm{u-u_h}_{\Omega} \nrm{ (u-u_h,p-p_h) }_{\TT,\ast}.
    \end{align*}
    Dividing by $\nrm{u-u_h}_{\Omega}$ then yields the claim.
    \end{proof}

\subsection{Approximation and a-priori error bounds}

We conclude the stability analysis with the following optimal error estimate. 

\begin{lemma}
Let $(u,p) \in [H^{k+1}(\Th)]^d\cap [H^1(\Omega)]^d \times H^{k}(\Th)
/\, \RR$ be the solution of the Stokes problem.
  There holds
  \begin{equation} \label{eq:approx}
    \inf_{(v_h,q_h) \in \TT_{f,g}} \nrm{(u-v_h,p-q_h)}_{\TT,\ast} \lesssim
    \nu^{\frac12} h^{k} \snrm{ u}_{H^{k+1}(\Th)} + \nu^{-\frac12} h^{k} \snrm{p}_{H^{k}(\Th)}.
  \end{equation}
\end{lemma}
\begin{proof}
  Let $(u',p') \in [H^{2}(\Omega)]^d \times H^{1}(\Omega)$ be the
  solutions to the Stokes problem with right-hand side data
  $\Pi^{k-2}f$ and $\Pi^{k-1}g$. Let $T^k: L^1(\Th) \to \PP^k(\Th)$
  denote the element-wise (potentially vectorial) averaged Taylor
  polynomial (averaged on a proper inner ball of each element) of
  degree $k$, cf. \cite[Section 4.1-4.3]{brennerscott}. Then, we set $v_h = T^k u'$ and
  $q_h = T^{k-1} p'$ and have on each element
\begin{align*}
  -\Delta v_h + \Grad q_h\! &=\! - \Delta T^k u' + \Grad T^{k-1} p'\! 
  = T^{k-2}(\! \overbrace{- \Delta u' + \Grad \!p'}^{\Pi^{k-2} f} )\! = \Pi^{k-2} f, \\
 \Div v_h &= \Div T^k u' = T^{k-1} \Div u' = T^{k-1} \Pi^{k-1} g = \Pi^{k-1} g.
\end{align*}
Hence $(v_h,q_h) \in \TT_{f,g}$.
Now we bound:
\begin{align*}
  \nrm{(u-v_h,&p-q_h)}_{\TT,\ast} 
= \nrm{(u-T^ku',p-T^{k-1}p')}_{\TT,\ast} \\
 &\leq \nrm{(u-T^ku,p-T^{k-1}p)}_{\TT,\ast} + \nrm{(T^k(u-u'),T^{k-1}(p-p'))}_{\TT,\ast} = 
 I + II.
\end{align*}
The first part, $I$, can be directly bounded by the right-hand side of \eqref{eq:approx} due to the approximation properties of the averaged Taylor polynomial, cf. \cite[Lemma 4.3.8]{brennerscott}.
We hence turn our attention to $II$. As the terms in $II$ are discrete functions (piecewise polynomials), we can apply inverse inequalities to obtain
\begin{align*}
  II & \lesssim  \nu^{\frac12} \snrm{\nabla T^k (u-u')}_{\Th} + \nu^{\frac12} \nrm{h^{-\frac12} \jmp{T^k (u-u')}}_{\Fh} + \nu^{-\frac12} \nrm{T^{k-1}(p - p')}_{\Omega} \\ &= \nu^{\frac12} (II_a + II_b) + \nu^{-\frac12} II_c.
\end{align*}
Define $v
= u -u'$. Then we have by the properties of the averaged Taylor
polynomial, cf. \cite[Lemma 4.3.8]{brennerscott},
\begin{align*}
  II_a & \lesssim  \snrm{ \nabla (T^k - \id) v}_{\Th} + \snrm{ \nabla v}_{\Th} \lesssim \snrm{ \nabla v}_{\Th}.
\end{align*}
Let $\Pi_C$ be the (vector-valued) Clément interpolation operator $\Pi_C:
[H^1(\Omega)]^d \to [\PP^1(\Th)]^d\cap [C^0(\Omega)]^d$, see \cite{clement}.
  Now, using (in order),
  $\Pi_C v \in C^0(\Omega)$, 
  a trace inequality, triangle inequalities,   
  the interpolation properties of $T^k$ and $\Pi_C$,
  and (local) $H^1$-continuity of $T^k$ and $\Pi_C$ 
  we have
  \begin{align*}
  II_b & \lesssim \nrm{h^{-\frac12} \jmp{ (T^k - \Pi_C) v}}_{\partial \Th} \lesssim  \nrm{ h^{-1}(T^k - \Pi_C) v}_{\Th} + \nrm{\nabla (T^k - \Pi_C) }_{\Th} \\
 &\lesssim  \nrm{ h^{-1}(T^k - \id) v}_{\Th} + \nrm{h^{-1} (\Pi_C - \id) v}_{\Th}
 + \nrm{\nabla T^k v }_{\Th} + \nrm{\nabla \Pi_C v }_{\Th} \lesssim  \snrm{ \nabla v}_{\Th}.
 \end{align*}
Similarly we also have $II_c  =  \nrm{ T^k (p-p')}_{\Omega} \lesssim
 \nrm{ p-p'}_{\Omega}$, hence 
 $$II \leq \nu^{\frac12}  \nrm{ u-u'}_{H^1(\Omega)} + \nu^{-\frac12}
 \nrm{ p-p'}_{\Omega}.$$ We note that $(u-u', p-p')$ solves the Stokes
 problem for the data $(f- \Pi^{k-2}f,g-\Pi^{k-1}g)$. Exploiting
 linearity and stability of the continuous problem (see \cite[Theorem 8.2.1]{di2011mathematical}) and that $f-
 \Pi^{k-2}f$ is ($L^2$-)orthogonal to $\Pi^0 v$ for all $v \in
 H^1(\Omega)$ we have
\begin{align*}
  II & \lesssim \nu^{-\frac12} \nrm{f - \Pi^{k-2} f}_{H^{-1}(\Omega)} + \nu^{-\frac12} \nrm{g - \Pi^{k-1} g}_{0} \\
  & \lesssim \nu^{-\frac12} \sup_{v \in H^1(\Omega)} \frac{ (f - \Pi^{k-2} f,v)}{\nrm{v}_{H^{1}(\Omega)}} + \nu^{-\frac12} \snrm{h^k \ g}_{H^k(\Th)} \\
  & = \nu^{-\frac12} \sup_{v \in H^1(\Omega)} \frac{ ((\id - \Pi^{k-2}) f,v - \Pi^0 v)}{\nrm{v}_{H^{1}(\Omega)}} + \nu^{-\frac12} \snrm{h^k \ g}_{H^k(\Th)} \\
  & \lesssim \nu^{-\frac12} \sup_{v \in H^1(\Omega)} \frac{\nrm{h(\id - \Pi^{k-2}) f}_{\Th} \snrm{ v}_{H^1(\Th)}}{\nrm{v}_{H^{1}(\Omega)}} + \nu^{-\frac12} \snrm{h^k\ g}_{H^k(\Th)} \\
  &   \lesssim \nu^{-\frac12} \nrm{h (\id - \Pi^{k-2}) f}_{\Th} + \nu^{-\frac12} \snrm{h^k \ g}_{H^k(\Th)}
  \\
  & \lesssim \nu^{-\frac12} \left( \snrm{h^k f}_{H^{k-1}(\Th)} + \snrm{h^k \ g}_{H^k(\Th)} \right)
    \lesssim h^k \left( \nu^{\frac12} \snrm{u}_{H^{k+1}(\Th)} + \nu^{-\frac12}  \snrm{p}_{H^k(\Th)} \right),
\end{align*}
what concludes the proof.

\end{proof}

\begin{cor}
  Let $(u,p) \in [H^{k+1}(\Th)]^d\cap [H^1(\Omega)]^d \times H^{k}(\Th) /\, \RR$ be the solution of the Stokes problem and $(u_h,p_h) \in \TT_{f,g}(\Th)$ be the discrete solution to \eqref{eq:discreteTrefftz}. 
  Then, there holds
  \begin{equation} \label{eq:aprioribound1}
    \nrm{(u-v_h,p-q_h)}_{\TT} \lesssim \nu^{\frac12} h^{k} \snrm{u}_{H^{k+1}(\Th)} + \nu^{-\frac12} h^{k} \snrm{p}_{H^{k}(\Th)}.
  \end{equation}
  Assuming $L^2$-$H^2$-regularity there further holds
  \begin{equation}
    \nrm{u-u_h}_{\Omega} \lesssim \nu^{\frac12} h^{k+1} \snrm{u}_{H^{k+1}(\Th)} + \nu^{-\frac12} h^{k+1} \snrm{p}_{H^{k}(\Th)}.
  \end{equation}
\end{cor}    
\begin{proof}
  This is a direct consequence of the previous Lemma and \cref{th::aubin}.
\end{proof}


\section{Numerical examples}\label{sec:num}

The method has been implemented using \texttt{NGSolve} \cite{ngsolve} and \texttt{NGSTrefftz} \cite{ngstrefftz} \footnote{Reproduction material is available in \cite{UHPEQG_2023}.}.

\subsection{Exact solution}

For the first numerical example, we consider the domain $\Omega =
(0,1)^d$ and the solution
\begin{align*}
\begin{cases}
  u=\Curl \zeta, \qquad\quad\,\,\,  p = \sin(\pi (x+y)) & \textrm{for } d = 2,\\
  u=\Curl (\zeta,\zeta,\zeta),\quad  p =  \sin(\pi (x+y+z)) + 8/\pi^3 & \textrm{for } d = 3,
\end{cases}
\end{align*}
with the potential 
    $\zeta = \cos(\pi x(1-x) y(1-y))$
and 
$\zeta = \cos(\pi x(1-x) y(1-y) z(1-z))$
for $d=2,3$, respectively. We then compute the numerical solution with
the given right hand side $f=-\nu \Delta u + \nabla p$ (using above
exact solution), $g=0$ and homogeneous Dirichlet boundary data.
Further, for simplicity, we fix $\nu=1$. The penalty parameter is
chosen as $\alpha = 20$.

In Figure \ref{fig:stokes} we compare the $L^2$-error of the Trefftz-DG
method \eqref{eq:discreteTrefftz} (marked with $\TT^k$) to the
solution of the standard DG method \eqref{eq:DGK} (marked with
$\IP^k$) for orders $k=2,3,4$. The error of the velocity and pressure
approximation with respect to the exact solution is measured in the
$L^2$-norm. As expected, see \eqref{eq:approx} and
\eqref{eq::aubinerror}, we observe that the solution of the Trefftz-DG
method converges with the same (optimal) order as the solution of the
standard DG method for both the velocity and the pressure error.

\begin{figure}[ht]
\centering
    \begin{tikzpicture}[scale=0.75,spy using outlines={circle, magnification=4, size=2cm, connect spies}]
    \begin{groupplot}[%
      group style={%
        group name={my plots},
        group size=2 by 2,
        vertical sep=1.5cm,
        horizontal sep=2cm,
      },
    legend style={
        legend columns=8,
        at={(-0.22,-0.15)},
        anchor=north,
        draw=none
    },
    every axis x label/.style={at={(0.5,0)},below,yshift=-.75em},
    xlabel={$h$},
    ymajorgrids=true,
    grid style=dashed,
    cycle list name=paulcolors2,
    ]      
    \nextgroupplot[ymode=log,xmode=log,x dir=reverse, ylabel={$\norm{u-u_h}_\Th$}]
        \foreach \k in {2,3,4}{
            \addplot+[discard if not={k}{\k},discard if not={trefftz}{True}] table [x=h, y=ul2error, col sep=comma] {ex/ex_2d.csv};
            \addplot+[discard if not={k}{\k},discard if not={trefftz}{False}] table [x=h, y=ul2error, col sep=comma] {ex/ex_2d.csv};
        }
        \addplot[domain=0.06:0.7] {exp(-3*ln(1/x)-3.0)};
        \addplot[domain=0.06:0.7] {exp(-4*ln(1/x)-3.5)};
        \addplot[domain=0.06:0.7] {exp(-5*ln(1/x)-4.0)};
    \nextgroupplot[ymode=log,xmode=log,x dir=reverse, ylabel={$\norm{p-p_h}_\Th$}]
        \foreach \k in {2,3,4}{
            \addplot+[discard if not={k}{\k},discard if not={trefftz}{True}] table [x=h, y=pl2error, col sep=comma] {ex/ex_2d.csv};
            \addplot+[discard if not={k}{\k},discard if not={trefftz}{False}] table [x=h, y=pl2error, col sep=comma] {ex/ex_2d.csv};
        }
        \addlegendimage{solid}
        \addplot[dashed, domain=0.06:0.7] {exp(-2*ln(1/x)-0.5)};
        \addplot[dashed, domain=0.06:0.7] {exp(-3*ln(1/x)-0.7)};
        \addplot[dashed, domain=0.06:0.7] {exp(-4*ln(1/x)-1.0)};

    \nextgroupplot[ymode=log,xmode=log,x dir=reverse, ylabel={$\norm{u-u_h}_\Th$}]
        \foreach \k in {2,3,4}{
            \addplot+[discard if not={k}{\k},discard if not={trefftz}{True}] table [x=h, y=ul2error, col sep=comma] {ex/ex_3d.csv};
            \addplot+[discard if not={k}{\k},discard if not={trefftz}{False}] table [x=h, y=ul2error, col sep=comma] {ex/ex_3d.csv};
        }
        \addplot[domain=0.12:0.7] {exp(-3*ln(1/x)-5.5)};
        \addplot[domain=0.12:0.7] {exp(-4*ln(1/x)-6.0)};
        \addplot[domain=0.12:0.7] {exp(-5*ln(1/x)-7.0)};
    \nextgroupplot[ymode=log,xmode=log,x dir=reverse, ylabel={$\norm{p-p_h}_\Th$}]
        \foreach \k in {2,3,4}{
            \addplot+[discard if not={k}{\k},discard if not={trefftz}{True}] table [x=h, y=pl2error, col sep=comma] {ex/ex_3d.csv};
            \addplot+[discard if not={k}{\k},discard if not={trefftz}{False}] table [x=h, y=pl2error, col sep=comma] {ex/ex_3d.csv};
        }
        \addlegendimage{solid}
        \addplot[dashed, domain=0.12:0.7] {exp(-2*ln(1/x)-1.7)};
        \addplot[dashed, domain=0.12:0.7] {exp(-3*ln(1/x)-2.9)};
        \addplot[dashed, domain=0.12:0.7] {exp(-4*ln(1/x)-3.5)};
        \legend{$\IT^2$,$\IP^2$,$\IT^3$,$\IP^3$,$\IT^4$,$\IP^4$,$\mathcal O(h^{k+1})$,{$\mathcal O(h^{k})$ for $k=2,3,4$},}

    \end{groupplot}
    \node[anchor=south] at ($(my plots c1r1.north east)!0.5!(my plots c2r1.north west)$){Example in 2 dimensions};
    \node[anchor=south] at ($(my plots c1r2.north east)!0.5!(my plots c2r2.north west)$){Example in 3 dimensions};
\end{tikzpicture}
\vspace{-0.5cm}
\caption{Numerical results for the DG and Trefftz-DG method for the two dimensional example on the top row and three dimensional example on the bottom. The gray (solid and dashed) lines indicate the expected convergence rates.}
    \label{fig:stokes}
\end{figure}
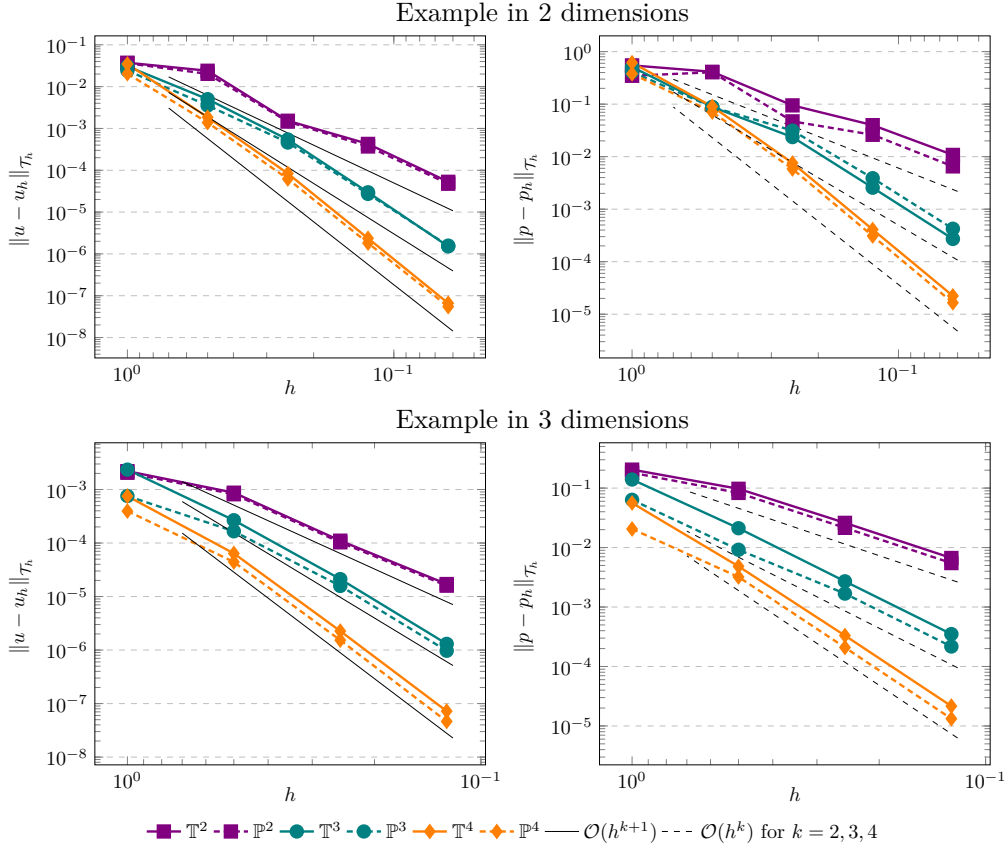

\subsection{Moffatt eddies}
In \cite{moffatt_1964} Moffatt presents an example on a wedge that
produces an infinite amount of eddies that differ greatly in
magnitude, making it a challenging numerical example. This benchmark
has also been considered in \cite{ainsworthparker}.
We consider a triangular domain, shown in \Cref{fig:moffatt} on the right, with non-homogeneous Dirichlet boundary conditions on one side, and zero source term, $f=(0,0)$.
To implement the non-homogeneous Dirichlet boundary conditions we add to the right hand side of our Trefftz-DG discretization \eqref{eq:discreteTrefftz} the following boundary terms
\begin{align*}
    \frac{\alpha\nu}{h} (u_D, v)_{\partial\Omega} - (\nu u_D, \partial_n v_h)_{\partial\Omega},
\end{align*}
where $u_D$ is the Dirichlet boundary data.
We impose an inflow on the part of the boundary given by the line $y=0$ with a
parabolic velocity profile and a no-slip condition (homogeneous Dirichlet boundary) is imposed on the
remaining boundary, i.e. the boundary data $u_D$ is given by
\begin{equation*}
    u_D(x,0)=
      (1-x^2,0)
        \quad \text{for} \quad -1\leq x\leq 1,\quad
        u_D=(0,0)\quad\text{on}\quad \partial\Omega\setminus (-1,1)\times\{0\}.
\end{equation*}
The domain
$\Omega$ is given by a wedge with a sharp angle of approximately
$2\alpha=36.87^\circ$ on the corner opposite of the boundary with the
inflow. 
The velocity solution comprises an endless series of swirls, with each subsequent swirl being approximately 400 times less intense than the one preceding it.
The series of swirls converges to $(0,-2)$
Additionally, the pressure field exhibits two point singularities at $(-1,0)$ and $(1,0)$.

In \Cref{fig:moffatt} we show the numerical results on a mesh with 28 elements for $k=10$.
We see that the Trefftz-DG method is able to capture the sharp corner and the eddies, resolving four to five eddies, encompassing a scale range of $10^{13}$.

\begin{figure}[ht]
    \centering
    \vspace*{-0.5cm}
    \begin{tikzpicture}
    \begin{scope}
    \node {\includegraphics[width=0.39\textwidth]{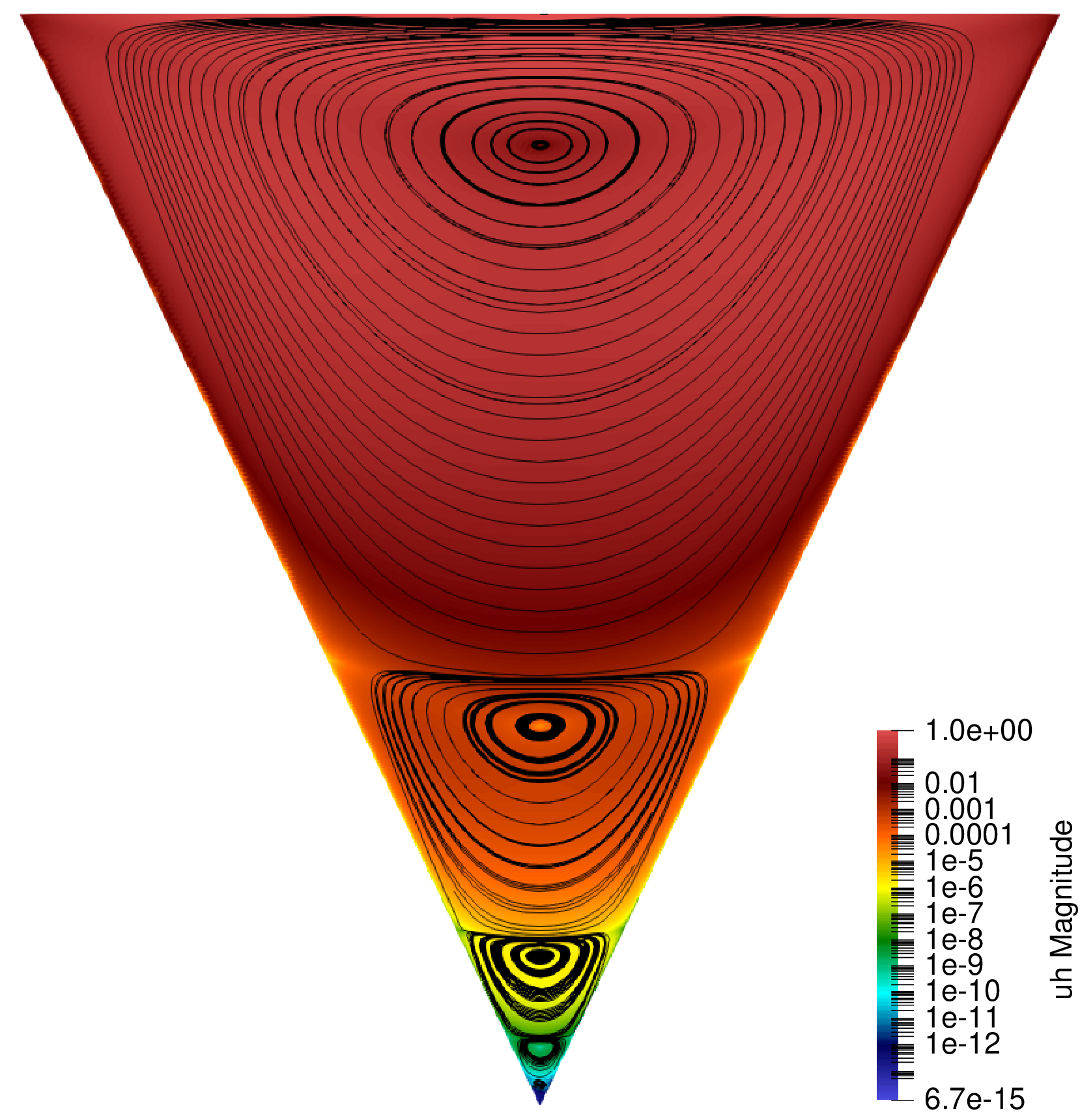} };
    \node at (-1.7,-1.7) {\includegraphics[width=0.12\textwidth]{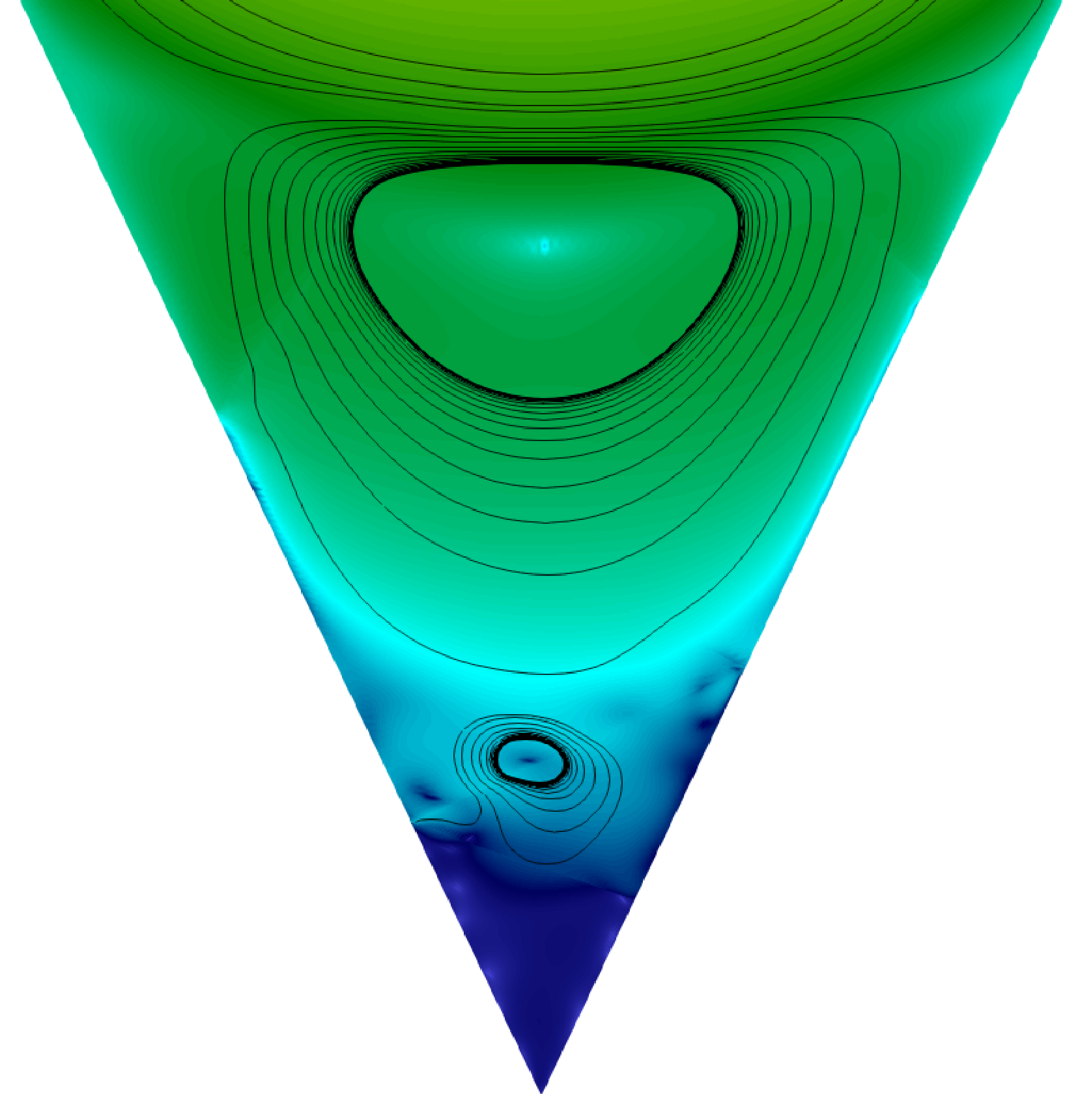} };
    \end{scope}
    \begin{scope}[xshift=0.5\textwidth]
    \node {\includegraphics[width=0.45\textwidth]{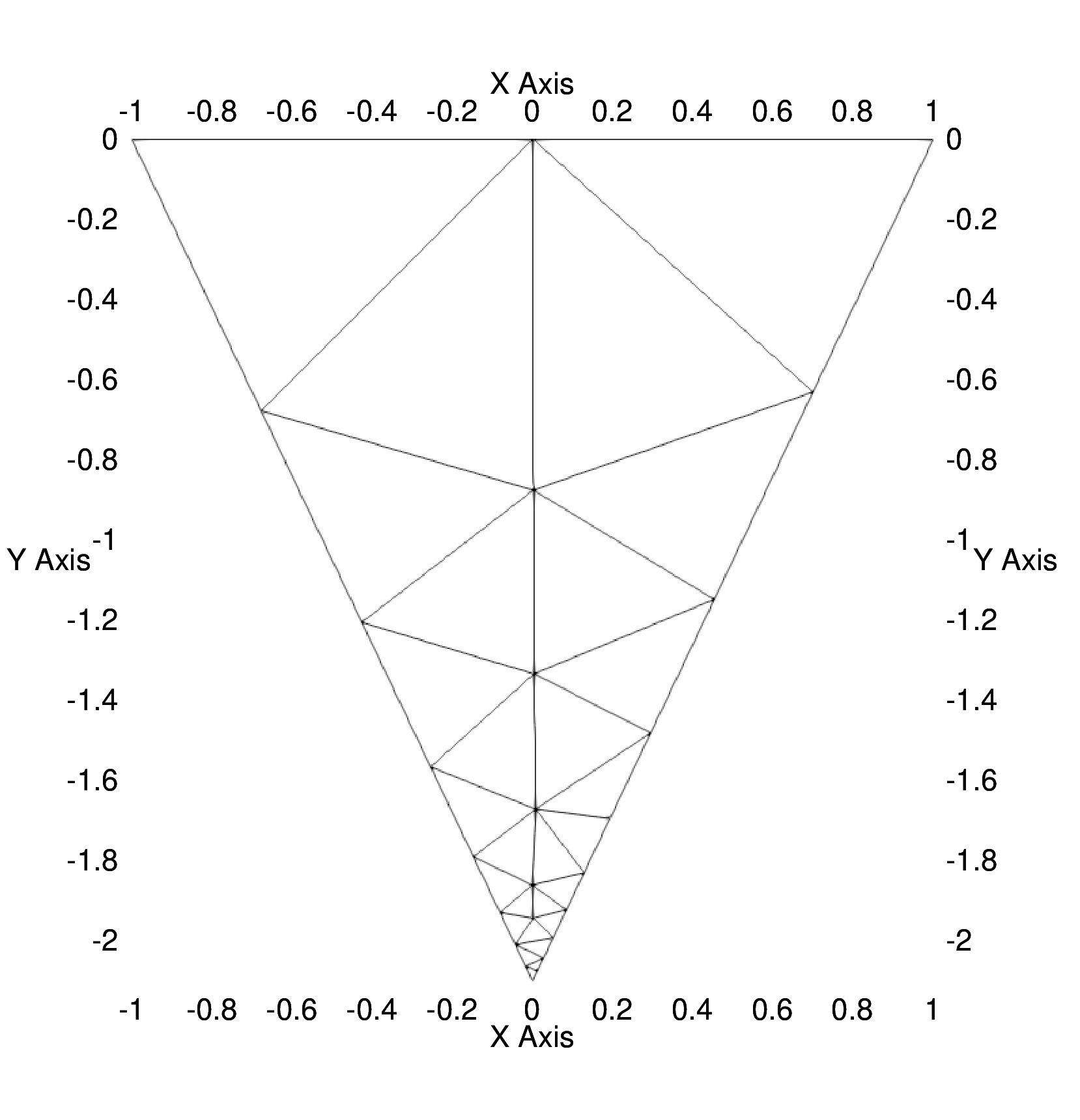} };
    \end{scope}
    \end{tikzpicture}
    \vspace*{-0.5cm}
    \caption{Moffatt eddies with $k=10$ on the left, including a zoom on the bottom eddies on the bottom left. On the right we show the computational mesh.}\label{fig:moffatt}
\end{figure}

\subsection{Comparison of computational effort to different Stokes discretizations} \label{sec:numbercrunching}

In this section, we want to compare the Trefftz-DG method with other popular discretizations for the Stokes problem with respect to the number of unknowns and their sparsity patterns that finally determine the computational costs to a large extent.
For a comparison, we only consider simplicial meshes and mostly restrict to different variants of DG methods.
We focus on a comparison of the total number of degrees of freedom (\texttt{ndof}), the \texttt{ndof} that remains after elimination of all interior unknowns (i.e. after static condensation) denoting them as \emph{coupling} \texttt{ndof} (\texttt{ncdof}), and the resulting sparsity pattern of the methods (after static condensation) in terms of the non-zero entries in the resulting sparse system matrices (\texttt{nnze}).

We include several hybrid Discontinuous Galerkin (HDG) methods, where the velocity couplings across element interfaces stemming from DG-terms are avoided by introducing additional facet unknowns for the velocities.
As an exception, we also include the Taylor-Hood method (of arbitrarily high order) which has continuous velocity and pressure approximations. All methods share the same order of convergence in a discrete $H^1$-type norm for the velocity and an $L^2$-type norm for the pressure. 
The list of methods presented here is by no means complete, for additional material on HDG methods we refer to e.g. \cite{zbMATH05837531,COCKBURN2014221,L_MTH_2010}, and for further exploration, we refer to the relevant literature, including the books \cite{di2011mathematical,brezzi,zbMATH06625569} and the references therein.

First, we consider discontinuous Galerkin methods, with basis functions that are discontinuous across inter-element boundaries, including the method presented in this work.  
The methods can be considered the closest relatives to the Trefftz-DG method as the inter-element regularity of the solution is solely enforced in the bilinear form. 
This makes them very flexible in the mesh choice, allowing easy for e.g. polygonal meshes.
We consider the following:
\begin{itemize}
  \item \underline{Standard DG}: \\[0.4ex] As a representative of a ``standard'' Discontinuous Galerkin method we consider the interior penalty discretization from 
  \cref{sec:underlyingDG}, cf. especially \eqref{eq:dgstokes}, respectively from \cite[Section 6.1.5]{di2011mathematical} with $u_h \in [\PP^k]^d$ and $p_h \in \PP^{k-1}$, cf. \cref{sec:underlyingDG}.
  \item \underline{Trefftz-DG}: \\ The Trefftz-DG method, presented in this work, uses the same variational formulation as the standard DG method, but with a different finite element space so that $(u_h,p_h) \in \TT^k$ with $\TT^k$ as in \eqref{eq:Trefftzconditions}. 
  The method only has unknowns on the elements and static condensation cannot be applied.
  \item \underline{Solenoidal DG}: \\[0.4ex] In \cite{montlaur08,montlaur2010discontinuous} the velocity space is reduced to element-wise divergence-free polynomial functions. For the enforcement of normal-continuity, a facet variable for the pressure is then introduced leading to exactly divergence-free discrete solutions. Velocities $u_h \in \{ v \in [\PP^k(\Th)]^d | \Div v|_T = 0, T \in \Th \} $ and pressures $p_h \in \PP^k(\Fh)$ are assumed to appear in the global linear system, i.e. no static condensation is applied. Volume pressure unknowns do not appear in the global linear system (neither as \texttt{ndof} nor \texttt{ncdof}), but can be reconstructed in element-by-element post-processing.
  \item \underline{Rhebergen-Wells-HDG}: \\[0.4ex] In \cite{RhebergenWells} Rhebergen and Wells introduce $d+1$ scalar facet variables of degree $k$ to enforce normal continuity and the weak continuity of the velocity vector. All element interior unknowns for velocity and pressure can then be eliminated by static condensation.
\end{itemize}
Next, we will consider $H(\Div)$-conforming methods, i.e. methods based on a discrete velocity space that is normal-continuous across element interfaces.
We recall that $\BDM^{k}= [\PP^k]^d \cap H(\Div;\Omega)$ is the Brezzi-Douglas-Marini space, see e.g. \cite[Section 2.3.1]{brezzi}.
\begin{itemize}
  \item \underline{$H(\Div)$-DG}: \\[0.4ex] By considering discrete velocities from $H(\Div)$-conforming finite element spaces the pressure facet variable can be avoided while still obtaining exactly solenoidal solutions (after computation). This has been proposed in \cite{cockburn2007note} with $u_h \in \BDM^k$ and $p_h \in \PP^{k-1}$. 
  \item \underline{High order divergence-free (\texttt{hodf}) $H(\Div)$-DG}: \\[0.4ex] Exploiting the a-priori knowledge that velocity solutions in $H(\Div)$-DG are exactly divergence-free allows to reduce the velocity basis functions a-priorily to those with a piecewise constant divergence, $u_h \in \{v \in [\PP^k]^d \mid \Div v \in \PP^0 \}$, cf. \cite{zaglmayr2006high}. Note that the basis functions in $\BDM^k$ generating a divergence in $\PP^0$ are not local and remain in the system. As pressure couplings across element boundaries vanish in an $H(\Div)$-conforming setting we can correspondingly remove the pressure variables of higher order so that w.r.t. the pressure only $p_h \in \PP^0$ remains for the solution of the global linear system. Higher order pressures can again be reconstructed element-by-element in a simple post-processing and are not considered in this comparison. Compared to the previous $H(\Div)$-DG method this step can be considered as a version of static condensation. 
  \item \underline{$H(\Div)$-HDG}: \\[0.4ex] In the Rhebergen-Wells-HDG method the normal component is made continuous essentially twice, through a pressure facet variable and as one component of the weak continuity stemming from the viscosity term. This redundancy can be removed. In~\cite{LS_CMAME_2016} an $H(\Div)$-conforming space for the velocity is considered, $u_h \in \BDM^k$ with $p_h \in \PP^{k-1}$ and a tangential vector function on the facets $\hat{u} \in \PP_\tau^k(\Fh) := \{ v \in L^2(\Fh)  \mid v|_F \in \PP^k(F), v \cdot n_F |_F = 0 \}$ is introduced to avoid direct couplings between neighboring elements.
  Let $\BDM^\ell = \BDM_{\circ}^\ell \oplus \BDM_{\partial \Th}^\ell$ be the decomposition of the $\BDM^\ell$ space into \emph{interior normal-bubbles} $\BDM_{\circ}^\ell = \{v \in \BDM^\ell \mid v \cdot n|_{\partial \Th} = 0\}$ and \emph{interface} functions $\BDM_{\partial \Th}^\ell$ with $\Div \BDM_{\partial \Th}^\ell = \PP^0$, cf. \cite{zaglmayr2006high}.
  Then, the interior normal-bubbles (which determine the higher-order divergence of the velocity) can be statically condensed alongside the higher-order pressures. Only the facet variables (normal \texttt{dof} of the $\BDM_{\partial \Th}^k$ space and the tangential vector functions in $\PP_\tau^k(\Fh)$) and the piecewise constant pressure functions in $\PP^0(\Th)$ remain in the global linear system.
  \item \underline{Projected jumps modification (\texttt{pj}) of the $H(\Div)$-HDG method}: \\[0.4ex] An improvement of the $H(\Div)$-HDG method is obtained when reducing the polynomial degree of the tangential facet variables by one degree using a \emph{projected jumps} modification of the hybrid interior penalty formulation, cf. \cite{LS_CMAME_2016}.
  \item \underline{Highest order discontinuous facet modification (\texttt{hodc}) of the $H(\Div)$-HDG method}: \\[0.4ex] A second improvement of the $H(\Div)$-HDG method is obtained when normal continuity is relaxed by one degree leading to a special velocity space $\BDM^{\star,k} = \{ v \in [\PP^k]^d \mid (\jump{v \cdot n},w)_F = 0, \forall w \in \PP^{k-1}(F), F \in \Fh\}$. Then, the facet dofs of $\BDM^{\star,k}$ coincide with those of $\BDM_{\partial \Th}^{k-1}$ and together with $\PP_\tau^{k-1}(\Fh)$ they suffice to couple velocity functions on neighboring elements, cf. \cite{LLS_SIAM_2017,LLS_ESAIM_2019}. 
\end{itemize}    
Finally, we also include an $H^1$-conforming method, here we consider:
\begin{itemize}
    \item \underline{Taylor-Hood}: \\[0.4ex] The Taylor-Hood method is an $H^1$-conforming Galerkin method with $u_h \in [\PP^k \cap C^0(\Omega)]^d$ and $p_h \in \PP^{k-1} \cap C^0(\Omega)$ using bilinear and linear form as in the continuous weak formulation. 
    By $\PP^\ell = \PP_{\circ}^\ell \oplus \PP_{\partial \Th}^\ell$ we denote the decomposition of the polynomial space $\PP^\ell$ into \emph{interior bubbles} $\PP_{\circ}^\ell = \{v \in \PP^\ell \mid v|_{\partial \Th} = 0\}$ and \emph{interface} polynomials $\PP_{\partial \Th}^\ell$.
    For the Taylor-Hood element, the interior bubbles can be eliminated by static condensation and only the interface polynomials remain in the reduced linear system. 
\end{itemize}    

In \cref{tab:comparison.methods} we summarized the considered methods in a table with an emphasis on the used discretization spaces and those that remain after static condensation. 
In \cref{fig:drawdofs} we illustrate the different degrees of freedom for some of the methods in two dimensions and $k=4$.

\begin{table}[!h]
\begin{tabular}{r c c c c}
\toprule
method & \multicolumn{4}{c@{~}}{discretization space}  \\ 
\cmidrule[0.4pt](r{0.125em}){1-1}%
\cmidrule[0.4pt](lr{0.25em}){2-5}%
& \multicolumn{2}{c@{~}@{~}}{velocity}  
& \multicolumn{2}{c@{~}}{pressure}  \\ 
\cmidrule[0.4pt](r{0.125em}){1-1}%
\cmidrule[0.4pt](lr{0.125em}){2-3}%
\cmidrule[0.4pt](lr{0.25em}){4-5}%
& element & facet & element & facet \\ 
\cmidrule[0.4pt](r{0.125em}){1-1}%
\cmidrule[0.4pt](lr{0.125em}){2-2}%
\cmidrule[0.4pt](lr{0.125em}){3-3}%
\cmidrule[0.4pt](lr{0.125em}){4-4}%
\cmidrule[0.4pt](lr{0.25em}){5-5}%
\multicolumn{5}{l@{~}}{ \textbf{$L^2$-conforming} } \\ 
\cmidrule[0.4pt](r{0.125em}){1-1}%
Standard DG & $[\PP^k]^d$ & --- & $\PP^{k-1}$ & --- \\ 
Solenoidal DG & $[\PP^k]^d \!\cap\! \{\!\, \Div v\! =\! 0 \!\,\}$ & --- & --- & $\PP^{k}(\Fh)$ \\ 
Rhebergen-Wells-HDG & $[\PP^k]^d$ & $[\PP^k(\Fh)]^d$ & $\PP^{k-1}$ & $\PP^k(\Fh)$ \\
(reduced) &  --- & $[\PP^k(\Fh)]^d$ & --- & $\PP^k(\Fh)$ \\ 
Trefftz-DG & $\TT^k$ & --- & $\TT^k$ & --- \\
\midrule
\multicolumn{5}{l@{~}}{ \textbf{$H(\Div)$-conforming} } \\ 
\cmidrule[0.4pt](r{0.125em}){1-1}%
$H(\Div)$-DG\hphantom{(\texttt{hodf})} & $\BDM^k$ & --- & $\PP^{k-1}$ & --- \\
$H(\Div)$-DG(\texttt{hodf}) & $\BDM^k \!\!\cap\! \{\! \Div v \!\in\! \PP^0 \!\}$ & --- & $\PP^{0}$ & --- \\ 
$H(\Div)$-HDG\hphantom{ (\texttt{hodc})} & $\BDM^k$ & $\PP_\tau^k(\Fh)$ & $\PP^{k-1}$ & --- \\ 
(reduced) & $\BDM_{\partial \Th}^k$ & $\PP_\tau^k(\Fh)$ & $\PP^{0}$ & --- \\[0.4ex] 
$H(\Div)$-HDG (\texttt{pj})\hphantom{\texttt{dc}} & $\BDM^k$ & $\PP_\tau^{k-1}(\Fh)$ & $\PP^{k-1}$ & --- \\ (reduced) & $\BDM_{\partial \Th}^k$ & $\PP_\tau^{k-1}(\Fh)$ & $\PP^{0}$ & --- \\[0.4ex] 
$H(\Div)$-HDG (\texttt{hodc}) & ${\BDM}^{\star,k}$ & $\PP_\tau^{k-1}(\Fh)$ & $\PP^{k-1}$ & --- \\ (reduced) & ${\BDM}_{\partial \Th}^{k-1}$ & $\PP_\tau^{k-1}(\Fh)$ & $\PP^{0}$ & --- \\[0.4ex] 
\midrule
\multicolumn{5}{l@{~}}{ \textbf{$H^1$-conforming} } \\ 
\cmidrule[0.4pt](r{0.125em}){1-1}%
Taylor-Hood & \multicolumn{2}{c@{~}@{~}}{$[\PP^k \cap C^0]^d$} & \multicolumn{2}{c@{~}}{$\PP^{k-1} \cap C^0$} \\ 
(reduced) & \multicolumn{2}{c@{~}@{~}}{$[\PP_{\partial \Th}^k \cap C^0]^d$} & \multicolumn{2}{c@{~}}{$\PP_{\partial \Th}^{k-1} \cap C^0$} \\
\bottomrule
\end{tabular}
\vspace{1em}
\caption{Overview of the considered methods with respect to the used discretization spaces; before and after static condensation (if applicable).}
\label{tab:comparison.methods}
\vspace{-2em}
\end{table}

\input{draw_dofs}

For the experiment, we consider two fixed domains, the unit square $(0,1)^2$ and the unit cube $(0,1)^3$, with fixed meshes with 34 elements and 59 facets in 2D and 492 elements and 1087 facets in 3D and only vary the polynomial degree~$k$. 

\begin{figure}[ht]
  \begin{center}    
  {
      \begin{tikzpicture}[scale=0.75,spy using outlines={circle, magnification=4, size=2cm, connect spies}]
      \begin{groupplot}[%
        group style={%
          group name={my plots},
          group size=3 by 2,
          vertical sep=0.5cm,
          horizontal sep=0.7cm,
        },
        title style={at={(0.01,0.94)},
        anchor=north west,
        draw=none,fill=white},
        every axis x label/.style={
          scale=0.8,
        at={(0.95,-0.05)},
      anchor=west,
      },
      legend style={
          legend columns=4,
          at={(-0.6,-0.1)},
          anchor=north,
          draw=none
      },
      xlabel={$k$},
      ymajorgrids=true,
      grid style=dashed,
      cycle list name=christophcolors,
      width=6.6cm,
      ]     
      \nextgroupplot[ymode=log,xmode=linear, title={\texttt{ndofs} (2D)}]
      \node [above left] at (rel axis cs:1.05,-0.05) {\includegraphics[width=2cm]{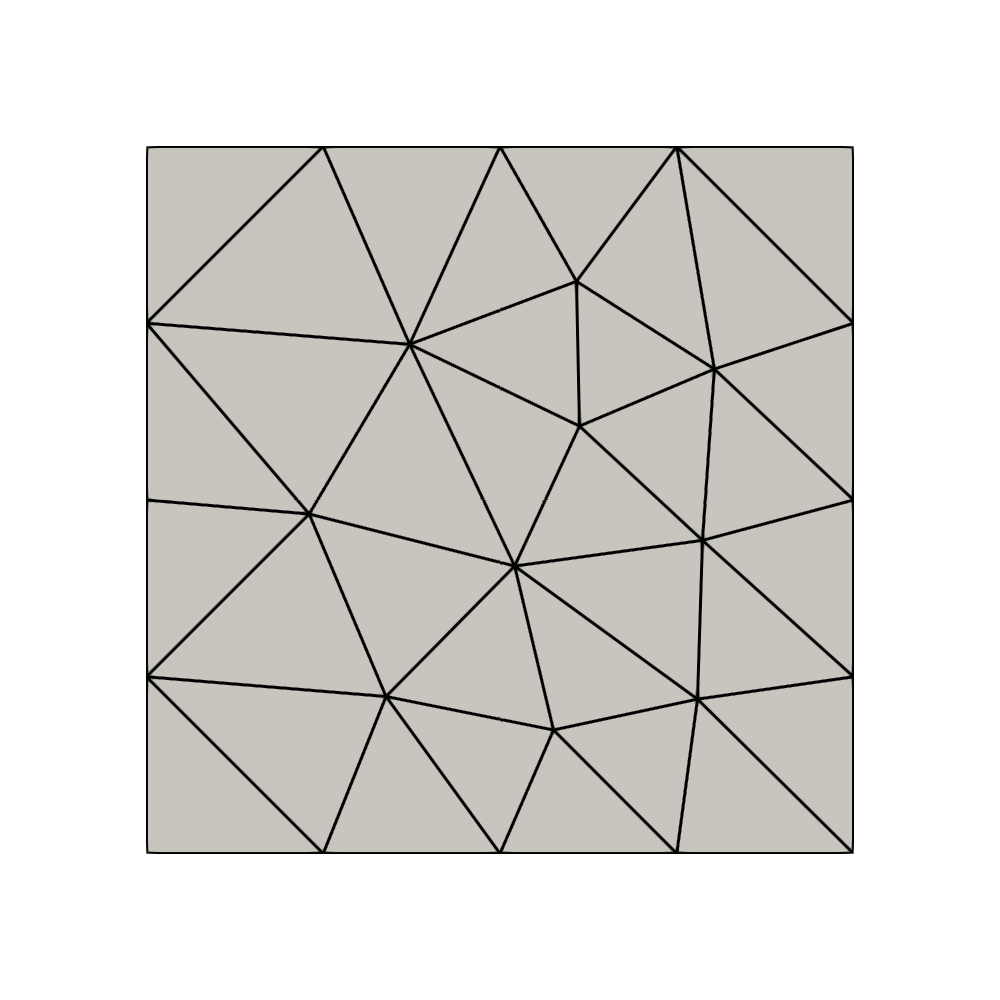}};      
          \foreach \k in {CGTH, DG, HdivHDGpjmp, TDG}{
              \addplot+[discard if not={method}{\k}] table [x=order, y=ndofs, col sep=comma] {ex/ex_compare_2D.csv};
              }

      \nextgroupplot[ymode=log,xmode=linear, title={\texttt{ncdofs} (2D)}]
      \node [above left] at (rel axis cs:1.05,-0.05) {\includegraphics[width=2cm]{ex/mesh2D.png}};      
      \foreach \k in {CGTH, DG, HdivHDGpjmp, TDG}{
              \addplot+[discard if not={method}{\k}] table [x=order, y=ncdofs, col sep=comma] {ex/ex_compare_2D.csv};
          }
      \nextgroupplot[ymode=log,xmode=linear, title={\texttt{nnzes} (2D)}]
      \node [above left] at (rel axis cs:1.05,-0.05) {\includegraphics[width=2cm]{ex/mesh2D.png}};      
      \foreach \k in {CGTH, DG, HdivHDGpjmp, TDG}{
          \addplot+[discard if not={method}{\k}] table [x=order, y=nnzes, col sep=comma] {ex/ex_compare_2D.csv};
      }
      \nextgroupplot[ymode=log,xmode=linear, title={\texttt{ndofs} (3D)}]
      \node [above left] at (rel axis cs:1.05,-0.05) {\includegraphics[width=2cm]{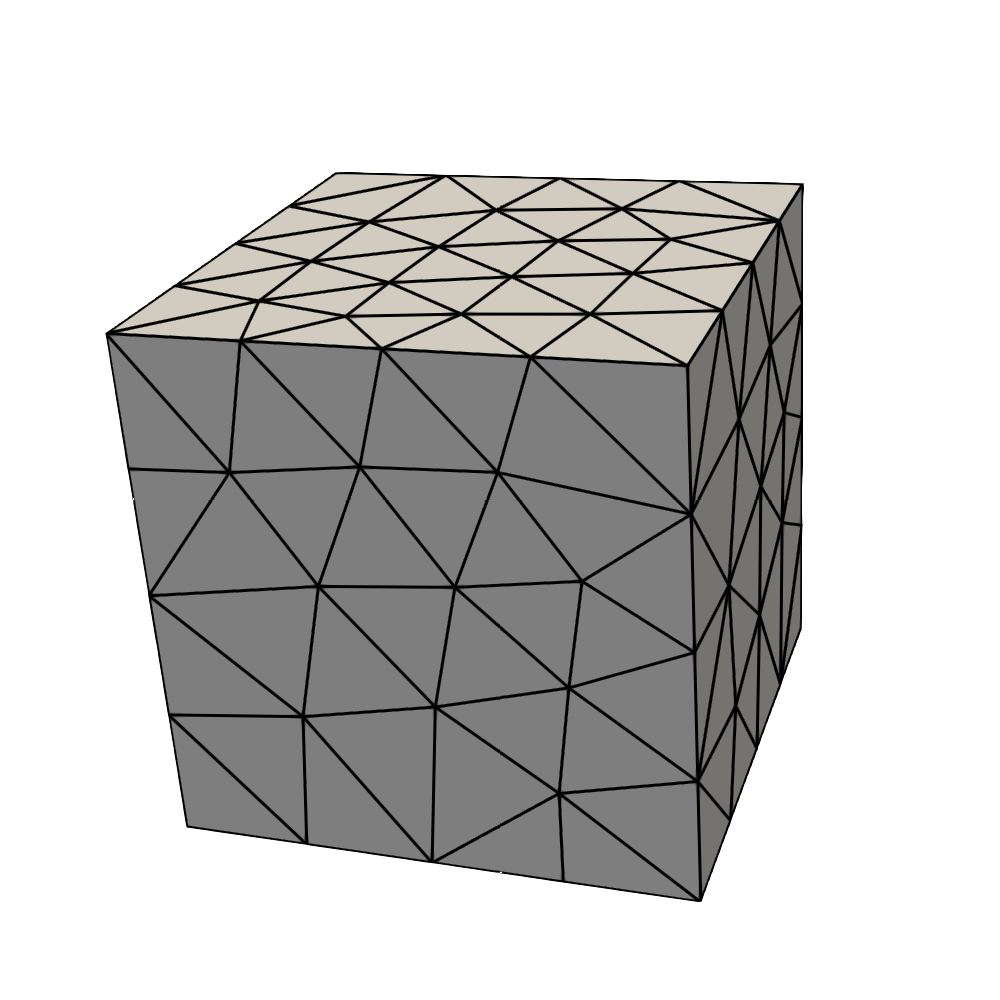}};      
          \foreach \k in {CGTH, DG, HdivHDGpjmp, TDG}{
              \addplot+[discard if not={method}{\k}] table [x=order, y=ndofs, col sep=comma] {ex/ex_compare_3D.csv};
          }
      \nextgroupplot[ymode=log,xmode=linear, title={\texttt{ncdofs} (3D)}]
      \node [above left] at (rel axis cs:1.05,-0.05) {\includegraphics[width=2cm]{ex/mesh3D.png}};      
      \foreach \k in {CGTH, DG, HdivHDGpjmp, TDG}{
              \addplot+[discard if not={method}{\k}] table [x=order, y=ncdofs, col sep=comma] {ex/ex_compare_3D.csv};
          }
      \nextgroupplot[ymode=log,xmode=linear, title={\texttt{nnzes} (3D)}]
      \node [above left] at (rel axis cs:1.05,-0.05) {\includegraphics[width=2cm]{ex/mesh3D.png}};      
      \foreach \k in {CGTH, DG, HdivHDGpjmp, TDG}{
              \addplot+[discard if not={method}{\k}] table [x=order, y=nnzes, col sep=comma] {ex/ex_compare_3D.csv};
          }
          \addlegendimage{solid}
          \legend{Taylor-Hood, 
          Standard DG~~, 
          $H(\mathrm{div})$-HDG (\texttt{hodc})~~, 
          Trefftz-DG~~}
          \end{groupplot}
  \end{tikzpicture}}
  \end{center} 
  \caption{Comparison of \texttt{ndof}, \texttt{ncdof} and \texttt{nnze} for a selection of methods in \cref{sec:numbercrunching} for the 2D and 3D Stokes problem on the displayed mesh.}
      \label{fig:stokes:dofscompare}
  \end{figure}

For a high level overview we first focus only on results for four methods, the Taylor-Hood method, the standard DG method, the $H(\Div)$-DG method (with \texttt{hodc}) and the Trefftz-DG method in terms of plots in \cref{fig:stokes:dofscompare}.
In the appendix in \cref{app:moretables}, in \cref{tab:ndof2d,tab:ncdof2d,tab:nnze2d,tab:ndof3d,tab:ncdof3d,tab:nnze3d} we show the corresponding results for all methods mentioned above.
We observe that the Trefftz-DG method brings a significant improvement over the Standard DG method, especially in terms of non-zero entries in the matrix.
Compared to the Hybrid DG methods it is competitive in terms of \texttt{ndof} and \texttt{ncdof} and only slightly worse in terms of \texttt{nnze} when compared to the optimized HDG variants.
It even compares quite well to the popular Taylor-Hood method for higher orders.

\section{Conclusion \& Outlook}\label{sec13} In this paper we
introduced a new Stokes discretization based on local basis functions
that are \emph{Trefftz}, i.e. that fulfill the Stokes equations
pointwise (w.r.t. to an approximated r.h.s.). This leads to a strong
reduction of unknowns compared to standard DG methods. To construct
the corresponding basis we use the embedded Trefftz-DG method which allows us to deal with inhomogeneous forces and sources. The method
is analyzed and a priori error bounds are derived. To the best of our
knowledge, this is the first Trefftz-DG discretization for the Stokes
problem. Crucial and new components in the analysis are a splitting of
the pressure unknowns into higher order and lower order parts and the
analysis for inhomogeneous forces and sources. Especially the latter
will also be useful for the analysis of (embedded) Trefftz-DG methods
for other PDEs. 

A topic that we have not addressed in this work is pressure robustness. Although discrete solutions of the presented scheme will be pointwise divergence-free, discrete solutions will in general not be normal-continuous or \emph{pressure robust}. We leave improvements of the proposed scheme in that regard for future research.

Due to the generic construction of the Trefftz-DG space using the embedded Trefftz-DG approach, the presented approach allows to explore extensions to Oseen-type or non- and stationary Navier-Stokes problems where the Trefftz reduction of the underlying DG space is applied to each linear problem arising within potential linearization steps. Further investigations and analysis are left for future work.
\section{Acknowledgement}

PL acknowledges the support by the Austrian Science Fund (FWF) via the
standalone project “High-Fidelity and Efficient Direct Aeroacoustic
Simulations” P35931-N. CL and PS were funded by DFG SFB 1456 project 432680300.

\begin{appendices}

\section{Selected proof}
\begin{cor}\label{cor:lapsurj}
    The operator $\Delta : \calP^k \to \calP^{k-2}$ is surjective.
\end{cor}
\begin{proof}
    Considering $k<2$ we have $\calP^{k-2} = \{0\}$ and the result is trivial. Now take $k\geq 2$. 
    We can write $\calP^k$ as the following algebraic direct sum
    $$\calP^k = \calH^k \oplus |x|^2 \calP^{k-2}, \quad \text{for} \quad k\geq 2,$$
    where $\calH^k$ is the space of harmonic polynomials of degree $k$, i.e. $\calH^k=\ker(\Delta: \calP^k \to \calP^{k-2})$, see \cite[Proposition 5.5]{harmonicfcttheory}.
    Now consider the map $F : \calP^{k-2} \to \calP^{k-2}$ given by $F(p) = \Delta(|x|^2 p)$. 
    Since $|x|^2 p\in \calP^{k}$ and $F$ has a trivial kernel by the above decomposition, the result follows.
\end{proof}


\section{Further results of computational comparison in \cref{sec:numbercrunching}} \label{app:moretables}

\begin{table}[ht!]
   \begin{tabular}{rNNNNNNNNNN}
    \toprule
    $k$
    & \multicolumn{1}{r}{\rotatebox{90}{\begin{minipage}{2.05cm}Taylor-Hood \end{minipage}}}
    & \multicolumn{1}{r}{\rotatebox{90}{\begin{minipage}{2.05cm}Standard DG \\ \cite[Sect. 6.1.5]{di2011mathematical}\end{minipage}}}
    & \multicolumn{1}{r}{\rotatebox{90}{\begin{minipage}{2.05cm}Solenoidal DG \cite{montlaur08}\end{minipage}}}
    & \multicolumn{1}{r}{\rotatebox{90}{\begin{minipage}{2.05cm}$H(\text{div})$-{DG} \\ \cite{cockburn2007note}\end{minipage}}}
    & \multicolumn{1}{r}{\rotatebox{90}{\begin{minipage}{2.05cm}$H(\text{div})$-{DG} \\ \texttt{hodf}\end{minipage}}}
    & \multicolumn{1}{r}{\rotatebox{90}{\begin{minipage}{2.05cm}Rhebergen-Wells HDG \cite{RhebergenWells}\end{minipage}}}
    & \multicolumn{1}{r}{\rotatebox{90}{\begin{minipage}{2.05cm}$H(\text{div})$-{HDG} \\ \cite{LS_CMAME_2016} \end{minipage}}}
    & \multicolumn{1}{r}{\rotatebox{90}{\begin{minipage}{2.05cm}$H(\text{div})$-{HDG} \\ \texttt{pj} \cite{LS_CMAME_2016} \end{minipage}}}
    & \multicolumn{1}{r}{\rotatebox{90}{\begin{minipage}{2.05cm}$H(\text{div})$-{HDG} \\ \texttt{hodc} \cite{LS_CMAME_2016,LLS_SIAM_2017,LLS_ESAIM_2019} \end{minipage}}}
    & \multicolumn{1}{r}{\rotatebox{90}{\begin{minipage}{2.05cm}Trefftz-DG\end{minipage}}}
     \\
    \midrule
1&	 \multicolumn{1}{r}{---}&	 238&	 288&	 \ca{152}&	 \ca{152}&	592&	  270&	 211& 254&	 204\\
2&	 \ca{196}&	 510&	 483&	 381&	 245&	1041&  558&	 499&  542&	340\\
3&	 441&	 884&	 712&	 712&	 \ca{372}&	1592&  948&	 889&  932&	476\\
4&	 788&	1360&	 975&	1145&	 \ca{533}&	2245& 1440&	1381& 1424&	612\\
5&	1237&	1938&	1272&	1680&	 \ca{728}&	3000& 2034&	1975& 2018&	748\\
6&	1788&	2618&	1603&	2317&	 957&	3857& 2730&	2671& 2714&	\ca{ 884}\\
7&	2441&	3400&	1968&	3056&	1220&	4816&	3528&	3469& 3512&	\ca{1020}\\
8&	3196&	4284&	2367&	3897&	1517&	5877&	4428&	4369& 4412&	\ca{1156}\\
9&	4053&	5270&	2800&	4840&	1848&	7040&	5430&	5371& 5414&	\ca{1292}\\
10&	5012&	6358&	3267&	5885&	2213&	8305&	6534&	6475& 6518&	\ca{1428}\\
\bottomrule
\end{tabular}
\caption{Comparison of the total number of degrees of freedom (\texttt{ndof}) for the methods considered in \cref{sec:numbercrunching} in two dimensions (possibly without (higher order) pressure functions that may be obtained from post-processing).} \label{tab:ndof2d}
\end{table}

\begin{table}[ht!]
   \begin{tabular}{rNNNNNNNNNN}
    \toprule
    $k$
    & \multicolumn{1}{r}{\rotatebox{90}{\begin{minipage}{2.05cm}Taylor-Hood \end{minipage}}}
    & \multicolumn{1}{r}{\rotatebox{90}{\begin{minipage}{2.05cm}Standard DG \\ \cite[Sect. 6.1.5]{di2011mathematical}\end{minipage}}}
    & \multicolumn{1}{r}{\rotatebox{90}{\begin{minipage}{2.05cm}Solenoidal DG \cite{montlaur08}\end{minipage}}}
    & \multicolumn{1}{r}{\rotatebox{90}{\begin{minipage}{2.05cm}$H(\text{div})$-{DG} \\ \cite{cockburn2007note}\end{minipage}}}
    & \multicolumn{1}{r}{\rotatebox{90}{\begin{minipage}{2.05cm}$H(\text{div})$-{DG} \\ \texttt{hodf}\end{minipage}}}
    & \multicolumn{1}{r}{\rotatebox{90}{\begin{minipage}{2.05cm}Rhebergen-Wells HDG \cite{RhebergenWells}\end{minipage}}}
    & \multicolumn{1}{r}{\rotatebox{90}{\begin{minipage}{2.05cm}$H(\text{div})$-{HDG} \\ \cite{LS_CMAME_2016} \end{minipage}}}
    & \multicolumn{1}{r}{\rotatebox{90}{\begin{minipage}{2.05cm}$H(\text{div})$-{HDG} \\ \texttt{pj} \cite{LS_CMAME_2016} \end{minipage}}}
    & \multicolumn{1}{r}{\rotatebox{90}{\begin{minipage}{2.05cm}$H(\text{div})$-{HDG} \\ \texttt{hodc} \cite{LS_CMAME_2016,LLS_SIAM_2017,LLS_ESAIM_2019} \end{minipage}}}
    & \multicolumn{1}{r}{\rotatebox{90}{\begin{minipage}{2.05cm}Trefftz-DG\end{minipage}}}
     \\
    \midrule
1& \multicolumn{1}{c}{---} &238 &	288&	\ca{152}&	\ca{152}&	354&	270&	211&	168&	204\\
2&\ca{196}  &	510&	483&	381&	245&	531&	388&	329&	286&	340 \\
3&373 &	884&	712&	712&	\ca{372}&	708&	506&	447&	404&	476 \\
4&550 &	1360&	975&	1145&	533&	885&	624&	565&	\ca{522}&	612 \\
5& 727 &	1938&	1272&	1680&	728&	1062&	742&	683&	\ca{640}&	748 \\
6& 904 &	2618&	1603&	2317&	957&	1239&	860&	801&	\ca{758}&	884 \\
7&1081 &	3400&	1968&	3056&	1220&	1416&	978&	919&	\ca{876}&	1020 \\
8&1258 &	4284&	2367&	3897&	1517&	1593&	1096&	1037&	\ca{994}&	1156 \\
9&1435 &	5270&	2800&	4840&	1848&	1770&	1214&	1155&	\ca{1112}&	1292 \\
10&1612 &	6358&	3267&	5885&	2213&	1947&	1332&	1273&	\ca{1230}&	1428 \\
\bottomrule
\end{tabular}
\caption{Comparison of the number of coupling degrees of freedom (\texttt{ncdof}) for the methods considered in \cref{sec:numbercrunching} in two dimensions (possibly without (higher-order) pressure functions that may be obtained from post-processing).} \label{tab:ncdof2d}
\end{table}

\begin{table}[ht!]
   \begin{tabular}{r@{~~}N@{~~}N@{~~}N@{~~}N@{~~}N@{~~}N@{~~}N@{~~}N@{~~}N@{~~}N}
    \toprule
    $k$
    & \multicolumn{1}{r}{\rotatebox{90}{\begin{minipage}{2.05cm}Taylor-Hood \end{minipage}}}
    & \multicolumn{1}{r}{\rotatebox{90}{\begin{minipage}{2.05cm}Standard DG \\ \cite[Sect. 6.1.5]{di2011mathematical}\end{minipage}}}
    & \multicolumn{1}{r}{\rotatebox{90}{\begin{minipage}{2.05cm}Solenoidal DG \cite{montlaur08}\end{minipage}}}
    & \multicolumn{1}{r}{\rotatebox{90}{\begin{minipage}{2.05cm}$H(\text{div})$-{DG} \\ \cite{cockburn2007note}\end{minipage}}}
    & \multicolumn{1}{r}{\rotatebox{90}{\begin{minipage}{2.05cm}$H(\text{div})$-{DG} \\ \texttt{hodf}\end{minipage}}}
    & \multicolumn{1}{r}{\rotatebox{90}{\begin{minipage}{2.05cm}Rhebergen-Wells HDG \cite{RhebergenWells}\end{minipage}}}
    & \multicolumn{1}{r}{\rotatebox{90}{\begin{minipage}{2.05cm}$H(\text{div})$-{HDG} \\ \cite{LS_CMAME_2016} \end{minipage}}}
    & \multicolumn{1}{r}{\rotatebox{90}{\begin{minipage}{2.05cm}$H(\text{div})$-{HDG} \\ \texttt{pj} \cite{LS_CMAME_2016} \end{minipage}}}
    & \multicolumn{1}{r}{\rotatebox{90}{\begin{minipage}{2.05cm}$H(\text{div})$-{HDG} \\ \texttt{hodc} \cite{LS_CMAME_2016,LLS_SIAM_2017,LLS_ESAIM_2019} \end{minipage}}}
    & \multicolumn{1}{r}{\rotatebox{90}{\begin{minipage}{2.05cm}Trefftz-DG\end{minipage}}}
     \\
    \midrule
1&	\multicolumn{1}{c}{---} & 5760	&6092	&2836	&2836	&9468	&5024	&2979	&\ca{1706}	&4320 \\
2&	\ca{4844}	 & 25920	&17595	&13923	&7907	&21303	&10692	&7595	&5462	&12000 \\
3&	13624	 & 76800	&	39152&43088	&18388	&37872	&18464	&14315	&\ca{11322}	&23520 \\
4&	26612	 & 180000	&	74975&103675	&37363	&59175	&28340	&23139	&\ca{19286}	&38880 \\
5&	43808	 & 362880	&	129996&212724	&68636	&85212	&40320	&34067	&\ca{29354}	&58080 \\
6&	65212	 & 658560	&	209867&390971	&116731	&115983	&54404	&47099	&\ca{41526}	&81120 \\
7&	90824	 & 1105920	&	320960&662848	&186892	&151488	&70592	&62235	&\ca{55802}	&108000 \\
8&	120644 & 1749600	&	470367&1056483	&285083	&191727	&88884	&79475	&\ca{72182}	&138720 \\
9&	154672 & 2640000	&	665900&1603700	&417988	&236700	&109280	&98819	&\ca{90666}	&173280 \\
10&	192908 & 3833280	&	916091&2340019	&593011	&286407	&131780	&120267	&\ca{111254}	&211680 \\
\bottomrule
\end{tabular}
\caption{Comparison of the number of non-zero entries (\texttt{nnze}) for the methods considered in \cref{sec:numbercrunching} in two dimensions.} \label{tab:nnze2d}
\end{table}

\begin{table}[ht!]
   \begin{tabular}{rNNNNNNNNNN}
    \toprule
    $k$
    & \multicolumn{1}{r}{\rotatebox{90}{\begin{minipage}{2.05cm}Taylor-Hood \end{minipage}}}
    & \multicolumn{1}{r}{\rotatebox{90}{\begin{minipage}{2.05cm}Standard DG \\ \cite[Sect. 6.1.5]{di2011mathematical}\end{minipage}}}
    & \multicolumn{1}{r}{\rotatebox{90}{\begin{minipage}{2.05cm}Solenoidal DG \cite{montlaur08}\end{minipage}}}
    & \multicolumn{1}{r}{\rotatebox{90}{\begin{minipage}{2.05cm}$H(\text{div})$-{DG} \\ \cite{cockburn2007note}\end{minipage}}}
    & \multicolumn{1}{r}{\rotatebox{90}{\begin{minipage}{2.05cm}$H(\text{div})$-{DG} \\ \texttt{hodf}\end{minipage}}}
    & \multicolumn{1}{r}{\rotatebox{90}{\begin{minipage}{2.05cm}Rhebergen-Wells HDG \cite{RhebergenWells}\end{minipage}}}
    & \multicolumn{1}{r}{\rotatebox{90}{\begin{minipage}{2.05cm}$H(\text{div})$-{HDG} \\ \cite{LS_CMAME_2016} \end{minipage}}}
    & \multicolumn{1}{r}{\rotatebox{90}{\begin{minipage}{2.05cm}$H(\text{div})$-{HDG} \\ \texttt{pj} \cite{LS_CMAME_2016} \end{minipage}}}
    & \multicolumn{1}{r}{\rotatebox{90}{\begin{minipage}{2.05cm}$H(\text{div})$-{HDG} \\ \texttt{hodc} \cite{LS_CMAME_2016,LLS_SIAM_2017,LLS_ESAIM_2019} \end{minipage}}}
    & \multicolumn{1}{r}{\rotatebox{90}{\begin{minipage}{2.05cm}Trefftz-DG\end{minipage}}}
     \\
    \midrule
1 &	\multicolumn{1}{c}{---}&6396&8673&\ca{3753}&\ca{3753}&19440&10275&5927&7689&5904\\
2 &	\ca{2811}&16728&19314&11442&8490&42816&24486&17964&20607&13284\\
3 &	\ca{9036}&34440&35470&25630&16774&77920&47370&38674&42198&23616\\
4 &	\ca{21085}&61500&58125&48285&29589&126720&80895&70025&74430&36900\\
5 &	\ca{40926}&99876&88263&81375&47919&191184&127029&113985&119271&53136\\
6 &	\ca{70527}&151536&126868&126868&72748&273280&187740&172522&178689&72324\\
7 &	111856&218448&174924&186732&105060&374976&264996&247604&254652&\ca{94464}\\
8 &	166881&302580&233415&262935&145839&498240&360765&341199&349128&\ca{119556}\\
9 &	237570&405900&303325&357445&196069&645040&477015&455275&464085&\ca{147600}\\
10 &	325891&530376&385638&472230&256734&817344&615714&591800&601491&\ca{178596}\\
\bottomrule
\end{tabular}
\caption{Comparison of total number of degrees of freedom (\texttt{ndof}) for the methods considered in \cref{sec:numbercrunching} in three dimensions (possibly without (higher order) pressure functions that may be obtained from post-processing).} \label{tab:ndof3d}
\end{table}

\begin{table}[ht!]
   \begin{tabular}{rNNNNNNNNNN}
    \toprule
    $k$
    & \multicolumn{1}{r}{\rotatebox{90}{\begin{minipage}{2.05cm}Taylor-Hood \end{minipage}}}
    & \multicolumn{1}{r}{\rotatebox{90}{\begin{minipage}{2.05cm}Standard DG \\ \cite[Sect. 6.1.5]{di2011mathematical}\end{minipage}}}
    & \multicolumn{1}{r}{\rotatebox{90}{\begin{minipage}{2.05cm}Solenoidal DG \cite{montlaur08}\end{minipage}}}
    & \multicolumn{1}{r}{\rotatebox{90}{\begin{minipage}{2.05cm}$H(\text{div})$-{DG} \\ \cite{cockburn2007note}\end{minipage}}}
    & \multicolumn{1}{r}{\rotatebox{90}{\begin{minipage}{2.05cm}$H(\text{div})$-{DG} \\ \texttt{hodf}\end{minipage}}}
    & \multicolumn{1}{r}{\rotatebox{90}{\begin{minipage}{2.05cm}Rhebergen-Wells HDG \cite{RhebergenWells}\end{minipage}}}
    & \multicolumn{1}{r}{\rotatebox{90}{\begin{minipage}{2.05cm}$H(\text{div})$-{HDG} \\ \cite{LS_CMAME_2016} \end{minipage}}}
    & \multicolumn{1}{r}{\rotatebox{90}{\begin{minipage}{2.05cm}$H(\text{div})$-{HDG} \\ \texttt{pj} \cite{LS_CMAME_2016} \end{minipage}}}
    & \multicolumn{1}{r}{\rotatebox{90}{\begin{minipage}{2.05cm}$H(\text{div})$-{HDG} \\ \texttt{hodc} \cite{LS_CMAME_2016,LLS_SIAM_2017,LLS_ESAIM_2019} \end{minipage}}}
    & \multicolumn{1}{r}{\rotatebox{90}{\begin{minipage}{2.05cm}Trefftz-DG\end{minipage}}}
     \\
    \midrule
1 &	\multicolumn{1}{c}{---}&6396&8673&\ca{3753}&\ca{3753}&13044&10275&5927&4165&5904 \\
2 &	\ca{2811}&16728&19314&11442&8490&26088&20058&13536&10893&13284 \\
3 &	\ca{9036}&34440&35470&25630&16774&43480&33102&24406&20882&23616 \\
4 &	\ca{19609}&61500&58125&48285&29589&65220&49407&38537&34132&36900 \\
5 &	\ca{34530}&99876&88263&81375&47919&91308&68973&55929&50643&53136 \\
6 &	\ca{53799}&151536&126868&126868&72748&121744&91800&76582&70415&72324 \\
7 &	\ca{77416}&218448&174924&186732&105060&156528&117888&100496&93448&94464 \\
8 &	\ca{105381}&302580&233415&262935&145839&195660&147237&127671&119742&119556 \\
9 &	\ca{137694}&405900&303325&357445&196069&239140&179847&158107&149297&147600 \\
10 &	\ca{174355}&530376&385638&472230&256734&286968&215718&191804&182113&178596 \\
\bottomrule
\end{tabular}
\caption{Comparison of number of coupling degrees of freedom (\texttt{ncdof}) for the methods considered in \cref{sec:numbercrunching} in three dimensions (possibly without (higher order) pressure functions that may be obtained from post-processing).} \label{tab:ncdof3d}
\end{table}

\makeatletter \@rot@twosidetrue \makeatother
\begin{sidewaystable}[ht!]
   \begin{tabular}{r@{~~}N@{~K~~}N@{~K~~}N@{~K~~}N@{~K~~}N@{~K~~}N@{~K~~}N@{~K~~}N@{~K~~}N@{~K~~}N@{~K~~}}
    \toprule
    $k$
    & \multicolumn{1}{r}{\rotatebox{90}{\begin{minipage}{2.05cm}Taylor-Hood \end{minipage}}}
    & \multicolumn{1}{r}{\rotatebox{90}{\begin{minipage}{2.05cm}Standard DG \\ \cite[Sect. 6.1.5]{di2011mathematical}\end{minipage}}}
    & \multicolumn{1}{r}{\rotatebox{90}{\begin{minipage}{2.05cm}Solenoidal DG \cite{montlaur08}\end{minipage}}}
    & \multicolumn{1}{r}{\rotatebox{90}{\begin{minipage}{2.05cm}$H(\text{div})$-{DG} \\ \cite{cockburn2007note}\end{minipage}}}
    & \multicolumn{1}{r}{\rotatebox{90}{\begin{minipage}{2.05cm}$H(\text{div})$-{DG} \\ \texttt{hodf}\end{minipage}}}
    & \multicolumn{1}{r}{\rotatebox{90}{\begin{minipage}{2.05cm}Rhebergen-Wells HDG \cite{RhebergenWells}\end{minipage}}}
    & \multicolumn{1}{r}{\rotatebox{90}{\begin{minipage}{2.05cm}$H(\text{div})$-{HDG} \\ \cite{LS_CMAME_2016} \end{minipage}}}
    & \multicolumn{1}{r}{\rotatebox{90}{\begin{minipage}{2.05cm}$H(\text{div})$-{HDG} \\ \texttt{pj} \cite{LS_CMAME_2016} \end{minipage}}}
    & \multicolumn{1}{r}{\rotatebox{90}{\begin{minipage}{2.05cm}$H(\text{div})$-{HDG} \\ \texttt{hodc} \cite{LS_CMAME_2016,LLS_SIAM_2017,LLS_ESAIM_2019} \end{minipage}}}
    & \multicolumn{1}{r}{\rotatebox{90}{\begin{minipage}{2.05cm}Trefftz-DG\end{minipage}}}
     \\
    \midrule
1&\multicolumn{1}{c}{---}&379& 466&213&213&1007&602&194&\caK{87}&325\\
2&\caK{230}&2570& 2389&1527&1114&4027&2336&1054&650&1643\\
3&\caK{1344}&10819& 8302&6632&4157&11186&6410&3470&2461&5193\\
4&\caK{4698}&34317& 22877&21459&12302&25168&14334&8702&6666&12679\\
5&\caK{12248}&90124& 53947&57152&30870&49328&27995&18384&14791&26291\\
6&\caK{26581}&206755& 113671&132483&68555&87695&49659&34531&28739&48707\\
7&50914&428440& 219869&276739&138614&144965&81968&59534&\caK{50788}&83091\\
8&89091&820062& 397546&533075&260242&226508&127942&96160&\caK{83596}&133096\\
9&145586&1472764& 680592&962328&460123&338364&190979&147556&\caK{130199}&202860\\
10&225502&2510244& 1113667&1647295&774162&487245&274854&217246&\caK{194009}&297007\\
\bottomrule
\end{tabular}
\caption{Comparison of number of non-zero entries (\texttt{nnze}) for the methods considered in \cref{sec:numbercrunching} in three dimensions.} \label{tab:nnze3d}
\end{sidewaystable}




\end{appendices}

\clearpage
\bibliography{literature}


\begin{thebibliography}{50}
\ifx \bisbn   \undefined \def \bisbn  #1{ISBN #1}\fi
\ifx \binits  \undefined \def \binits#1{#1}\fi
\ifx \bauthor  \undefined \def \bauthor#1{#1}\fi
\ifx \batitle  \undefined \def \batitle#1{#1}\fi
\ifx \bjtitle  \undefined \def \bjtitle#1{#1}\fi
\ifx \bvolume  \undefined \def \bvolume#1{\textbf{#1}}\fi
\ifx \byear  \undefined \def \byear#1{#1}\fi
\ifx \bissue  \undefined \def \bissue#1{#1}\fi
\ifx \bfpage  \undefined \def \bfpage#1{#1}\fi
\ifx \blpage  \undefined \def \blpage #1{#1}\fi
\ifx \burl  \undefined \def \burl#1{\textsf{#1}}\fi
\ifx \doiurl  \undefined \def \doiurl#1{\url{https://doi.org/#1}}\fi
\ifx \betal  \undefined \def \betal{\textit{et al.}}\fi
\ifx \binstitute  \undefined \def \binstitute#1{#1}\fi
\ifx \binstitutionaled  \undefined \def \binstitutionaled#1{#1}\fi
\ifx \bctitle  \undefined \def \bctitle#1{#1}\fi
\ifx \beditor  \undefined \def \beditor#1{#1}\fi
\ifx \bpublisher  \undefined \def \bpublisher#1{#1}\fi
\ifx \bbtitle  \undefined \def \bbtitle#1{#1}\fi
\ifx \bedition  \undefined \def \bedition#1{#1}\fi
\ifx \bseriesno  \undefined \def \bseriesno#1{#1}\fi
\ifx \blocation  \undefined \def \blocation#1{#1}\fi
\ifx \bsertitle  \undefined \def \bsertitle#1{#1}\fi
\ifx \bsnm \undefined \def \bsnm#1{#1}\fi
\ifx \bsuffix \undefined \def \bsuffix#1{#1}\fi
\ifx \bparticle \undefined \def \bparticle#1{#1}\fi
\ifx \barticle \undefined \def \barticle#1{#1}\fi
\bibcommenthead
\ifx \bconfdate \undefined \def \bconfdate #1{#1}\fi
\ifx \botherref \undefined \def \botherref #1{#1}\fi
\ifx \url \undefined \def \url#1{\textsf{#1}}\fi
\ifx \bchapter \undefined \def \bchapter#1{#1}\fi
\ifx \bbook \undefined \def \bbook#1{#1}\fi
\ifx \bcomment \undefined \def \bcomment#1{#1}\fi
\ifx \oauthor \undefined \def \oauthor#1{#1}\fi
\ifx \citeauthoryear \undefined \def \citeauthoryear#1{#1}\fi
\ifx \endbibitem  \undefined \def \endbibitem {}\fi
\ifx \bconflocation  \undefined \def \bconflocation#1{#1}\fi
\ifx \arxivurl  \undefined \def \arxivurl#1{\textsf{#1}}\fi
\csname PreBibitemsHook\endcsname

\bibitem[\protect\citeauthoryear{Trefftz}{1926}]{trefftz1926}
\begin{botherref}
\oauthor{\bsnm{Trefftz}, \binits{E.}}:
{Ein Gegenst{\"u}ck zum Ritzschen Verfahren}.
Proc. 2nd Int. Cong. Appl. Mech., Zurich, 1926,
131--137
(1926)
\end{botherref}
\endbibitem

\bibitem[\protect\citeauthoryear{Hiptmair et~al.}{2014}]{HMPS14}
\begin{barticle}
\bauthor{\bsnm{Hiptmair}, \binits{R.}},
\bauthor{\bsnm{Moiola}, \binits{A.}},
\bauthor{\bsnm{Perugia}, \binits{I.}},
\bauthor{\bsnm{Schwab}, \binits{C.}}:
\batitle{Approximation by harmonic polynomials in star-shaped domains and exponential convergence of {Trefftz} {$hp$-dGFEM}}.
\bjtitle{ESAIM Math. Model. Num. Anal.}
\bvolume{48},
\bfpage{727}--\blpage{752}
(\byear{2014})
\doiurl{10.1051/m2an/2013137}
\end{barticle}
\endbibitem

\bibitem[\protect\citeauthoryear{Egger et~al.}{2015}]{EKSW15}
\begin{barticle}
\bauthor{\bsnm{Egger}, \binits{H.}},
\bauthor{\bsnm{Kretzschmar}, \binits{F.}},
\bauthor{\bsnm{Schnepp}, \binits{S.M.}},
\bauthor{\bsnm{Weiland}, \binits{T.}}:
\batitle{A space-time discontinuous {Galerkin} {Trefftz} method for time dependent {Maxwell}'s equations}.
\bjtitle{SIAM J. Sci. Comput.}
\bvolume{37}(\bissue{5}),
\bfpage{689}--\blpage{711}
(\byear{2015})
\doiurl{10.1137/140999323}
\end{barticle}
\endbibitem

\bibitem[\protect\citeauthoryear{Huttunen et~al.}{2007}]{Huttunen}
\begin{barticle}
\bauthor{\bsnm{Huttunen}, \binits{T.}},
\bauthor{\bsnm{Malinen}, \binits{M.}},
\bauthor{\bsnm{Monk}, \binits{P.}}:
\batitle{Solving {Maxwell}’s equations using the ultra weak variational formulation}.
\bjtitle{Journal of Computational Physics}
\bvolume{223}(\bissue{2}),
\bfpage{731}--\blpage{758}
(\byear{2007})
\doiurl{10.1016/j.jcp.2006.10.016}
\end{barticle}
\endbibitem

\bibitem[\protect\citeauthoryear{Gómez et~al.}{2023}]{2306.09571}
\begin{barticle}
\bauthor{\bsnm{Gómez}, \binits{S.}},
\bauthor{\bsnm{Moiola}, \binits{A.}},
\bauthor{\bsnm{Perugia}, \binits{I.}},
\bauthor{\bsnm{Stocker}, \binits{P.}}:
\batitle{{On polynomial Trefftz spaces for the linear time-dependent Schr\"odinger equation}}.
\bjtitle{arXiv preprint arxiv:2306.09571}
(\byear{2023})
\doiurl{10.48550/arXiv.2306.09571}
\end{barticle}
\endbibitem

\bibitem[\protect\citeauthoryear{G\'{o}mez and Moiola}{2022}]{21M1426079}
\begin{barticle}
\bauthor{\bsnm{G\'{o}mez}, \binits{S.}},
\bauthor{\bsnm{Moiola}, \binits{A.}}:
\batitle{A space-time {Trefftz} discontinuous {Galerkin} method for the linear {Schr\"odinger} equation}.
\bjtitle{SIAM Journal on Numerical Analysis}
\bvolume{60}(\bissue{2}),
\bfpage{688}--\blpage{714}
(\byear{2022})
\doiurl{10.1137/21M1426079}
\end{barticle}
\endbibitem

\bibitem[\protect\citeauthoryear{Moiola and Perugia}{2018}]{mope18}
\begin{barticle}
\bauthor{\bsnm{Moiola}, \binits{A.}},
\bauthor{\bsnm{Perugia}, \binits{I.}}:
\batitle{A space--time {Trefftz} discontinuous {Galerkin} method for the acoustic wave equation in first-order formulation}.
\bjtitle{Numer. Math.}
\bvolume{138}(\bissue{2}),
\bfpage{389}--\blpage{435}
(\byear{2018})
\doiurl{10.1007/s00211-017-0910-x}
\end{barticle}
\endbibitem

\bibitem[\protect\citeauthoryear{Banjai et~al.}{2017}]{bgl2016}
\begin{barticle}
\bauthor{\bsnm{Banjai}, \binits{L.}},
\bauthor{\bsnm{Georgoulis}, \binits{E.H.}},
\bauthor{\bsnm{Lijoka}, \binits{O.}}:
\batitle{A {Trefftz} polynomial space-time discontinuous {Galerkin} method for the second order wave equation}.
\bjtitle{SIAM J. Numer. Anal.}
\bvolume{55}(\bissue{1}),
\bfpage{63}--\blpage{86}
(\byear{2017})
\doiurl{10.1137/16M1065744}
\end{barticle}
\endbibitem

\bibitem[\protect\citeauthoryear{Barucq et~al.}{2020}]{bcds20}
\begin{barticle}
\bauthor{\bsnm{Barucq}, \binits{H.}},
\bauthor{\bsnm{Calandra}, \binits{H.}},
\bauthor{\bsnm{Diaz}, \binits{J.}},
\bauthor{\bsnm{Shishenina}, \binits{E.}}:
\batitle{Space--time {Trefftz}-{DG} approximation for elasto-acoustics}.
\bjtitle{Appl. Anal.}
\bvolume{99}(\bissue{5}),
\bfpage{747}--\blpage{760}
(\byear{2020})
\doiurl{10.1080/00036811.2018.1510489}
\end{barticle}
\endbibitem

\bibitem[\protect\citeauthoryear{Kretzschmar et~al.}{2016}]{SpaceTimeTDG}
\begin{barticle}
\bauthor{\bsnm{Kretzschmar}, \binits{F.}},
\bauthor{\bsnm{Moiola}, \binits{A.}},
\bauthor{\bsnm{Perugia}, \binits{I.}},
\bauthor{\bsnm{Schnepp}, \binits{S.M.}}:
\batitle{{A priori error analysis of space-time Trefftz discontinuous Galerkin methods for wave problems}}.
\bjtitle{IMA J. Numer. Anal.}
\bvolume{36}(\bissue{4}),
\bfpage{1599}--\blpage{1635}
(\byear{2016})
\doiurl{10.1093/imanum/drv064}
\end{barticle}
\endbibitem

\bibitem[\protect\citeauthoryear{Kretzschmar et~al.}{2014}]{KSTW2014}
\begin{barticle}
\bauthor{\bsnm{Kretzschmar}, \binits{F.}},
\bauthor{\bsnm{Schnepp}, \binits{S.M.}},
\bauthor{\bsnm{Tsukerman}, \binits{I.}},
\bauthor{\bsnm{Weiland}, \binits{T.}}:
\batitle{{Discontinuous Galerkin methods with Trefftz approximations}}.
\bjtitle{J. Comput. Appl. Math.}
\bvolume{270},
\bfpage{211}--\blpage{222}
(\byear{2014})
\end{barticle}
\endbibitem

\bibitem[\protect\citeauthoryear{Perugia et~al.}{2020}]{tpwave}
\begin{barticle}
\bauthor{\bsnm{Perugia}, \binits{I.}},
\bauthor{\bsnm{Sch\"{o}berl}, \binits{J.}},
\bauthor{\bsnm{Stocker}, \binits{P.}},
\bauthor{\bsnm{Wintersteiger}, \binits{C.}}:
\batitle{Tent pitching and {Trefftz}-{DG} method for the acoustic wave equation}.
\bjtitle{Comput. Math. Appl.}
\bvolume{79}(\bissue{10}),
\bfpage{2987}--\blpage{3000}
(\byear{2020})
\doiurl{10.1016/j.camwa.2020.01.006}
\end{barticle}
\endbibitem

\bibitem[\protect\citeauthoryear{Hiptmair et~al.}{2016}]{TrefftzSurvey}
\begin{bchapter}
\bauthor{\bsnm{Hiptmair}, \binits{R.}},
\bauthor{\bsnm{Moiola}, \binits{A.}},
\bauthor{\bsnm{Perugia}, \binits{I.}}:
\bctitle{A survey of {Trefftz} methods for the {H}elmholtz equation}.
In: \bbtitle{Building Bridges: Connections and Challenges in Modern Approaches to Numerical PDEs}.
\bsertitle{Lect. Notes Comput. Sci. Eng.},
pp. \bfpage{237}--\blpage{278}.
\bpublisher{Springer},
\blocation{Cham.}
(\byear{2016}).
\doiurl{10.1007/978-3-319-41640-3_8}
\end{bchapter}
\endbibitem

\bibitem[\protect\citeauthoryear{Lehrenfeld and Stocker}{2023}]{2201.07041}
\begin{barticle}
\bauthor{\bsnm{Lehrenfeld}, \binits{C.}},
\bauthor{\bsnm{Stocker}, \binits{P.}}:
\batitle{{E}mbedded {T}refftz discontinuous {G}alerkin methods}.
\bjtitle{International Journal for Numerical Methods in Engineering}
(\byear{2023})
\doiurl{10.1002/nme.7258}
\end{barticle}
\endbibitem

\bibitem[\protect\citeauthoryear{Poitou et~al.}{2000}]{POITOU2000561}
\begin{barticle}
\bauthor{\bsnm{Poitou}, \binits{A.}},
\bauthor{\bsnm{Bouberbachene}, \binits{M.}},
\bauthor{\bsnm{Hochard}, \binits{C.}}:
\batitle{{Resolution of three-dimensional Stokes fluid flows using a Trefftz method}}.
\bjtitle{Computer Methods in Applied Mechanics and Engineering}
\bvolume{190}(\bissue{5}),
\bfpage{561}--\blpage{578}
(\byear{2000})
\doiurl{10.1016/S0045-7825(99)00427-2}
\end{barticle}
\endbibitem

\bibitem[\protect\citeauthoryear{Bouberbachene et~al.}{1997}]{Bouberbachene}
\begin{botherref}
\oauthor{\bsnm{Bouberbachene}, \binits{M.}},
\oauthor{\bsnm{Hochard}, \binits{C.}},
\oauthor{\bsnm{Poitou}, \binits{A.}}:
Domain optimisation using {Trefftz} functions - application to free boundaries.
Computer Assisted Mechanics and Engineering Sciences
\textbf{4}
(1997)
\end{botherref}
\endbibitem

\bibitem[\protect\citeauthoryear{Lifits et~al.}{1997}]{LifitsQTSM}
\begin{barticle}
\bauthor{\bsnm{Lifits}, \binits{S.A.}},
\bauthor{\bsnm{Reutskiy}, \binits{S.Y.}},
\bauthor{\bsnm{Pontrelli}, \binits{G.}},
\bauthor{\bsnm{Tirozzi}, \binits{B.}}:
\batitle{Quasi {Trefftz} spectral method for {Stokes} problem}.
\bjtitle{Mathematical Models and Methods in Applied Sciences}
\bvolume{07}(\bissue{08}),
\bfpage{1187}--\blpage{1212}
(\byear{1997})
\doiurl{10.1142/S021820259700058X}
\end{barticle}
\endbibitem

\bibitem[\protect\citeauthoryear{Li et~al.}{2013}]{21710}
\begin{barticle}
\bauthor{\bsnm{Li}, \binits{Z.-C.}},
\bauthor{\bsnm{Lee}, \binits{M.-G.}},
\bauthor{\bsnm{Chiang}, \binits{J.Y.}}:
\batitle{Collocation {Trefftz} methods for the {Stokes} equations with singularity}.
\bjtitle{Numerical Methods for Partial Differential Equations}
\bvolume{29}(\bissue{2}),
\bfpage{361}--\blpage{395}
(\byear{2013})
\doiurl{10.1002/num.21710}
\end{barticle}
\endbibitem

\bibitem[\protect\citeauthoryear{Cockburn et~al.}{2009}]{cockburn2009unified}
\begin{barticle}
\bauthor{\bsnm{Cockburn}, \binits{B.}},
\bauthor{\bsnm{Gopalakrishnan}, \binits{J.}},
\bauthor{\bsnm{Lazarov}, \binits{R.}}:
\batitle{Unified hybridization of discontinuous {Galerkin}, mixed, and continuous {Galerkin} methods for second order elliptic problems}.
\bjtitle{{SIAM J. Numer. Anal.}}
\bvolume{47}(\bissue{2}),
\bfpage{1319}--\blpage{1365}
(\byear{2009})
\end{barticle}
\endbibitem

\bibitem[\protect\citeauthoryear{Farhat et~al.}{2001}]{farhat2001}
\begin{barticle}
\bauthor{\bsnm{Farhat}, \binits{C.}},
\bauthor{\bsnm{Harari}, \binits{I.}},
\bauthor{\bsnm{Franca}, \binits{L.P.}}:
\batitle{The discontinuous enrichment method}.
\bjtitle{Computer Methods in Applied Mechanics and Engineering}
\bvolume{190}(\bissue{48}),
\bfpage{6455}--\blpage{6479}
(\byear{2001})
\doiurl{10.1016/S0045-7825(01)00232-8}
\end{barticle}
\endbibitem

\bibitem[\protect\citeauthoryear{Montlaur et~al.}{2010}]{montlaur2010discontinuous}
\begin{barticle}
\bauthor{\bsnm{Montlaur}, \binits{A.}},
\bauthor{\bsnm{Fernandez-Mendez}, \binits{S.}},
\bauthor{\bsnm{Peraire}, \binits{J.}},
\bauthor{\bsnm{Huerta}, \binits{A.}}:
\batitle{{Discontinuous Galerkin methods for the Navier--Stokes equations using solenoidal approximations}}.
\bjtitle{International Journal for Numerical Methods in Fluids}
\bvolume{64}(\bissue{5}),
\bfpage{549}--\blpage{564}
(\byear{2010})
\end{barticle}
\endbibitem

\bibitem[\protect\citeauthoryear{Montlaur et~al.}{2008}]{montlaur08}
\begin{barticle}
\bauthor{\bsnm{Montlaur}, \binits{A.}},
\bauthor{\bsnm{Fernandez-Mendez}, \binits{S.}},
\bauthor{\bsnm{Huerta}, \binits{A.}}:
\batitle{{Discontinuous Galerkin methods for the Stokes equations using divergence-free approximations}}.
\bjtitle{International Journal for Numerical Methods in Fluids}
\bvolume{57}(\bissue{9}),
\bfpage{1071}--\blpage{1092}
(\byear{2008})
\doiurl{10.1002/fld.1716}
\end{barticle}
\endbibitem

\bibitem[\protect\citeauthoryear{Di~Pietro and Ern}{2011}]{di2011mathematical}
\begin{bbook}
\bauthor{\bsnm{Di~Pietro}, \binits{D.A.}},
\bauthor{\bsnm{Ern}, \binits{A.}}:
\bbtitle{Mathematical Aspects of Discontinuous Galerkin Methods}
vol. \bseriesno{69}.
\bpublisher{Springer},
\blocation{Heidelberg}
(\byear{2011}).
\doiurl{10.1007/978-3-642-22980-0}
\end{bbook}
\endbibitem

\bibitem[\protect\citeauthoryear{Baker et~al.}{1990}]{BJK90}
\begin{barticle}
\bauthor{\bsnm{Baker}, \binits{G.A.}},
\bauthor{\bsnm{Jureidini}, \binits{W.N.}},
\bauthor{\bsnm{Karakashian}, \binits{O.A.}}:
\batitle{Piecewise solenoidal vector fields and the {Stokes} problem}.
\bjtitle{SIAM J. Numer. Anal.}
\bvolume{27}(\bissue{6}),
\bfpage{1466}--\blpage{1485}
(\byear{1990})
\doiurl{10.1137/0727085}
\end{barticle}
\endbibitem

\bibitem[\protect\citeauthoryear{John}{2016}]{zbMATH06625569}
\begin{bbook}
\bauthor{\bsnm{John}, \binits{V.}}:
\bbtitle{Finite Element Methods for Incompressible Flow Problems}.
\bsertitle{Springer Ser. Comput. Math.},
vol. \bseriesno{51}.
\bpublisher{Springer},
\blocation{Cham.}
(\byear{2016}).
\doiurl{10.1007/978-3-319-45750-5}
\end{bbook}
\endbibitem

\bibitem[\protect\citeauthoryear{Brenner}{2003}]{brenner2003poincare}
\begin{barticle}
\bauthor{\bsnm{Brenner}, \binits{S.C.}}:
\batitle{Poincaré--friedrichs inequalities for piecewise h1 functions}.
\bjtitle{SIAM Journal on Numerical Analysis}
\bvolume{41}(\bissue{1}),
\bfpage{306}--\blpage{324}
(\byear{2003})
\doiurl{10.1137/S0036142902401311}
{\href{https://arxiv.org/abs/https://doi.org/10.1137/S0036142902401311}{{https://doi.org/10.1137/S0036142902401311}}}
\end{barticle}
\endbibitem

\bibitem[\protect\citeauthoryear{Arnold et~al.}{2002}]{arnold2002unified}
\begin{barticle}
\bauthor{\bsnm{Arnold}, \binits{D.N.}},
\bauthor{\bsnm{Brezzi}, \binits{F.}},
\bauthor{\bsnm{Cockburn}, \binits{B.}},
\bauthor{\bsnm{Marini}, \binits{L.D.}}:
\batitle{Unified analysis of discontinuous {Galerkin} methods for elliptic problems}.
\bjtitle{SIAM Journal on Numerical Analysis}
\bvolume{39}(\bissue{5}),
\bfpage{1749}--\blpage{1779}
(\byear{2002})
\doiurl{10.1137/S0036142901384162}
\end{barticle}
\endbibitem

\bibitem[\protect\citeauthoryear{Boffi et~al.}{2013}]{brezzi}
\begin{bbook}
\bauthor{\bsnm{Boffi}, \binits{D.}},
\bauthor{\bsnm{Brezzi}, \binits{F.}},
\bauthor{\bsnm{Fortin}, \binits{M.}}:
\bbtitle{{Mixed Finite Element Methods and Applications}}.
\bpublisher{Springer},
\blocation{Heidelberg}
(\byear{2013}).
\doiurl{10.1007/978-3-642-36519-5}
\end{bbook}
\endbibitem

\bibitem[\protect\citeauthoryear{Cockburn et~al.}{2005}]{cockburn2005locally}
\begin{barticle}
\bauthor{\bsnm{Cockburn}, \binits{B.}},
\bauthor{\bsnm{Kanschat}, \binits{G.}},
\bauthor{\bsnm{Sch{\"o}tzau}, \binits{D.}}:
\batitle{{A locally conservative LDG method for the incompressible Navier-Stokes equations}}.
\bjtitle{Mathematics of computation}
\bvolume{74}(\bissue{251}),
\bfpage{1067}--\blpage{1095}
(\byear{2005})
\end{barticle}
\endbibitem

\bibitem[\protect\citeauthoryear{Cockburn et~al.}{2007}]{cockburn2007note}
\begin{barticle}
\bauthor{\bsnm{Cockburn}, \binits{B.}},
\bauthor{\bsnm{Kanschat}, \binits{G.}},
\bauthor{\bsnm{Sch{\"o}tzau}, \binits{D.}}:
\batitle{A note on discontinuous {Galerkin} divergence-free solutions of the {Navier}--{Stokes} equations}.
\bjtitle{Journal of Scientific Computing}
\bvolume{31}(\bissue{1-2}),
\bfpage{61}--\blpage{73}
(\byear{2007})
\end{barticle}
\endbibitem

\bibitem[\protect\citeauthoryear{{Alemán}}{2022}]{alemanmaster}
\begin{botherref}
\oauthor{\bsnm{{Alemán}}, \binits{T.}}:
{Robust Finite Element Discretizations for a PDE arising in Helioseismology}.
Master's thesis,
NAM, University of G\"ottingen
(March 2022).
\doiurl{10.25625/1GBYXP}
\end{botherref}
\endbibitem

\bibitem[\protect\citeauthoryear{Hansbo and Larson}{2002}]{hansbolarson02}
\begin{barticle}
\bauthor{\bsnm{Hansbo}, \binits{P.}},
\bauthor{\bsnm{Larson}, \binits{M.G.}}:
\batitle{Discontinuous {Galerkin} methods for incompressible and nearly incompressible elasticity by {N}itsche's method}.
\bjtitle{Comput. Methods Appl. Mech. Eng.}
\bvolume{191},
\bfpage{1895}--\blpage{1908}
(\byear{2002})
\end{barticle}
\endbibitem

\bibitem[\protect\citeauthoryear{Howell and Walkington}{2011}]{saddle}
\begin{barticle}
\bauthor{\bsnm{Howell}, \binits{J.}},
\bauthor{\bsnm{Walkington}, \binits{N.}}:
\batitle{Inf–sup conditions for twofold saddle point problems}.
\bjtitle{Numerische Mathematik}
\bvolume{118},
\bfpage{663}--\blpage{693}
(\byear{2011})
\doiurl{10.1007/s00211-011-0372-5}
\end{barticle}
\endbibitem

\bibitem[\protect\citeauthoryear{Amrouche and Girault}{1991}]{10.3792/pjaa.67.171}
\begin{barticle}
\bauthor{\bsnm{Amrouche}, \binits{C.}},
\bauthor{\bsnm{Girault}, \binits{V.}}:
\batitle{{On the existence and regularity of the solution of Stokes problem in arbitrary dimension}}.
\bjtitle{Proceedings of the Japan Academy, Series A, Mathematical Sciences}
\bvolume{67}(\bissue{5}),
\bfpage{171}--\blpage{175}
(\byear{1991})
\doiurl{10.3792/pjaa.67.171}
\end{barticle}
\endbibitem

\bibitem[\protect\citeauthoryear{Brenner and Scott}{2008}]{brennerscott}
\begin{bbook}
\bauthor{\bsnm{Brenner}, \binits{S.C.}},
\bauthor{\bsnm{Scott}, \binits{L.R.}}:
\bbtitle{The Mathematical Theory of Finite Element Methods}.
\bsertitle{Texts in Applied Mathematics},
vol. \bseriesno{15}.
\bpublisher{Springer},
\blocation{New York}
(\byear{2008}).
\doiurl{10.1007/978-0-387-75934-0}
\end{bbook}
\endbibitem

\bibitem[\protect\citeauthoryear{Cl\'ement}{1975}]{clement}
\begin{barticle}
\bauthor{\bsnm{Cl\'ement}, \binits{P.}}:
\batitle{Approximation by finite element functions using local regularization}.
\bjtitle{Rev. Franc. Automat. Inform. Rech. Operat., R}
\bvolume{9}(\bissue{2}),
\bfpage{77}--\blpage{84}
(\byear{1975})
\end{barticle}
\endbibitem

\bibitem[\protect\citeauthoryear{Sch{\"o}berl}{2014}]{ngsolve}
\begin{botherref}
\oauthor{\bsnm{Sch{\"o}berl}, \binits{J.}}:
{C++ 11 implementation of finite elements in NGSolve}.
Institute for analysis and scientific computing, Vienna University of Technology
\textbf{30}
(2014)
\end{botherref}
\endbibitem

\bibitem[\protect\citeauthoryear{Stocker}{2022}]{ngstrefftz}
\begin{barticle}
\bauthor{\bsnm{Stocker}, \binits{P.}}:
\batitle{{NGST}refftz: {A}dd-on to {NGS}olve for {T}refftz methods}.
\bjtitle{J. Open Source Softw.}
\bvolume{7}(\bissue{71}),
\bfpage{4135}
(\byear{2022})
\doiurl{10.21105/joss.04135}
\end{barticle}
\endbibitem

\bibitem[\protect\citeauthoryear{Lederer et~al.}{2023}]{UHPEQG_2023}
\begin{barticle}
\bauthor{\bsnm{Lederer}, \binits{P.L.}},
\bauthor{\bsnm{Lehrenfeld}, \binits{C.}},
\bauthor{\bsnm{Stocker}, \binits{P.}}:
\batitle{{Replication Data for: Trefftz Discontinuous Galerkin discretization for the Stokes problem}}.
\bjtitle{GRO.data}
(\byear{2023})
\doiurl{10.25625/UHPEQG}
\end{barticle}
\endbibitem

\bibitem[\protect\citeauthoryear{Moffatt}{1964}]{moffatt_1964}
\begin{barticle}
\bauthor{\bsnm{Moffatt}, \binits{H.K.}}:
\batitle{Viscous and resistive eddies near a sharp corner}.
\bjtitle{Journal of Fluid Mechanics}
\bvolume{18}(\bissue{1}),
\bfpage{1}--\blpage{18}
(\byear{1964})
\doiurl{10.1017/S0022112064000015}
\end{barticle}
\endbibitem

\bibitem[\protect\citeauthoryear{Ainsworth and Parker}{2021}]{ainsworthparker}
\begin{barticle}
\bauthor{\bsnm{Ainsworth}, \binits{M.}},
\bauthor{\bsnm{Parker}, \binits{C.}}:
\batitle{Mass conserving mixed $hp$-{FEM} approximations to {Stokes} flow. {P}art {II}: Optimal convergence}.
\bjtitle{SIAM Journal on Numerical Analysis}
\bvolume{59}(\bissue{3}),
\bfpage{1245}--\blpage{1272}
(\byear{2021})
\doiurl{10.1137/20M1359110}
\end{barticle}
\endbibitem

\bibitem[\protect\citeauthoryear{Cockburn et~al.}{2010}]{zbMATH05837531}
\begin{barticle}
\bauthor{\bsnm{Cockburn}, \binits{B.}},
\bauthor{\bsnm{Nguyen}, \binits{N.C.}},
\bauthor{\bsnm{Peraire}, \binits{J.}}:
\batitle{A comparison of {HDG} methods for {Stokes} flow}.
\bjtitle{J. Sci. Comput.}
\bvolume{45}(\bissue{1-3}),
\bfpage{215}--\blpage{237}
(\byear{2010})
\doiurl{10.1007/s10915-010-9359-0}
\end{barticle}
\endbibitem

\bibitem[\protect\citeauthoryear{Cockburn and Shi}{2014}]{COCKBURN2014221}
\begin{barticle}
\bauthor{\bsnm{Cockburn}, \binits{B.}},
\bauthor{\bsnm{Shi}, \binits{K.}}:
\batitle{{Devising HDG methods for Stokes flow: An overview}}.
\bjtitle{Computers \& Fluids}
\bvolume{98},
\bfpage{221}--\blpage{229}
(\byear{2014})
\doiurl{10.1016/j.compfluid.2013.11.017}
\end{barticle}
\endbibitem

\bibitem[\protect\citeauthoryear{Lehrenfeld}{2010}]{L_MTH_2010}
\begin{botherref}
\oauthor{\bsnm{Lehrenfeld}, \binits{C.}}:
{Hybrid Discontinuous Galerkin Methods for Incompressible Flow Problems}.
Master's thesis,
RWTH Aachen
(May 2010).
\doiurl{10.25625/O4VBYH}
\end{botherref}
\endbibitem

\bibitem[\protect\citeauthoryear{Rhebergen and Wells}{2018}]{RhebergenWells}
\begin{barticle}
\bauthor{\bsnm{Rhebergen}, \binits{S.}},
\bauthor{\bsnm{Wells}, \binits{G.N.}}:
\batitle{{A Hybridizable Discontinuous {Galerkin} Method for the {Navier–Stokes} Equations with Pointwise Divergence-Free Velocity Field}}.
\bjtitle{J. Sci. Comput.}
\bvolume{76},
\bfpage{1484}--\blpage{1501}
(\byear{2018})
\doiurl{10.1007/s10915-018-0671-4}
\end{barticle}
\endbibitem

\bibitem[\protect\citeauthoryear{Zaglmayr}{2006}]{zaglmayr2006high}
\begin{botherref}
\oauthor{\bsnm{Zaglmayr}, \binits{S.}}:
High order finite element methods for electromagnetic field computation.
PhD thesis,
Johannes-Kepler Universität Linz
(2006)
\end{botherref}
\endbibitem

\bibitem[\protect\citeauthoryear{Lehrenfeld and Sch\"{o}berl}{2016}]{LS_CMAME_2016}
\begin{barticle}
\bauthor{\bsnm{Lehrenfeld}, \binits{C.}},
\bauthor{\bsnm{Sch\"{o}berl}, \binits{J.}}:
\batitle{High order exactly divergence-free hybrid discontinuous {Galerkin} methods for unsteady incompressible flows}.
\bjtitle{Computer Methods in Applied Mechanics and Engineering}
\bvolume{307},
\bfpage{339}--\blpage{361}
(\byear{2016})
\doiurl{10.1016/j.cma.2016.04.025}
\end{barticle}
\endbibitem

\bibitem[\protect\citeauthoryear{Lederer et~al.}{2018}]{LLS_SIAM_2017}
\begin{barticle}
\bauthor{\bsnm{Lederer}, \binits{P.L.}},
\bauthor{\bsnm{Lehrenfeld}, \binits{C.}},
\bauthor{\bsnm{Sch{\"o}berl}, \binits{J.}}:
\batitle{Hybrid discontinuous {Galerkin} methods with relaxed {H(div)}-conformity for incompressible flows. {Part I}}.
\bjtitle{SIAM J. Numer. Anal.}
\bvolume{56},
\bfpage{2070}--\blpage{2094}
(\byear{2018})
\doiurl{10.1137/17M1138078}
\end{barticle}
\endbibitem

\bibitem[\protect\citeauthoryear{Lederer et~al.}{2019}]{LLS_ESAIM_2019}
\begin{barticle}
\bauthor{\bsnm{Lederer}, \binits{P.L.}},
\bauthor{\bsnm{Lehrenfeld}, \binits{C.}},
\bauthor{\bsnm{Sch{\"o}berl}, \binits{J.}}:
\batitle{Hybrid discontinuous {Galerkin} methods with relaxed {H(div)}-conformity for incompressible flows. {Part II}}.
\bjtitle{ESAIM: M2AN}
\bvolume{53},
\bfpage{503}--\blpage{522}
(\byear{2019})
\doiurl{10.1051/m2an/2018054}
\end{barticle}
\endbibitem

\bibitem[\protect\citeauthoryear{Axler et~al.}{2001}]{harmonicfcttheory}
\begin{bbook}
\bauthor{\bsnm{Axler}, \binits{S.}},
\bauthor{\bsnm{Bourdon}, \binits{P.}},
\bauthor{\bsnm{Ramey}, \binits{W.}}:
\bbtitle{Harmonic Function Theory.},
\bedition{2nd ed.} edn.
\bsertitle{Grad. Texts Math.},
vol. \bseriesno{137}.
\bpublisher{Springer},
\blocation{New York}
(\byear{2001}).
\doiurl{10.1007/b97238}
\end{bbook}
\endbibitem

\end{thebibliography}

\end{document}